\documentclass{zlondon}

\usepackage{amsthm}
\usepackage{amsmath}
\usepackage{amsfonts}
\usepackage{amssymb}
\usepackage{mathrsfs}  
\usepackage{enumitem}
\usepackage{chngcntr}

\newcommand{\comp}{\between}
\newcommand{\incomp}{\,\bot\,}

\usepackage{stmaryrd}  

\usepackage{tikz}
\usepackage{pgf} 
\usetikzlibrary{patterns,automata,arrows,shapes,snakes,topaths,trees,backgrounds,positioning,through,calc}

\usepackage[all]{xy}

\newtheorem{theorem}{Theorem}[subsection]
\newtheorem{lemma}[theorem]{Lemma}
\newtheorem{proposition}[theorem]{Proposition}
\newtheorem{corollary}[theorem]{Corollary}

\theoremstyle{definition}
\newtheorem{definition}[theorem]{Definition}
\newtheorem{example}[theorem]{Example}
\newtheorem{question}[theorem]{Question}

\theoremstyle{remark}
\newtheorem{remark}[theorem]{Remark}

\newcommand{\inqvee}{\,\rotatebox[origin=c]{-90}{$\geqslant$}\,}
\newcommand{\Inqvee}{\rotatebox[origin=c]{-90}{$\geqslant$}}

\newcommand{\upR}{\textbf{up}-$\boldsymbol{R}$}

\newcommand{\RWin}{$\boldsymbol{R\mathord{\Leftrightarrow}}\textbf{\underline{win}}$}
\newcommand{\Rwinweak}{$\boldsymbol{R\mathord{\Rightarrow}}\textbf{win}$}
\newcommand{\Rrule}{$\boldsymbol{R}$-\textbf{rule}}

\newcommand{\Rdown}{$\boldsymbol{R}$-\textbf{down}}
\newcommand{\Rdense}{$\boldsymbol{R}$-\textbf{dense}}
\newcommand{\Rref}{$\boldsymbol{R}$-\textbf{refinability}}
\newcommand{\RrefPlus}{$\boldsymbol{R}$-\textbf{refinability}$^+$}
\newcommand{\RrefPlusPlus}{$\boldsymbol{R}$-\textbf{refinability}$^{++}$}

\newcommand{\SqForth}{\textit{$\sqsubseteq$-forth}}

\newcommand{\SqBack}{\textit{$\sqsubseteq$-back}}

\newcommand{\RForth}{\textit{$R$-forth}}

\newcommand{\SRBack}{\textit{$R$-back}}

\newcommand{\pRBack}{\textit{$p$-$R$-back}}
\newcommand{\pSqBack}{\textit{$p$-$\sqsubseteq$-back}}

\usepackage{comment}

\usepackage{datetime}

\usepackage{hyperref}

	\jyear{2024}			
	\honour{\today}		
	\nopagenumber
	
\begin{document}


\settimeformat{ampmtime}

\title{Possibility Semantics}
\titlerunning{Possibility Semantics}

\addauthor[wesholliday@berkeley.edu]{Wesley H. Holliday}{University of California, Berkeley}
\authorrunning{Wesley H. Holliday}


\maketitle


\begin{quote} \small Previous version in \textit{Selected Topics from Contemporary Logics}, ed.~Melvin Fitting, Volume 2 of Landscapes in Logic, College Publications, London, 2021, ISBN 97-1-84890-350-0, pp.~363-476. This version corrects \S~4.3.
\end{quote}

\noindent

\begin{quote} 
\noindent\small\textsc{Abstract.} In traditional semantics for classical logic and its extensions, such as modal logic, propositions are interpreted as subsets of a set, as in discrete duality, or as clopen sets of a Stone space, as in topological duality. A point in such a set can be viewed as a ``possible world,'' with the key property of a world being \textit{primeness}---a world makes a disjunction true only if it makes one of the disjuncts true---which classically implies  \textit{totality}---for each proposition, a world either makes the proposition true or makes its negation true. This chapter surveys a more general approach to logical semantics, known as \textit{possibility semantics}, which replaces possible worlds with possibly \textit{partial} ``possibilities.'' In classical possibility semantics, propositions are interpreted as regular open sets of a poset, as in set-theoretic forcing, or as compact regular open sets of an upper Vietoris space, as in the recent theory of ``choice-free Stone duality.'' The elements of these sets, viewed as possibilities, may be partial in the sense of making a disjunction true without settling which disjunct is true. We explain how possibilities may be used in semantics for classical logic and modal logics and generalized to semantics for intuitionistic logics. The goals are to overcome or deepen incompleteness results for traditional semantics, to avoid the nonconstructivity of traditional semantics, and to provide richer structures for the interpretation of new languages. \end{quote} 

\noindent

\begin{quote} 
\end{quote} 

\noindent
\begin{quote} 
{\bf AMS classification (2010):} 
03B10, 03B20, 03B45, 03B55, 06D20, 06D22, 03F45, 03G05, 06E25
\end{quote} 



\newpage

\tableofcontents

\section{Introduction}\label{IntroSection}

Traditional semantics for classical logic and its extensions, such as modal logic, can be seen as based on two fundamental relationships:
\begin{enumerate}
\item Tarski's \cite{Tarski1935} discrete duality between complete and atomic Boolean algebras (CABAs) and sets;
\item Stone's \cite{Stone1936} topological duality between Boolean algebras (BAs) and Stone spaces.
\end{enumerate}
In the traditional approach, propositions are interpreted as subsets of a set, as in 1, or as clopen sets of a Stone space, as in 2. Negation, conjunction, and disjunction of propositions are interpreted as set-theoretic complementation, intersection, and union, respectively. As a consequence, the points in such sets can be viewed as ``possible worlds,'' with the key property of a possible world being:
\begin{itemize}
\item \textit{primeness}: a world makes a disjunction true only if it makes one of the disjuncts true.
\end{itemize}
This follows from the interpretation of disjunction as union. Given the interpretation of negation as complementation, classical worlds also have the property~of:
\begin{itemize}
\item \textit{totality}: for each proposition, a world either makes the proposition true or makes its negation true.
\end{itemize}
If we start with an algebra of propositions, then such worlds can be recovered as \textit{atoms}---if the algebra of propositions is a CABA---or \textit{prime filters}---if the algebra of propositions is an arbitrary BA and we avail ourselves of the Axiom of Choice.\footnote{The Prime Filter Theorem needed for Stone duality is weaker than the Axiom of Choice \cite{Halpern1971}, but still beyond ZF set theory \cite{Feferman1964} (as well as ZF with Dependent Choice \cite{Pincus1977}).}

In this chapter, we survey a more general approach to logical semantics, known as \textit{possibility semantics}, which replaces possible worlds with possibly \textit{partial} ``possibilities.'' The classical version of possibility semantics is based on the following two relationships:
\begin{enumerate}
\item[$1'$.] the representation of complete Boolean algebras (CBAs) using posets, exploited in set-theoretic forcing \cite{Takeuti1973};
\item[$2'$.] the duality between BAs and  upper Vietoris spaces (UV-spaces), developed in the recent theory of ``choice-free Stone duality'' \cite{BH2020}.
\end{enumerate}
In this approach, propositions are interpreted as \textit{regular open sets} of a poset, as in $1'$, or as \textit{compact regular open sets} of a UV-space, as in $2'$ (all of these notions will be defined in what follows). The partial order $\sqsubseteq$ in the poset of possibilities represents a relationship of \textit{refinement} between partial possibilities: roughly speaking, $x\sqsubseteq y$ means that $x$ contains all the information that $y$ does and possibly more. Instead of interpreting negation and disjunction as set-theoretic complementation and union, they are interpreted using the refinement relation $\sqsubseteq$:  
\begin{itemize}
\item a possibility $x$ makes the negation of a proposition true just in case no refinement $x'\sqsubseteq x$ makes the proposition true;
\item  a possibility $x$ makes the disjunction of two propositions true just in case for every refinement $x'\sqsubseteq x$ there is a further refinement $x''\sqsubseteq x'$ that makes one of the disjuncts true.
\end{itemize} 
As a result, a possibility may be neither total nor prime;  e.g., a possibility may settle the proposition ``$p$ or not $p$'' as true and yet not settle which disjunct is true, leaving this to be settled by refinements of the possibility.

By endowing posets of possibilities with additional structure, one can give possibility semantics for first-order logic, as done by Fine \cite{Fine1975} and van Benthem \cite{Benthem1981a}, and modal logics, as done by Humberstone \cite{Humberstone1981}, who coined the term `possibility semantics'. As we will see, for modal logic the move from worlds to possibilities allows us to overcome some incompleteness results for traditional~semantics. 

There is also a natural generalization of possibility semantics for intuitionistic logic. Possible world semantics for intuitionistic logic \cite{Kripke1965} is based on:
\begin{itemize}
\item[$3.$] the representation of certain complete Heyting algebras---namely those that are generated by completely join-prime elements---using posets, generalizing Birkhoff's \cite{Birkhoff1937} representation of finite distributive lattices.\footnote{In the case of topological duality, there is Esakia's \cite{Esakia85,Esakia2019} duality between Heyting algebras and Esakia spaces. For a possibility-based topological duality for Heyting algebras, several ideas have been proposed, but the advantages of different proposals are not yet clear.}
\end{itemize}
In this traditional approach, propositions are interpreted as downward-closed subsets (downsets) of the poset.\footnote{It is more common in the intuitionistic literature to think in terms of upward-closed subsets (upsets) (for an exception, see \cite{Dragalin1988}), but in line with the predominant practice in set-theoretic forcing, we think in terms of downsets in this chapter.} Intuitionistic implication and negation are interpreted using the relation $\sqsubseteq$ of the poset just as in possibility semantics. As a consequence, the elements of the poset need not have the property of \textit{totality}. However, disjunction is still interpreted as union, so the elements still have the \textit{primeness} property of worlds, explaining our use of the term ``possible world semantics for intuitionistic logic.'' By contrast, possibility semantics for intuitionistic logic is based on the following (for the definition of a nucleus, see \S~\ref{IntuitionisticSection}):
\begin{itemize}
\item[$3'$.] the representation of arbitrary complete Heyting algebras using posets \\ equipped with a nucleus on the algebra of downsets, in the manner of \\ Dragalin \cite{Dragalin1988}.
\end{itemize}
The nucleus can be realized concretely in various ways (the most concrete way uses a distinguished subrelation of the partial order), as we will review later. In this approach, propositions are interpreted as fixpoints of the nucleus. The disjunction of two propositions is interpreted by applying the nucleus to the union of the interpretations of the disjuncts. Since the result may strictly extend the union, a possibility may settle a disjunction as true without settling which disjunct is true, so the  primeness property need not hold. Possibility semantics for classical logic is the special case of this approach using the nucleus of double negation (see~\S~\ref{IntuitionisticSection}).

Possibility semantics for classical logic may also be seen as a special case of possibility semantics for \textit{orthologic}, the propositional logic of ortholattices. This is based on the representation of ortholattices using symmetric and reflexive compatibility relations (or symmetric and irreflexive orthogonality relations) on a set \cite{Goldblatt1974b,Goldblatt1975a,Holliday2021,McDonald2021}. Propositions are fixpoints of a closure operator defined using the compatibility relation. The disjunction of two propositions is interpreted by applying the closure operator to the union of the interpretations of the disjuncts. Classical possibility semantics is a special case where additional properties are assumed to hold of compatibility \cite[Example~3.16]{Holliday2021}. Rather than covering possibility semantics for orthologic here, we refer the reader to \cite{HM2021} for an introduction with applications to epistemic modals and conditionals.

There are at least three benefits in the move from worlds to possibilities:
\begin{enumerate}
\item In the setting of discrete duality, by lifting the restriction to \textit{atomic} CBAs, possibility semantics allows us to characterize logics that cannot be characterized using possible world semantics based on CABAs. Examples are given in \S\S~\ref{PropQuantSection}, \ref{BasicNeighSection}, \ref{FullRelFrameSection}, and \ref{QuantModalSection}.
\item In the setting of topological duality, while world-based topological dualities require nonconstructive choice principles, possibility-based topological dualities can be carried out in the setting of \textit{quasi-constructive mathematics}, defined as  ``mathematics that permits conventional rules of reasoning plus ZF + DC, but no stronger forms of Choice'' \cite[\S~14.76]{Schechter1996}.\footnote{In fact, ZF suffices for the dualities, though Dependent Choice may be useful in applications.}
\item Even when atomicity or nonconstructivity are not concerns (e.g., because finite algebras are sufficient for one's purposes), the extra refinement structure in possibility semantics makes possible the interpretation of \textit{extended languages} that cannot be handled in a purely world-based semantics. An example is given in \S~\ref{InquisitiveSection} with the language of \textit{inquisitive logic}.
\end{enumerate}

The rest of the chapter is organized as follows. In \S~\ref{PhilSection}, our philosophical preamble, we derive the essential elements of (classical) possibility semantics from a few basic axioms about truth and falsity. Successive subsections then explain possibility semantics for different algebras/logics: Boolean algebra/logic (\S~\ref{BooleanSection}), first-order logic (\S~\ref{FOSection}), modal algebra/logic (\S~\ref{ModalSection}), Heyting algebra/intuitionistic logic (\S~\ref{IntuitionisticSection}), and inquisitive logic (\S~\ref{InquisitiveSection}). \S\S~\ref{Intervals}-\ref{Bimodal} provide perspectives on possibility semantics from the point of view of interval semantics and traditional possible world semantics. In addition to surveying the basic theory and existing results for possibility semantics, we highlight a number of open problems. 

\subsection*{How to read this chapter}
This chapter is designed for the reader to dip in and out of sections based on the reader's interest.  The main prerequisites for later sections are \S\S~\ref{PosetsOfPoss}, \ref{CompleteBooleanSection}, and \ref{ArbitraryBooleanSection1}. Readers mainly interested in possibility semantics for propositional modal logic can skip \S~\ref{FOSection}; readers mainly interesting in possibility semantics for intuitionistic propositional logic can skip  \S\S~\ref{FOSection}-\ref{ModalSection}, etc.  Results from previous sections are cited explicitly rather than assumed. Theorems are numbered by subsection. Thus, e.g., Theorem 3.1.1 is the first theorem in \S~3.1. If Theorem 3.1.1 has multiple parts, we refer to Theorem 3.1.1.1 for the first part,  3.1.1.2 for the second,~etc.

\section{Philosophical explanation}\label{PhilSection}

In this section, we explain the core ideas of classical possibility semantics philosophically. This section is not meant to give set-theoretic mathematical definitions and results, so we do not use definition or theorem environments. But we do give ``proofs'' from axioms using classical logic, which could be fully formalized.

For further philosophical motivation, see \cite{Humberstone1981,Humberstone2015},  \cite[\S~6.44]{Humberstone2011}, and \cite{Garson2013,Hale2013,Rumfitt2015}. A comparison of possibility semantics to other semantics based on partial states, such as \textit{situation semantics} or \textit{truthmaker semantics}, is beyond the scope of this chapter; we recommend the surveys in \cite{Kratzer2019} and \cite{Fine2017} to interested readers.

\subsection{Partiality}

Unlike a possible world, a possibility may be \textit{partial}, in two related respects:

\begin{enumerate}
\item It may fail to settle the truth or falsity of a proposition (non-totality).
\item It may settle the truth or falsity of a proposition without settling exactly how the proposition is made true or false (non-primeness); e.g., it may settle that \textit{the particle is spin up or spin down} but not settle that \textit{the particle is spin up} and not settle that \textit{the particle is spin down}.\footnote{Hale \cite{Hale2013} calls these types of partiality \textit{global incompleteness} and \textit{local incompleteness}.}\end{enumerate}

\subsection{Refinement}\label{RefinementSection}

Some possibilities are less partial than others, inducing a relation of refinement: 

\begin{itemize}
\item[(A0)]A possibility $x$ is a \textbf{refinement} of a possibility $y$ (notation: $x\sqsubseteq y$) iff:
\begin{itemize}
\item[$\bullet$] $x$ settles as true every proposition that $y$ settles as true, and 
\item[$\bullet$] $x$ settles as false every proposition that $y$ settles as false.
\end{itemize}
\end{itemize}
The relation $\sqsubseteq$ is therefore \textit{reflexive} and \textit{transitive}. If we do not distinguish possibilities that settle as true/false exactly the same propositions, then we may also assume \textit{antisymmetry} for $\sqsubseteq$. Then $\sqsubseteq$ is a partial order.

\subsection{Truth and falsity}\label{TruthFalsity}

In possibility semantics, the notions of a possibility $x$ settling a proposition $\mathbf{P}$ as true or as false are related as stated in the following axioms:
\begin{itemize}
\item[(A1)] $x$ settles $\mathbf{P}$ as true iff no possibility refining $x$ settles $\mathbf{P}$ as false;
\item[(A2)] $x$ settles $\mathbf{P}$ as false iff no possibility refining $x$ settles $\mathbf{P}$ as true.
\end{itemize}

\noindent If one thinks of settling true as a kind of \textit{necessitation}, then one can see (A1) and (A2) as related to Aristotle's idea of the duality of necessity and possibility \cite{Malink2016}.

Using (A1) and (A2), we may argue that the $\sqsubseteq$ relation is \textit{separative}: 
\begin{itemize}
\item if $x$ is not a refinement of $y$ ($x\not\sqsubseteq y$), then there is a refinement $x'$ of $x$ that is incompatible with $y$, in the sense that $x'$ and $y$ have no common refinement. 
\end{itemize}

\begin{proof} If $x$ is not a refinement of $y$, then by (A0), there is a proposition $\mathbf{P}$ that (i) $y$ settles as true but $x$ does not or (ii) $y$ settles as false but $x$ does not. In case (i), it follows by (A1) that there is a refinement $x'$ of $x$ that settles $\mathbf{P}$ as false, which by (A2) implies that no refinement $x''$ of $x'$ settles $\mathbf{P}$ as true. It follows that no refinement $x''$ of $x'$ is a refinement of $y$, because  $y$ settles $\mathbf{P}$ as true and hence  any refinement of $y$ settles $\mathbf{P}$ as true by (A0). Thus, $x'$ is incompatible with $y$. In case (ii), the argument is analogous, flipping the roles of (A1) and (A2) and of truth and falsity.
\end{proof}

\subsection{Regular sets}\label{RegularSets}

From (A1) and (A2), we have as a consequence:
\begin{itemize}
\item[(C1)] $x$ settles $\mathbf{P}$ as true iff for every possibility $x'$ refining $x$, there is a possibility $x''$ refining $x'$ that settles $\mathbf{P}$ as true.
\end{itemize}
\begin{proof}From left to right, suppose $x$ settles $\mathbf{P}$ as true and $x'$ is a possibility refining $x$. Then by (A1), $x'$ does not settle $\mathbf{P}$ as false. Thus, by (A2), there is a possibility $x''$ refining $x'$ that settles $\mathbf{P}$ as true. From right to left, we prove the contrapositive. If $x$ does not settle $\mathbf{P}$ as true, then by (A1), there is a possibility $x'$ refining $x$ that settles $\mathbf{P}$ as false. Hence, by (A2), no possibility $x''$ refining $x'$ settles $\mathbf{P}$ as true.
\end{proof}

We say that a set $X$ of possibilities is \textbf{regular} if:
\begin{itemize}
\item $x\in X$ iff for every possibility $x'$ refining $x$, there is a possibility $x''$ refining $x'$ such that $x''\in X$.
\end{itemize}
We say that a set $X$ of possibilities \textbf{corresponds to} a proposition $\mathbf{P}$ if $X$ contains exactly the possibilities that settle $\mathbf{P}$ as true. Thus, as an immediate consequence of (C1), we have:
\begin{itemize}
\item[(C2)] If a set $X$ of possibilities corresponds to a proposition, then $X$ is regular. 
\end{itemize}

\subsection{Connectives}\label{ConnectivesSection}

Three final axioms relate truth and falsity for propositions formed using logical connectives:
\begin{itemize}
\item[(A3)] $x$ settles the negation of $\mathbf{P}$ as \textbf{true} iff $x$ settles $\mathbf{P}$ as \textbf{false};
\item[(A4)] $x$ settles the conjunction of $\mathbf{P}$ and $\mathbf{Q}$ as \textbf{true} iff $x$ settles $\mathbf{P}$ as \textbf{true} and $x$ settles $\mathbf{Q}$ as \textbf{true};
\item[(A5)] $x$ settles the disjunction of $\mathbf{P}$ and $\mathbf{Q}$ as \textbf{false} iff $x$ settles $\mathbf{P}$ as \textbf{false} and $x$ settles $\mathbf{Q}$ as \textbf{false}.
\end{itemize}

As a consequence of (A3) and (A0), assuming that every proposition has a negation, we can characterize refinement purely in terms of truth:
\begin{itemize}
\item[(C3)] A possibility $x$ is a refinement of a possibility $y$ iff $x$ settles as true every proposition that $y$ settles as true.
\end{itemize}

As a consequence of (A3) and (A2), we can give truth conditions for negation purely in terms of truth:

\begin{itemize}
\item[(C4)] $x$ settles the negation of $\mathbf{P}$ as \textbf{true} iff no possibility refining $x$ settles $\mathbf{P}$ as \textbf{true}.
\end{itemize}

Finally, as a consequence of (A5), (A1), and (A2), we can give truth conditions for disjunction purely in terms of truth:
\begin{itemize}
\item[(C5)] $x$ settles the disjunction of $\mathbf{P}$ and $\mathbf{Q}$ as \textbf{true} iff for every possibility $x'$ refining $x$, there is a possibility $x''$ refining $x'$ such that $x''$ settles $\mathbf{P}$ as true or $x''$ settles $\mathbf{Q}$ as true.
\end{itemize}

\begin{proof}From left to right, suppose $x$ settles the disjunction as true and $x'$ is a possibility refining $x$. Then by (A1), $x'$ does not settle the disjunction as false. Thus, by (A5), either $x'$ does not settle $\mathbf{P}$ as false, in which case by (A2) there is a possibility $x''$ refining $x'$ that settles $\mathbf{P}$ as true, or $x'$ does not settle $\mathbf{Q}$ as false, in which case by (A2) there is a possibility $x''$ refining $x'$ that settles $\mathbf{Q}$ as true. In either case, there is a possibility $x''$ refining $x'$ that settles $\mathbf{P}$ as true or settles $\mathbf{Q}$ as true. From right to left, we prove the contrapositive. Suppose $x$ does not settle the disjunction as true, so by (A1), there is a possibility $x'$ refining $x$ that settles the disjunction as false. It follows by (A5) that $x'$ settles $\mathbf{P}$ as false and settles $\mathbf{Q}$ as false, which by (A2) implies that no possibility $x''$ refining $x'$ settles $\mathbf{P}$ as true or settles $\mathbf{Q}$ as true.
\end{proof}

\subsection{Summary}

From (A0)-(A5), we have derived the essential elements of classical possibility semantics, which we will see in this chapter:
\begin{itemize}
\item there is a partially ordered set of possibilities;
\item propositions correspond to regular sets of possibilities;
\item truth conditions for negation and disjunction involve quantification over refinements as in (C3) and (C4).
\end{itemize}

\section{Boolean case}\label{BooleanSection}

In both classical possible world semantics and classical possibility semantics, propositions form a Boolean algebra (BA) under the operations of negation, conjunction, and disjunction. The key difference between these semantics is in how such BAs of propositions are represented. In this section, we review the ways of representing BAs used in possibility semantics. After some preliminaries concerning partially ordered sets in \S~\ref{PosetsOfPoss}, we review a well-known way of representing complete BAs in \S~\ref{CompleteBooleanSection}: we represent a complete BA as the BA of \textit{regular open sets} of a partially ordered set. The way of representing arbitrary BAs in \S\S~\ref{ArbitraryBooleanSection1}-\ref{TopFrameSection} is the basis for the recent theory of ``choice-free Stone duality''~\cite{BH2020}: we represent an arbitrary BA as the BA of \textit{compact regular open sets} of an appropriate topological space. It is a short step from these \textit{representations} to \textit{semantics} for formal languages: we simply interpret formulas as propositions in the BAs represented using our partially ordered sets or topological spaces. Thus, the algebraic and topological material in this section is at the core of possibility semantics.

\subsection{Posets of possibilities}\label{PosetsOfPoss}

A \textit{partially ordered set} (or \textit{poset}) is a pair $(S,\sqsubseteq)$ of a nonempty set $S$ and a partial order $\sqsubseteq$ on $S$. As in \S~\ref{PhilSection}, we think of elements of $S$ as possibilities, and for $x,y\in S$, we take $x\sqsubseteq y$ to mean that $x$ refines $y$. Define $x\sqsubset y$ to mean $x\sqsubseteq y$ and $x\neq y$. 

A set $U\subseteq S$ is a \textit{downward-closed subset} or \textit{downset} of $(S,\sqsubseteq)$ if $U$ is closed under refinement: for all $x\in U$ and $x'\sqsubseteq x$, we have $x'\in U$. For $x\in S$, the \textit{principal downset of $x$} is the set of all possibilities refining~$x$:
\[\mathord{\downarrow}x=\{y\in S\mid y\sqsubseteq x\}.\]

We are especially interested in posets in which every possibility can be properly refined (so there are no \textit{worlds} in the sense of Definition \ref{WorldDef}): for all $x\in S$, there is a $y\in S$ such that $y\sqsubset x$. Two running examples will be the full infinite binary tree and the collection of open intervals of $\mathbb{Q}$ ordered by inclusion.

\begin{example}\label{BinaryTree0} Consider the full infinite binary tree $2^{<\omega}$:
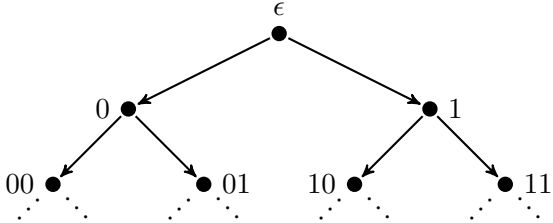
\begin{center}
\begin{tikzpicture}[yscale=1, ->,>=stealth',shorten >=1pt,shorten <=1pt, auto,node
distance=2cm,thick,every loop/.style={<-,shorten <=1pt}]
\tikzstyle{every state}=[fill=gray!20,draw=none,text=black]
\node[circle,draw=black!100,fill=black!100, label=above:$\epsilon$,inner sep=0pt,minimum size=.175cm] (0) at (0,0) {{}};
\node[circle,draw=black!100,fill=black!100, label=left:$0$,inner sep=0pt,minimum size=.175cm] (00) at (-2,-1) {{}};
\node[circle,draw=black!100,fill=black!100, label=right:$1$,inner sep=0pt,minimum size=.175cm] (01) at (2,-1) {{}};
\node[circle,draw=black!100,fill=black!100, label=left:$00$,inner sep=0pt,minimum size=.175cm] (000) at (-3,-2) {{}};
\node[circle,draw=black!100,fill=black!100, label=right:$01$,inner sep=0pt,minimum size=.175cm] (001) at (-1,-2) {{}};
\node[circle,draw=black!100,fill=black!100, label=left:$10$,inner sep=0pt,minimum size=.175cm] (010) at (1,-2) {{}};
\node[circle,draw=black!100,fill=black!100, label=right:$11$,inner sep=0pt,minimum size=.175cm] (011) at (3,-2) {{}};
\path (0) edge[->] node {{}} (00);
\path (0) edge[->] node {{}} (01);
\path (00) edge[->] node {{}} (000);
\path (00) edge[->] node {{}} (001);
\path (01) edge[->] node {{}} (010);
\path (01) edge[->] node {{}} (011);

\node at (-3.3,-2.3) {{$ \rotatebox{45}{\dots}$}};
\node at (-2.66,-2.34) {{$ \rotatebox{-45}{\dots}$}};

\node at (-1.3,-2.3) {{$ \rotatebox{45}{\dots}$}};
\node at (-0.66,-2.34) {{$ \rotatebox{-45}{\dots}$}};

\node at (3.34,-2.34) {{$ \rotatebox{-45}{\dots}$}};
\node at (2.7,-2.3) {{$ \rotatebox{45}{\dots}$}};

\node at (1.34,-2.34) {{$ \rotatebox{-45}{\dots}$}};
\node at (0.7,-2.3) {{$ \rotatebox{45}{\dots}$}};

\end{tikzpicture}
\end{center}
We may regard $2^{<\omega}$ as a poset $(S,\sqsubseteq)$ where $S$ is the set of all finite sequences of 0's and 1's and for any such sequences $\sigma$ and $\tau$, we have $\sigma\sqsubseteq\tau$ iff $\tau$ is an initial segment of $\sigma$. Thus, in the diagram above (and later diagrams), an arrow from $\tau$ to $\sigma$ indicates $\sigma\sqsubseteq\tau$. Viewing this as a poset of possibilities, each possibility may be taken to settle a finite sequence of ``yes or no'' questions. If $\sigma\sqsubseteq\tau$, then $\sigma$ answers all the questions that $\tau$ does in the same way as $\tau$ does but may answer additional questions on which $\tau$ is silent. No possibility settles all questions.
\end{example}

\begin{example}\label{IntervalEx} Consider the set  $S=\{(a,b)\mid a,b\in \mathbb{Q},a<b\}$ of all nonempty open intervals $(a,b)=\{x\in \mathbb{Q}\mid a<x<b\}$ of rational numbers. Let $\sqsubseteq$ be the inclusion order: $(a,b)\sqsubseteq (c,d)$ if $(a,b)\subseteq (c,d)$. Thus, we obtain infinite sequences of refinements from infinite chains of shrinking intervals:
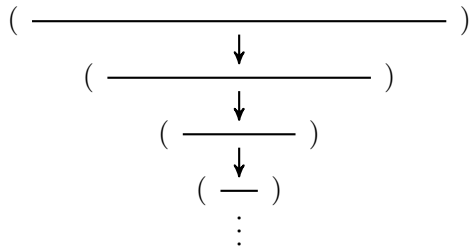
\begin{center}
\begin{tikzpicture}[yscale=1, ->,>=stealth',shorten >=1pt,shorten <=1pt, auto,node
distance=2cm,thick,every loop/.style={<-,shorten <=1pt}]
\tikzstyle{every state}=[fill=gray!20,draw=none,text=black]

\node (c) at (0,0) {{}};
\node (l) at (-3,0) {{$($}};
\node (r) at (3,0) {{$)$}};
\path (l) edge[-] node {{}} (r);

\node (c1) at (0,-.75) {{}};
\node (l1) at (-2,-.75) {{$($}};
\node (r1) at (2,-.75) {{$)$}};
\path (l1) edge[-] node {{}} (r1);

\node (c2) at (0,-1.5) {{}};
\node (l2) at (-1,-1.5) {{$($}};
\node (r2) at (1,-1.5) {{$)$}};
\path (l2) edge[-] node {{}} (r2);

\node (c3) at (0,-2.25) {{}};
\node (l3) at (-.5,-2.25) {{$($}};
\node (r3) at (.5,-2.25) {{$)$}};
\path (l3) edge[-] node {{}} (r3);

\node at (0,-2.75) {{$ \rotatebox{90}{\dots}$}};

\path (c) edge[->] node {{}} (c1);
\path (c1) edge[->] node {{}} (c2);
\path (c2) edge[->] node {{}} (c3);

\end{tikzpicture}
\end{center}
Adopting the temporal interpretation of the poset $(S,\sqsubseteq)$ in \cite[p.~59]{Benthem1983a}, we may think of a possibility as settling that we are now temporally located in some stretch (or ``period'' or ``region'') of time. There is no possibility of a sharpest localization, i.e., no temporal ``instants.''
\end{example}

\subsection{Posets and complete Boolean algebras}\label{CompleteBooleanSection}

In this section, we review how a poset of possibilities gives rise to a (complete) Boolean algebra of propositions.

Given a poset $(S,\sqsubseteq)$, say that a subset $U\subseteq S$ is \textit{regular open} in $(S,\sqsubseteq)$  if
\begin{equation}U=\{x\in S\mid \forall x'\sqsubseteq x\;\exists x''\sqsubseteq x': x''\in U\}.\label{ROeq}\end{equation}
This is the same notion of regularity introduced in \S~\ref{RegularSets}. Equivalently, $U\subseteq S$ is regular open in $(S,\sqsubseteq)$ if $U$ satisfies:
\begin{itemize}
\item \textit{persistence}: if $x\in U$ and $x'\sqsubseteq x$, then $x'\in U$;
\item \textit{refinability}: if $x\not\in U$, then $\exists x'\sqsubseteq x$ $\forall x''\sqsubseteq x'$ $x''\not\in U$.
\end{itemize}
Persistence is just the condition of being a downset from \S~\ref{PosetsOfPoss}. If we think of $U$ as a \textit{proposition}, then persistence is built into our notion of refinement from \S~\ref{RefinementSection}. As for refinability, the philosophical explanation comes from \S~\ref{TruthFalsity}: if $x$ does not make $U$ true, then there must be a refinement $x'$ of $x$ that makes $U$~\textit{false}.

\begin{example}\label{BinaryTree0b} In the infinite binary tree from Example \ref{BinaryTree0}, every principal downset is regular open: persistence is immediate, and for refinability, if $y\not\in\mathord{\downarrow}x$, then there is a child $y'\sqsubseteq y$ such that for all $y''\sqsubseteq y'$, ${y''\not\in \mathord{\downarrow}x}$. The union of $\mathord{\downarrow}x$ and $\mathord{\downarrow}x'$ is also regular open provided $x$ and $x'$ are not children of the same node. By contrast, consider $\mathord{\downarrow}00 \cup \mathord{\downarrow}01$, represented by the area under the dashed curve:
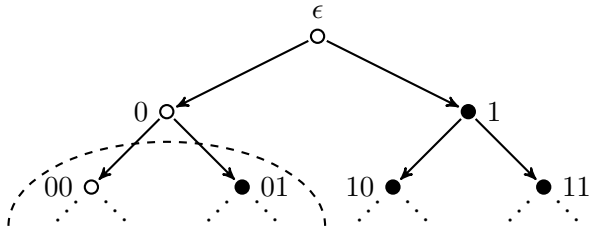
\begin{center}
\begin{tikzpicture}[yscale=1, ->,>=stealth',shorten >=1pt,shorten <=1pt, auto,node
distance=2cm,thick,every loop/.style={<-,shorten <=1pt}]
\tikzstyle{every state}=[fill=gray!20,draw=none,text=black]
\node[circle,draw=black!100, label=above:$\epsilon$,inner sep=0pt,minimum size=.175cm] (0) at (0,0) {{}};
\node[circle,draw=black!100, label=left:$0$,inner sep=0pt,minimum size=.175cm] (00) at (-2,-1) {{}};
\node[circle,draw=black!100,fill=black!100, label=right:$1$,inner sep=0pt,minimum size=.175cm] (01) at (2,-1) {{}};
\node[circle,draw=black!100, label=left:$00$,inner sep=0pt,minimum size=.175cm] (000) at (-3,-2) {{}};
\node[circle,draw=black!100,fill=black!100, label=right:$01$,inner sep=0pt,minimum size=.175cm] (001) at (-1,-2) {{}};
\node[circle,draw=black!100,fill=black!100, label=left:$10$,inner sep=0pt,minimum size=.175cm] (010) at (1,-2) {{}};
\node[circle,draw=black!100,fill=black!100, label=right:$11$,inner sep=0pt,minimum size=.175cm] (011) at (3,-2) {{}};
\path (0) edge[->] node {{}} (00);
\path (0) edge[->] node {{}} (01);
\path (00) edge[->] node {{}} (000);
\path (00) edge[->] node {{}} (001);
\path (01) edge[->] node {{}} (010);
\path (01) edge[->] node {{}} (011);

\node at (-3.3,-2.3) {{$ \rotatebox{45}{\dots}$}};
\node at (-2.66,-2.34) {{$ \rotatebox{-45}{\dots}$}};

\node at (-1.3,-2.3) {{$ \rotatebox{45}{\dots}$}};
\node at (-0.66,-2.34) {{$ \rotatebox{-45}{\dots}$}};

\node at (3.34,-2.34) {{$ \rotatebox{-45}{\dots}$}};
\node at (2.7,-2.3) {{$ \rotatebox{45}{\dots}$}};

\node at (1.34,-2.34) {{$ \rotatebox{-45}{\dots}$}};
\node at (0.7,-2.3) {{$ \rotatebox{45}{\dots}$}};

\draw[dashed,-] (-4.1,-2.55) arc (180:0: 2.1cm and 1.15cm);

\end{tikzpicture}
\end{center}
\noindent The set $\mathord{\downarrow}00 \cup \mathord{\downarrow}01$ is not regular open: for $0\not\in \mathord{\downarrow}00 \cup \mathord{\downarrow}01$, yet there is no  $y\sqsubseteq 0$ such that for all $z\sqsubseteq y$, $z\not\in \mathord{\downarrow}00 \cup \mathord{\downarrow}01$, so refinability fails for $\mathord{\downarrow}00 \cup \mathord{\downarrow}01$. Another set that is not regular open is $U=\{\sigma\in S\mid \sigma\mbox{ contains at least one }1\}$, represented by the filled-in black nodes in the diagram above: for $\epsilon\not\in U$, yet there is no $y\sqsubseteq \epsilon$ such that for all $z\sqsubseteq y$, $z\not\in U$, so refinability fails for $U$.\end{example}

\begin{example}\label{IntervalEx2} In the poset of intervals from Example \ref{IntervalEx}, every principal downset is regular open. Moreover, if $(a_i,b_i)$ and $(a_j,b_j)$ are such that ${b_i<a_j}$, then $U=\mathord{\downarrow}(a_i,b_i)\cup \mathord{\downarrow}(a_j,b_j)$ is regular open. For suppose $(x,y)\not\in U$, so that $(x,y)\not\subseteq (a_i,b_i)$ and $(x,y)\not\subseteq (a_j,b_j)$. We choose a subinterval $(x',y')\subseteq (x,y)$ as follows. If $(x,y)\cap ((a_i,b_i)\cup (a_j,b_j))=\varnothing$, let $(x',y')=(x,y)$. Otherwise, suppose $(x,y)\cap ((a_i,b_i)\cup (a_j,b_j))\neq\varnothing$. If $(x,y)\cap (a_i,b_i)\neq\varnothing$ and $(x,y)\cap (a_j,b_j)\neq\varnothing$, let $(x',y')=(b_i,a_j)$. In a diagram (for the case where $a_i<x$ and $y<b_j$):
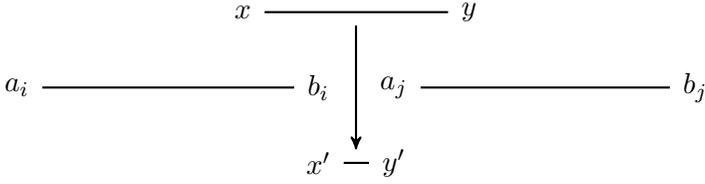
\begin{center}
\begin{tikzpicture}[yscale=1, ->,>=stealth',shorten >=1pt,shorten <=1pt, auto,node
distance=2cm,thick,every loop/.style={<-,shorten <=1pt}]
\tikzstyle{every state}=[fill=gray!20,draw=none,text=black]

\node (c) at (0,0) {{}};
\node (l) at (-2,0) {{$a_i$}};
\node (r) at (2,0) {{$b_i$}};
\path (l) edge[-] node {{}} (r);

\node (c1) at (5,0) {{}};
\node (l1) at (3,0) {{$a_j$}};
\node (r1) at (7,0) {{$b_j$}};
\path (l1) edge[-] node {{}} (r1);

\node (c2) at (2.5,1) {{}};
\node (l2) at (1,1) {{$x$}};
\node (r2) at (4,1) {{$y$}};
\path (l2) edge[-] node {{}} (r2);

\node (c3) at (2.5,-1) {{}};
\node (l3) at (2,-1) {{$x'$}};
\node (r3) at (3,-1) {{$y'$}};
\path (l3) edge[-] node {{}} (r3);

\path (c2) edge[->] node {{}} (c3);

\end{tikzpicture}
\end{center}
If $(x,y)\cap (a_i,b_i)=\varnothing$, then let $(x',y')=(x,a_j)$ if $x<a_j$ and $(x',y')=(b_j,y)$ if $a_j\leq x$ and $b_j<y$; similarly, if $(x,y)\cap (a_j,b_j)=\varnothing$, then let $(x',y')=(x,a_i)$ if $x<a_i$ and $(x',y')=(b_i,y)$ if $a_i\leq x$ and $b_i<y$. Then observe that for all $(x'',y'')\subseteq (x',y')$, we have $(x'',y'')\not\in U$, so refinability holds for $U$. 

By contrast, $V=\mathord{\downarrow}(a,b)\cup\mathord{\downarrow}(b,c)$ is \textit{not} regular open: for ${(a,c)\not\in V}$, yet for every subinterval of $(a,c)$, there is a further subinterval that is either a subinterval of $(a,b)$ or a subinterval of $(b,c)$, so refinability fails for $V$.\end{example}

We now recall how the regular open sets of a poset form a Boolean algebra of propositions. The following result is widely used in  forcing in set theory \cite{Takeuti1973} and can be considered the crucial starting point for possibility semantics.

\begin{theorem}\label{FirstThm} $\,$\textnormal{
\begin{enumerate}
\item\label{FirstThm1} For any poset $(S,\sqsubseteq)$, the collection $\mathcal{RO}(S,\sqsubseteq)$ of regular open sets, ordered by inclusion, forms a complete Boolean algebra with the following operations:
\begin{eqnarray*}
\neg U&=&\{x\in S\mid \forall x'\sqsubseteq x\;\, x'\not\in U\}; \\
\bigwedge\{U_i\mid i\in I\}&=&\bigcap \{U_i\mid i\in I\};\\
\bigvee\{U_i\mid i\in I\}&=&\{x\in S\mid \forall x'\sqsubseteq x\;\exists x''\sqsubseteq x': x''\in \bigcup \{U_i\mid i\in I\}\}.
\end{eqnarray*}
Elements of $\mathcal{RO}(S,\sqsubseteq)$ are precisely those $U\subseteq S$ such that $U=\neg\neg U$.
\item\label{FirstThm2} Given any Boolean algebra $(B,\leq)$, let $B_+$ be the result of deleting the bottom element of $B$, and let $\leq_+$ be the restriction of $\leq$ to $B_+$. Then $(B,\leq)$ embeds into $\mathcal{RO}(B_+,\leq_+)$ as a regular subalgebra;\footnote{Recall that a \textit{regular subalgebra} of a Boolean algebra $\mathcal{B}$ is a subalgebra $\mathcal{B}'$ such that if a set $S$ of elements of $\mathcal{B}'$ has a join $a$ in $\mathcal{B}'$, then $a$ is also the join of $S$ in $\mathcal{B}$. Contrast Theorem \ref{FirstThm}.\ref{FirstThm2} with the fact that not every Boolean $\sigma$-algebra can be represented as a regular subalgebra of a powerset algebra (see, e.g., \cite[Ch.~25]{Givant2009}).} and if $(B,\leq)$ is complete, then  it is isomorphic to $\mathcal{RO}(B_+,\leq_+)$. 
\end{enumerate}}\end{theorem}
\begin{proof}[Sketch of 2] We claim that the map $\varphi: B\to \mathcal{RO}(B_+,\leq_+)$ given by $\varphi(b)=\mathord{\downarrow}_+b:=\{b'\in B_+\mid b'\leq_+ b\}$ is a Boolean embedding such that for any family $\{a_i\mid i\in I\}\subseteq B$, if the join of $\{a_i\mid i\in I\}$ exists in $(B,\leq)$, then 
\begin{equation} \mathord{\downarrow}_+\bigvee \{a_i\mid i\in I\} = \bigvee \{\mathord{\downarrow}_+a_i \mid i\in I\}.\label{JoinComm}\end{equation}
It is easy to see that $\mathord{\downarrow}_+b\in\mathcal{RO}(B_+,\leq_+)$, $\varphi$ is injective, and $ \mathord{\downarrow}_+ \neg b = \neg \mathord{\downarrow}_+b$. As for (\ref{JoinComm}), let $a= \bigvee \{a_i\mid i\in I\}$ and $A=\bigcup \{\mathord{\downarrow}_+a_i\mid i\in I\}$. Suppose $b\in B_+$ is not in the left-hand side of (\ref{JoinComm}), so $b\not\leq a$. Then $b':=b\wedge\neg a \neq 0$, and for all $b''\leq_+ b'$, we have $b''\not\in A$. Thus, $b$ is not in the right-hand side by the definition of join in $\mathcal{RO}(B_+,\leq_+)$. Now suppose $b\in B_+$ is in the left-hand side, so $b\leq a$, and consider $b'\leq_+ b$. Suppose for contradiction that there is no $b''\leq_+ b'$ with $b''\in A$. It follows that $b'\wedge a_i=0$ for each $i\in I$, so $\bigvee \{b'\wedge a_i \mid i\in I\}=0$. Then by the join-infinite distributive law for Boolean algebras, $b'\wedge a=0$, contradicting $b'\leq b\leq a$ and the fact that $b'\in B_+$. Thus, we conclude there is a $b''\leq_+ b'$ such that $b''\in A$. Hence $b$ is in the right-hand side.

Since for each $U\in \mathcal{RO}(B_+,\leq_+)$, we have $U=\bigvee \{ \mathord{\downarrow}_+ a\mid a\in U\}$, if $(B,\leq)$ is complete, then (\ref{JoinComm}) implies that $\varphi$ is surjective and hence an isomorphism. \end{proof}

Thus, any complete Boolean algebra can be represented as the regular opens of a poset. The poset $(B_+,\leq_+)$ is an example of a \textit{separative} poset, as in \S~\ref{TruthFalsity}.

\begin{definition}\label{SepDef} A poset $(S,\sqsubseteq)$ is \textit{separative} if  every principal downset $\mathord{\downarrow}x$ is regular open in $(S,\sqsubseteq)$; equivalently, for all $x,y\in S$, if $y\not\sqsubseteq x$, then there is a $z\sqsubseteq y$ such that $\mathord{\downarrow}z\cap\mathord{\downarrow}x=\varnothing$.
\end{definition}
\noindent The posets in Examples \ref{BinaryTree0} and \ref{IntervalEx} are also separative. By contrast, a linear order with more than one element is an example of a non-separative poset.

Separative posets have the following good properties for our purposes.

\begin{proposition}\label{SepProp} $\,$\textnormal{
\begin{enumerate}
\item\label{SepProp1} For each poset, its regular open algebra is isomorphic to the regular open algebra of a separative poset.
\item\label{SepProp2} The restriction of a separative poset to a regular open subset is again a separative poset.
\item\label{SepProp3} If $(S,\sqsubseteq)$ is a separative poset, then the map $x\mapsto \mathord{\downarrow}x$ is a dense order-embedding of the poset $(S,\sqsubseteq)$ into the poset $(\mathcal{RO}(S,\sqsubseteq),\subseteq)$.\footnote{By a \textit{dense} order-embedding of a poset $(S,\sqsubseteq)$ into a poset $(S',\sqsubseteq')$, we mean an order-embedding $e$ of $(S,\sqsubseteq)$ into $(S',\sqsubseteq')$ such that $e[S]$ is a dense subset of $(S',\sqsubseteq')$, i.e., for every $x'\in S'$ there is a $y'\sqsubseteq ' x'$ such that $y'\in e[S]$.}
\item\label{SepProp4} If $(S,\sqsubseteq)$ is a separative poset and $U,V\in\mathcal{RO}(S,\sqsubseteq)$ are such that $U\vee V$ is infinite, then $U$ is infinite or $V$ is infinite.
\end{enumerate}}
\end{proposition}

\begin{proof} For part \ref{SepProp1}, given $(S,\sqsubseteq)$, the poset $(\mathcal{RO}(S,\sqsubseteq)_+,\subseteq_+)$ obtained in Theorem \ref{FirstThm}.\ref{FirstThm2} from the BA $\mathcal{RO}(S,\sqsubseteq)$ is such a separative poset. More directly, one can take the quotient of $(S,\sqsubseteq)$ by the equivalence relation $\sim$ defined by: $x\sim y$ iff (i) $\forall x'\sqsubseteq x$ $\exists x''\sqsubseteq x': x''\sqsubseteq y$ and (ii) $\forall y'\sqsubseteq y$ $\exists y''\sqsubseteq y': y''\sqsubseteq x$. 

Parts \ref{SepProp2}-\ref{SepProp3} are also easy to check.

For part \ref{SepProp4},\footnote{The proof is nearly the same as in \cite[Lemma 2.3]{BH2020} for the case of $V=\neg U$.} by part \ref{SepProp2} we may assume without loss of generality that $S=U\vee V$. Let $x\sim y$ iff $\mathord{\downarrow}x\cap U=\mathord{\downarrow}y\cap U$. If $U$ is finite, then $\sim$ partitions the infinite set $S$ into finitely many cells, one of which must be infinite. Call it $I$, and define $f\colon I\to \wp(V)$ by $f(x)=\mathord{\downarrow}x\cap  V$. We claim that $f$ is injective. For if $x,y\in I$ and $y\not\sqsubseteq x$, then by separativity, there is a $z\in \mathord{\downarrow}y$ such that $\mathord{\downarrow}z\cap\mathord{\downarrow}x=\varnothing$. It follows, since $\mathord{\downarrow}x\cap U=\mathord{\downarrow}y\cap U$, that $\mathord{\downarrow}z\cap U=\varnothing$, so $z\in \neg U$, which with $z\in U\vee V$ implies $z\in V$. Thus, $z\in f(y)$ but $z\not\in f(x)$, so $f$ is injective. Then since $I$ is infinite, it follows that $\wp(V)$ is infinite and hence $V$ is infinite\end{proof}

Part \ref{FirstThm2} of Theorem \ref{FirstThm} can be restated in the form of Proposition \ref{MacNeilleProp} below.  Recall that the MacNeille completion of a Boolean algebra $B$ is the unique (up to isomorphism) complete Boolean algebra $B^*$ such that there is a dense embedding of $B$ into $B^*$, i.e., a Boolean embedding of $B$ into $B^*$ such that every element of $B^*$ is a join of images of elements of $B$ (see, e.g., \cite[Ch.~25]{Givant2009}).\footnote{Compare the realization of $B^*$ as in Proposition \ref{MacNeilleProp} with its realization as the regular open algebra of the Stone space of $B$ (see Remark \ref{MacNeilleCanonicalComparison}).} 

\begin{proposition}\label{MacNeilleProp} \textnormal{For any Boolean algebra $B$, $\mathcal{RO}(B_+,\leq_+)$ is (up to isomorphism) the MacNeille completion of $B$.}
\end{proposition}

\begin{proof} For $A\subseteq B$, let $A^u$ be the set of all upper bounds of $A$ and $A^\ell$ the set of all lower bounds of $A$. Recall that the MacNeille completion of $B$ (as a poset) can be constructed as $(\{A \mid A\subseteq B, A=A^{u\ell}\},\subseteq)$. Note that this is isomorphic to $(\{A\setminus\{0\} \mid A\subseteq B, A=A^{u\ell}\},\subseteq)$. Now we claim that for any $A\subseteq B$,
\begin{equation}A^{u\ell}\setminus\{0\}=\neg\neg A\label{MacEq}\end{equation}
where $\neg\neg A=\{b\in B_+\mid \forall b'\leq_+ b\;\exists b''\leq_+ b':b''\in A \}$. Since elements of $\mathcal{RO}(B_+,\leq_+)$ are precisely those $A\subseteq B$ such that $A=\neg\neg A$, equation (\ref{MacEq}) implies that  $(\{A\setminus\{0\} \mid A\subseteq B, A=A^{u\ell}\},\subseteq)= (\mathcal{RO}(B_+,\leq_+),\subseteq)$. To prove (\ref{MacEq}), let $b$ be a nonzero element of $B$. If $b$ is not in the left-hand side of (\ref{MacEq}), then there is some $a\in A^u$ such that $b\not\leq a$, which implies $b\wedge \neg a\neq 0$. Since $a\in A^u$, for every $c\in A$, $c\leq a$ and hence $c\wedge\neg a=0$. Thus, setting $b'=b\wedge\neg a$, we have $b'\leq_+ b$ and for all $ b''\leq_+b'$, $b''\not\in A$. Hence $b$ is not in the right-hand side. Conversely, suppose $b$ is not in the right-hand side of (\ref{MacEq}), so there is some $b'\leq_+b$ such that for all $b''\leq_+b'$, $b''\not\in A$. It follows that for every $c\in A$, $b'\wedge c=0$, so $c\leq \neg b'$. Thus, $\neg b'\in A^u$, so $b'\not\in A^{u\ell}$ and then $b\not\in A^{u\ell}$. Hence $b$ is not in the left-hand side.
\end{proof}

The reason for calling sets satisfying (\ref{ROeq}) ``regular open'' is the following. In a topological space, a \textit{regular open} set is an open set $U$ such that \[U=\mathsf{int}(\mathsf{cl}(U))\]
where $\mathsf{int}$ and $\mathsf{cl}$ are the interior and closure operations, respectively. Any poset $(S,\sqsubseteq)$ may be regarded as a topological space $(S,\tau)$, where the family $\tau$ of open sets is the family $\mathsf{Down}(S,\sqsubseteq)$ of all downsets of $(S,\sqsubseteq)$. Such spaces are $T_0$, i.e., for any two distinct points, there is an open set containing one point but not the other, as well as Alexandroff, i.e., the intersection of any family of opens is open. Indeed, $T_0$ Alexandroff spaces are in one-to-one correspondence with posets. Recall that the \textit{specialization order} of a space $(S,\tau)$ is the binary relation on $S$ defined by: $x\leqslant y$ iff for every $U\in\tau$, $x\in U$ implies $y\in U$. Given a $T_0$ Alexandroff space $(S,\tau)$ with specialization order $\leqslant$, whose converse is $\geqslant$, the poset $(S,\geqslant)$ is such that $(S,\mathsf{Down}(S,\geqslant))=(S,\tau)$.\footnote{This is a case where thinking in terms of the family of all upward-closed sets would be more convenient, but Theorem \ref{FirstThm}.\ref{FirstThm2} is a case where thinking in terms of downsets is more convenient, as we do not have to flip the relation of the Boolean algebra. There is no winning in choosing up or down.} For any poset $(S,\sqsubseteq)$, the interior and closure operations of the space $(S,\mathsf{Down}(S,\sqsubseteq))$ are given by
\begin{eqnarray*}
\mathsf{int}(U)&=&\{x\in S\mid \forall x'\sqsubseteq x\;\, x'\in U\}\\
\mathsf{cl}(U)&=&\{x\in S\mid \exists x'\sqsubseteq x: x'\in U\},
\end{eqnarray*}
so that 
\[\mathsf{int}(\mathsf{cl}(U))=\{x\in S\mid \forall x'\sqsubseteq x\,\exists x''\sqsubseteq x': x''\in U\}.\]
Thus, a subset $U$ of a poset is regular open in the topological space $(S,\mathsf{Down}(S,\sqsubseteq))$ iff $U$ is regular open as defined above Theorem \ref{FirstThm}. 

Now Theorem \ref{FirstThm}.\ref{FirstThm1} is a special case of the following result of Tarski.

\begin{theorem}[\cite{Tarski1937,Tarski1938}]\label{TarskiThm} \textnormal{For any topological space $X$, the collection $\mathcal{RO}(X)$ of regular open sets of $X$, ordered by inclusion, forms a complete Boolean algebra with the following operations:
\begin{eqnarray*}
\neg U&=&\mathsf{int}(X\setminus U); \\
\bigwedge \{U_i\mid i\in I\}&=&\mathsf{int}(\bigcap \{U_i\mid i\in I\});  \\
 \bigvee \{U_i\mid i\in I\}&=&\mathsf{int}(\mathsf{cl}(\bigcup \{U_i\mid i\in I\}).
 \end{eqnarray*}}
\end{theorem}
\noindent Theorem \ref{TarskiThm} can in turn be seen to follow from a more general theorem about nuclei on Heyting algebras, explained in \S~\ref{IntuitionisticSection} (Theorem \ref{FixpointThm}).

Tarski \cite{Tarski1935} also observed that for any set $W$, its powerset $\wp(W)$ is a complete and atomic Boolean algebra (CABA), and for any CABA $B$ whose set of atoms is $\mathsf{At}(B)$, $\wp(\mathsf{At}(B))$ is isomorphic to $B$.\footnote{For another route to CABAs based on postulates about consistent sets of propositions, see~\cite{FritzForthcoming}.} Let us relate this representation of CABAs to the representation of complete Boolean algebras (CBAs) in Theorem \ref{FirstThm}. For this we introduce the following notion, inspired by the philosophers' conception of a ``possible world'' as maximally specific and hence having no proper refinements.

\begin{definition}\label{WorldDef} A \textit{world} in a poset $(S,\sqsubseteq)$ is an element $x\in S$ such that for all $x'\in S$, if $x'\sqsubseteq x$ then $x'=x$.\end{definition}

\begin{proposition} \textnormal{Let $(S,\sqsubseteq)$ be a poset in which for every $x\in S$, there is a world $x'\sqsubseteq x$. Then $\mathcal{RO}(S,\sqsubseteq)$ is isomorphic to the powerset algebra $\wp(W)$, where $W$ is the set of worlds of $(S,\sqsubseteq)$.}\end{proposition}

\begin{proof}[Proof sketch] The isomorphism sends $U\in \mathcal{RO}(S,\sqsubseteq)$ to $U\cap W$.\end{proof}

We conclude this section with a basic fact about complete BAs whose proof shows (i) how the representation in Theorem \ref{FirstThm} can be used to prove facts about complete BAs and (ii) how every poset has the same regular open algebra as the union of a \textit{poset of only worlds} and a \textit{poset with no worlds}.

\begin{proposition} \textnormal{Any complete BA $B$ is the product of a complete and atomic BA and a complete and atomless BA.}
\end{proposition}

\begin{proof} Let $(B_+,\leq_+)$ be the poset given by Theorem \ref{FirstThm}.\ref{FirstThm2}. Where $x<_+y$ iff $x\leq_+y$ and $y\not\leq_+x$, let \[B_\star=\{x\in B_+\mid \forall x' <_+ x \;\exists x''<_+x'\},\] and let $\leq_\star$ be the restriction of $\leq_+$ to $B_\star$. Thus, we have \textit{deleted all possibilities that are properly refined by worlds}. Then it is easy to check the following:
\begin{enumerate}
\item The map $U\mapsto U\cap B_\star$ is an isomorphism from $\mathcal{RO}(B_+,\leq_+)$ to $\mathcal{RO}(B_\star,\leq_\star)$.
\item $B_\star$ is the union of its ``atomic'' part $A=\{x\in B_\star\mid \neg\exists x' <_\star x\}$ and its ``atomless'' part $C=\{x\in B_\star\mid \forall x'\leq_\star x \;\exists x''<_\star x'\}$.
\item Where $\leq_A$ and $\leq_C$ are the restrictions of $\leq_\star$ to $A$ and $C$, respectively, the BAs $\mathcal{RO}(A,\leq_A)$ and $\mathcal{RO}(C,\leq_C)$ are atomic and atomless, respectively.
\item The map $(U,V)\mapsto U\cup V$ is an isomorphism from $\mathcal{RO}(A,\leq_A)\times \mathcal{RO}(C,\leq_C)$ to $\mathcal{RO}(B_\star,\leq_\star)$.
\end{enumerate}
Putting the above facts together, $B$ is isomorphic to $\mathcal{RO}(A,\leq_A)\times \mathcal{RO}(C,\leq_C)$.
\end{proof}

\subsection{Possibility frames and Boolean algebras}\label{ArbitraryBooleanSection1}

We now move from the representation of complete Boolean algebras to that of arbitrary Boolean algebras. We consider two closely related representations: the first, given in this section, is in the spirit of ``general frame'' theory in modal logic \cite[\S~5]{Blackburn2001}: propositions are not \textit{arbitrary} regular open sets of possibilities, but rather certain distinguished regular open sets. The second representation, given in \S~\ref{TopFrameSection}, ``topologizes'' the first representation: propositions are not arbitrary regular open sets of possibilities, but rather those regular open sets with a special property---namely compactness---in a certain kind of topological space.

The most straightforward way to move from complete to arbitrary Boolean algebras is to simply add to a poset $(S,\sqsubseteq)$ a distinguished collection of regular open sets, forming a Boolean subalgebra of $\mathcal{RO}(S,\sqsubseteq)$.

\begin{definition}\label{PosFrameDef} A \textit{possibility frame} is a triple $\mathcal{F}=(S,\sqsubseteq, P)$ such that:
\begin{enumerate}
\item $(S,\sqsubseteq)$ is a poset;
\item $\varnothing\neq P\subseteq\mathcal{RO}(S,\sqsubseteq)$ (elements of $P$ are called \textit{admissible sets});
\item for all $U,V\in P$, $\neg U\in P$ and $U\cap V\in P$.
\end{enumerate}
A possibility frame $\mathcal{F}$ is \textit{full} if $P=\mathcal{RO}(S,\sqsubseteq)$. A \textit{world} frame is a possibility frame in which $\sqsubseteq$ is the identity relation.
\end{definition}
\noindent Obviously full possibility frames are in one-to-one correspondence with posets, and full world frames are in one-to-one correspondence with sets.

Every possibility frame give rise to an associated BA as follows.

\begin{lemma}\label{BAsfromPosFrames} \textnormal{Given a possibility frame $\mathcal{F}= ( S,\sqsubseteq,P ) $, the collection $P$ forms a Boolean subalgebra of  $\mathcal{RO}(S,\sqsubseteq)$, which we denote by $\mathcal{F}^\mathsf{b}$.}\end{lemma}
\begin{proof} Follows from Theorem \ref{FirstThm}.\ref{FirstThm1}.
\end{proof}

\begin{example}\label{BinaryTree} Recall from Example \ref{BinaryTree} the full infinite binary tree regarded as a poset $(S,\sqsubseteq)$. A downset $U$ in $(S,\sqsubseteq)$ is \textit{finitely generated} if there is a finite set $U_0\subseteq S$ such that \[U=\mathord{\downarrow}U_0:=\{y\in S\mid \exists x\in U_0: y\sqsubseteq x\}.\]
Let 
\[P=\{U\in \mathcal{RO}(S,\sqsubseteq)\mid U\mbox{ is finitely generated}\}.\]
Since each node has finite depth and the tree is finitely branching, one can check that for all $U,V\in P$, we have $\neg U\in P$ and $U\cap V\in P$. Hence $\mathcal{F}=(S,\sqsubseteq, P)$ is a possibility frame. The Boolean algebra $\mathcal{F}^\mathsf{b}$ is clearly atomless: for any $U\in P$, $x\in U$, and $y\sqsubset x$, we have $\mathord{\downarrow}y\in P$ and $\mathord{\downarrow}y\subsetneq U$. Then since there are only countably many finitely generated downsets, $\mathcal{F}^\mathsf{b}$ is the unique (up to isomorphism) countable atomless Boolean algebra \cite[Ch.~16]{Givant2009}. Since $\mathcal{F}^\mathsf{b}$ is a dense subalgebra of $\mathcal{RO}(S,\sqsubseteq)$, it follows that $\mathcal{RO}(S,\sqsubseteq)$  is (up to isomorphism) the MacNeille completion of the countable atomless Boolean algebra.\end{example}

\begin{example} Another source of examples of possibility frames comes from semantics for intuitionistic logic. An \textit{intuitionistic general frame} \cite[\S~8.1]{Chagrov1997} is a triple $(S,\sqsubseteq,P)$ where $(S,\sqsubseteq)$ is a poset and $P$ is a set of downsets of $(S,\sqsubseteq)$ containing $\varnothing$ and closed under $\cap$, $\cup$, and the operation $\to$ defined by \[U\to V = \{x\in U\mid \forall x'\sqsubseteq x\, (x'\in U\Rightarrow x'\in V) \}.\]
Then $(S,\sqsubseteq, P\cap \mathcal{RO}(S,\sqsubseteq))$ is a possibility frame, using the fact that $\neg U=U\to \varnothing$.
\end{example}

Below we will consider two main approaches to constructing possibility frames from BAs in \S~\ref{NonzeroProp} and \S~\ref{PropFilters}, respectively: constructing possibilities as \textit{nonzero propositions} or as \textit{proper filters of propositions}---compare this with the use of \textit{atoms} and \textit{prime filters} in possible world semantics, as in Figure \ref{WorldVsPossFig}. Each of these two approaches has two versions, giving us four ways of constructing a possibility frame from a BA, summarized in Figure \ref{FourWays}.

\renewcommand{\arraystretch}{1.5}

\begin{figure}
\begin{tabular}{l|l|l}
& world representation uses\dots & possibility representation uses\dots \\
\hline
for CABAs & atoms (or principal & nonzero elements (or principal  \\
or CBAs: & prime filters) of a CABA & proper filters) of a CBA \\
 \hline
for BAs: & prime filters of a BA & proper filters of a BA
\end{tabular}
\caption{world vs. possibility representations}\label{WorldVsPossFig}
\end{figure}

\begin{figure}{\small
\begin{tabular}{l|l}
\textbf{possibility frame built from a BA $B$} & \textbf{associated BA} \\ 
\hline
full: $B_\mathsf{u}=(B_+,\leq_+, \mathcal{RO}(B_+,\leq_+))$ & $(B_\mathsf{u})^\mathsf{b}$ is MacNeille of $B$  \\
principal: $B_\mathsf{p}= (B_+,\leq_+, \{\mathord{\downarrow}_+a\mid a\in B\})$ & $(B_\mathsf{p})^\mathsf{b}\cong B$ \\
filter: $B_\mathsf{f}= (\mathsf{PropFilt}(B), \supseteq, \mathcal{RO}(\mathsf{PropFilt}(B), \supseteq))$ & $(B_\mathsf{f})^\mathsf{b}$ is canonical extension~of~$B$  \\
general filter: $B_\mathsf{g}= (\mathsf{PropFilt}(B), \supseteq, \{\widehat{a}\mid a\in B\})$ & $(B_\mathsf{g})^\mathsf{b}\cong B$ 
\end{tabular}}

\caption{four ways of building a possibility frame from a BA}\label{FourWays}
\end{figure}

\subsubsection{Constructing possibilities from nonzero propositions}\label{NonzeroProp}

The first way of constructing possibilities already appeared in Theorem \ref{FirstThm}.\ref{FirstThm2}: simply take possibilities to be nonzero propositions. Philosophically, this fits with how we often describe possibilities in natural language, e.g., the possibility \textit{that it is raining in Beijing}, using \textit{that}-clauses that philosophers traditionally take to refer to propositions \cite{McGrath2020}. On this view, given a BA $B$, the set of possibilities is $B_+$ as in Theorem~\ref{FirstThm}.\ref{FirstThm2}. 

\begin{proposition}[\cite{Holliday2018}] \textnormal{Given a Boolean algebra $B$, define its \textit{full frame}
\[B_\mathsf{u}=(B_+,\leq_+, \mathcal{RO}(B_+,\leq_+))\]
and its \textit{principal frame} 
\[B_\mathsf{p}=(B_+,\leq_+, \{\mathord{\downarrow}_+a\mid a\in B\})\]
where $(B_+,\leq_+)$ is as in Theorem \ref{FirstThm}.\ref{FirstThm2} and \[\mathord{\downarrow}_+a=\{b\in B_+\mid b\leq_+ a\},\]
noting that where $0$ is the bottom element of $B$, $\mathord{\downarrow}_+0=\varnothing$. Then:
\begin{enumerate}
\item $B_\mathsf{u}$ and $B_\mathsf{p}$ are possibility frames;
\item $(B_\mathsf{u})^\mathsf{b}$, i.e., $\mathcal{RO}(B_+,\leq_+)$, is (up to isomorphism) the MacNeille completion of $B$ (see Proposition \ref{MacNeilleProp});
\item $(B_\mathsf{p})^\mathsf{b}$ is isomorphic to $B$;\footnote{Note that $\{\mathord{\downarrow}_+a\mid a\in B\}$ ordered by inclusion is isomorphic to $\{\mathord{\downarrow} a\mid a\in B\}$ ordered by inclusion. Dana Scott (personal communication) reports having discussed this ``baby'' representation of BAs by lattices of sets with Tarski over 60 years ago.}
\item $B_\mathsf{u}=B_\mathsf{p}$ if and only if $B$ is complete.
\end{enumerate}}
\end{proposition}

\subsubsection{Constructing possibilities from proper filters of propositions}\label{PropFilters}

To prepare for the second way of constructing possibilities, recall that in a lattice $L$, a \textit{filter} is a nonempty set $F$ of elements of $L$ such that $a,b\in L$ implies $a\wedge b\in L$, and $a\in F$ and $a\leq b$ together imply $b\in F$. $F$ is \textit{proper} if it does not contain all elements of $L$. A proper filter $F$ is \textit{prime} if $a\vee b\in F$ implies that $a\in F$ or $b\in F$. In a Boolean algebra $B$, the prime filters are exactly the \textit{ultrafilters}, those proper filters $F$ such that for all $a\in B$, $a\in F$ or $\neg a\in F$.

The second way of constructing possibilities is inspired by the idea of defining worlds as ultrafilters, or in logical terms, as \textit{maximally consistent sets of sentences}; but since possibilities may be partial, we construct them as proper filters, or in logical terms, as \textit{consistent and deductively closed sets of sentence}. 

\begin{remark}In logic, the idea of constructing canonical models using consistent and deductively closed sets of formulas appears in \cite{Roper1980,Humberstone1981,Benthem1981a,Benthem1986,Benthem1988}. In lattice theory, a representation of ortholattices using proper filters appears in \cite{Goldblatt1974b}, and a representation of arbitrary lattices using filters appears in \cite{Moshier2014}.\end{remark}

While obtaining worlds from arbitrary Boolean algebras requires the nonconstructive Boolean Prime Filter Theorem, which is not provable in ZF set theory \cite{Feferman1964} or even ZF plus Dependent Choice (ZF+DC) \cite{Pincus1977}, this is not needed for the construction of possibilities from arbitrary BAs. All results in this section are provable in ZF, with the exception of Proposition \ref{ChainsProp}, which uses ZF+DC. 

\begin{proposition}[\cite{Holliday2018}]\label{FiltFrameDef} \textnormal{Given a Boolean algebra $B$, define its \textit{filter frame}
  \[B_\mathsf{f}= (\mathsf{PropFilt}(B), \supseteq, \mathcal{RO}(\mathsf{PropFilt}(B), \supseteq)) \]
  and \textit{general filter frame}
  \[B_\mathsf{g}= (\mathsf{PropFilt}(B), \supseteq, \{\widehat{a}\mid a\in B\}) \]
where  $\mathsf{PropFilt}(B)$ is the set of proper filters of $B$, and \[\widehat{a}=\{F\in \mathsf{PropFilt}(B)\mid a\in F\}.\]
Then:
\begin{enumerate}
\item $B_\mathsf{f}$ and $B_\mathsf{g}$ are possibility frames;
\item\label{FiltFrameDef2} $(B_\mathsf{f})^\mathsf{b}$, i.e., $\mathcal{RO}(\mathsf{PropFilt}(B), \supseteq)$, is (up to isomorphism) the canonical extension of $B$;\footnote{Here we use the ``constructive'' definition of canonical extension from \cite{Gehrke2001}, which is equivalent to the standard definition in ZFC but not in ZF: the \textit{constructive canonical extension} of a Boolean algebra $B$ is the unique (up to isomorphism) complete Boolean algebra $C$ for which there is an embedding $e$ of $B$ into $C$ such that every element of $C$ is a join of meets of $e$-images of elements of $B$ (or equivalently in this Boolean case, every element of $C$ is a meet of joins of $e$-images of elements of $B$), and for any sets $X,Y$ of elements of $B$, if $\bigwedge^C e[X]\leq^C \bigvee^C e[Y]$, then there are finite $X'\subseteq X$ and $Y'\subseteq Y$ such that $\bigwedge X'\leq \bigvee Y'$. Compare the realization of the canonical extension in Theorem \ref{FiltFrameDef}.\ref{FiltFrameDef2} as the regular open algebra of the poset of proper filters with the traditional realization as the powerset of the set of prime filters (see Remark \ref{MacNeilleCanonicalComparison} below).}
\item $(B_\mathsf{g})^\mathsf{b}$ is isomorphic to $B$.
\end{enumerate}}
\end{proposition}

Part 3 points to a duality between Boolean algebras and those special possibility frames in the image of the $(\cdot)_\mathsf{g}$ map. They can be given the following characterization, analogous to the definition of \textit{descriptive frames} in possible world semantics \cite{Goldblatt1974} but using filters instead of ultrafilers. 

\begin{proposition}[\cite{Holliday2018}]\label{SepReal}\textnormal{For any possibility frame $\mathcal{F}= ( S,\sqsubseteq,P) $, the following are equivalent:
\begin{enumerate}
\item $\mathcal{F}$ is isomorphic to $(\mathcal{F}^\mathsf{b})_\mathsf{g}$;
\item $\mathcal{F}$ satisfies the following conditions:
\begin{enumerate}
\item the \textit{separation} condition: for all $x,y\in S$, if $y\not\sqsubseteq x$, then there is a $U\in P$ such that $x\in U$ but $y\not\in U$; 
\item the \textit{filter realization} condition: for every proper filter $F$ in $\mathcal{F}^\mathsf{b}$, there is an $x\in S$ such that $F=\{U\in P\mid x\in U\}$.
\end{enumerate}
\end{enumerate}
We call an $\mathcal{F}$ satisfying (a) \textit{separative}. We call an $\mathcal{F}$ satisfying (a) and (b) \textit{filter-descriptive}.}
\end{proposition}
The separation condition corresponds to the characterization of refinement in (C4) of \S~\ref{ConnectivesSection}. The filter realization condition captures the idea that for any consistent set $F$ of propositions, there is \textit{the possibility of all propositions in $\mathcal{F}$ being true}, which is included in the frame. In contrast to filter-descriptive frames, full frames may easily fail the filter realization condition. For example, when we regard the infinite binary tree in Examples \ref{BinaryTree0}, \ref{BinaryTree0b}, and \ref{BinaryTree} as a full possibility frame $(S,\sqsubseteq,\mathcal{RO}(S,\sqsubseteq))$, the filter  $F$ generated by the set $F_0=\{\mathord{\downarrow}0,\mathord{\downarrow}00, \mathord{\downarrow}000,\dots\}$ of propositions is  proper, as any finite subset of $F_0$ has a nonempty intersection; but there is no possibility in $S$ that belongs to all propositions in $F_0$.

Let us note two useful consequences of separation and filter realization.

\begin{lemma}\label{FiltDLem} \textnormal{
\begin{enumerate}
\item\label{FiltDLem1} If $\mathcal{F}= ( S,\sqsubseteq,P) $ is separative, then the poset $(S,\sqsubseteq)$ is separative.
\item\label{FiltDLem2} If $\mathcal{F}=(S,\sqsubseteq,P)$ satisfies filter realization, then for any $\{U_i\mid i\in I\}\subseteq P$ such that $S=\bigvee \{U_i\mid i\in I\}$, there is a finite $I_0\subseteq I$ such that $S=\bigvee\{U_i\mid i\in I_0\}$.
\end{enumerate}}
\end{lemma}
\begin{proof} For part \ref{FiltDLem1}, suppose $y\not\sqsubseteq x$. Then by the separation condition, there is a $U\in P$ such that $x\in U$ but $y\not\in U$. Then since $U\in\mathcal{RO}(S,\sqsubseteq)$, there is a $z\sqsubseteq y$ such that for all $z'\sqsubseteq z$, $z''\not\in U$, which implies $\mathord{\downarrow}z\cap\mathord{\downarrow}x=\varnothing$.

For part \ref{FiltDLem2}, we first claim that if $P_0\subseteq P$ has the finite intersection property (fip), then $\bigcap P_0\neq\varnothing$. For if $P_0$ has the fip, then the filter $F$ in $\mathcal{F}^\mathsf{b}$ generated by $P_0$ is proper, in which case filter realization implies there is an $x\in S$ such that $F=\{U\in P\mid x\in U\}$, so $x\in \bigcap P_0$. Now suppose toward a  contradiction that for every finite $J\subseteq I$, we have $S\not\subseteq\bigvee\{U_j\mid j\in J\}$, so there is an $x_J\in S\setminus\bigvee  \{U_j\mid j\in J\}$. Then as $\bigvee\{U_j\mid j\in J\}\in \mathcal{RO}(S,\sqsubseteq)$, there is an $x'_J\sqsubseteq x_J$ such that $x'_J\in \neg \bigvee\{U_j\mid j\in J\}=\bigwedge\{\neg U_j\mid j\in J\}=\bigcap \{\neg U_j\mid j\in J\}$. Thus, $\{\neg U_i\mid i\in I\}$ has the fip, so $ \bigcap \{\neg U_i\mid i\in I\}\neq\varnothing$, contradicting $S=\bigvee \{U_i\mid i\in I\}$.\end{proof}

To obtain a categorical duality between filter-descriptive frames and BAs, we introduce morphisms on the possibility side, based on the standard notion of a p-morphism between relational structures (see, e.g., \cite[p.~30]{Chagrov1997}).

\begin{definition}\label{PossPMorph} Given possibility frames $\mathcal{F}= ( S,\sqsubseteq,P ) $ and $\mathcal{F}'= ( S',\sqsubseteq',P' ) $, a \textit{p-morphism} from $\mathcal{F}$ to $\mathcal{F}'$ is a map $h:S\to S'$ satisfying the following conditions for all $x,y\in S$ and $y'\in S'$:
\begin{enumerate}
\item\label{PossPMorph1} for all $ U'\in P'$, $h^{-1}[U']\in P$;
\item \SqForth{}: if $y\sqsubseteq x$, then $h(y)\sqsubseteq h(x)$;
\item\label{PossPMorph2} \SqBack{}: if $y'\sqsubseteq' h(x)$, then $\exists y$: $y\sqsubseteq x$ and $h(y)=y'$ (see Figure \ref{p-morphismFig}).
\end{enumerate}
\end{definition}
\begin{remark} If $\mathcal{F}$ is a \textit{full} possibility frame, then together \SqForth{}  and \SqBack{} imply condition \ref{PossPMorph1} of Definition \ref{PossPMorph} (see Fact 3.5 of \cite{Holliday2018}).
\end{remark}

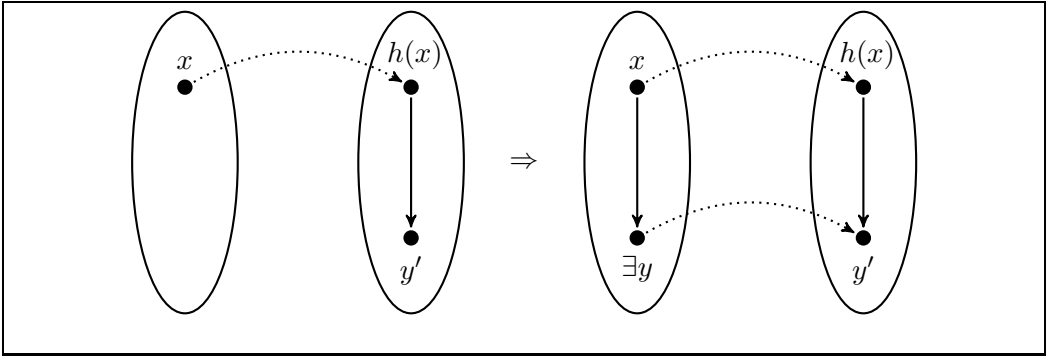
\begin{figure}
\begin{center}
\begin{tikzpicture}[->,>=stealth',shorten >=1pt,shorten <=1pt, auto,node
distance=2cm,thick,every loop/.style={<-,shorten <=1pt}] \tikzstyle{every state}=[fill=gray!20,draw=none,text=black]

\node[circle,draw=black!100,fill=black!100, label=above:$x$,inner sep=0pt,minimum size=.175cm] (x) at (0,2) {{}};

\draw[rotate=90] (1,0) ellipse (2cm and .7cm);

\draw[rotate=90] (1,-3) ellipse (2cm and .7cm);

\draw[rotate=90] (1,-6) ellipse (2cm and .7cm);

\draw[rotate=90] (1,-9) ellipse (2cm and .7cm);

\node[circle,draw=black!100,fill=black!100, label=above:$\;h(x)$,inner sep=0pt,minimum size=.175cm] (x') at (3,2) {{}};

\node[circle,draw=black!100,fill=black!100, label=below:$y'$,inner sep=0pt,minimum size=.175cm] (y') at (3,0) {{}};

\path (x') edge[->] node {{}} (y'); 

\path (x) edge[bend left,dotted,->] node {{}} (x');

\node[circle,draw=black!100,fill=black!100, label=above:$x$,inner sep=0pt,minimum size=.175cm] (x2) at (6,2) {{}};

\node[circle,draw=black!100,fill=black!100, label=below:$\exists y$,inner sep=0pt,minimum size=.175cm] (y2) at (6,0) {{}};

\node[circle,draw=black!100,fill=black!100, label=above:$\;h(x)$,inner sep=0pt,minimum size=.175cm] (x'2) at (9,2) {{}};

\node[circle,draw=black!100,fill=black!100, label=below:$y'$,inner sep=0pt,minimum size=.175cm] (y'2) at (9,0) {{}};

\path (x2) edge[->] node {{}} (y2); 
\path (x'2) edge[->] node {{}} (y'2); 

\node at (4.5,1) {{$\Rightarrow$}};

\path (x2) edge[bend left,dotted,->] node {{}} (x'2); 
\path (y2) edge[bend left,dotted,->] node {{}} (y'2); 

\end{tikzpicture}
\end{center}
\caption{The \SqBack{} condition of p-morphisms, where solid arrows represent the relation $\sqsubseteq$ and dotted arrows represent the function $h$.}\label{p-morphismFig}
\end{figure}

We can now state a choice-free duality theorem for BAs.

\begin{theorem}[\cite{Holliday2018}]\label{FiltPossDuality} \textnormal{(ZF) The category \textbf{FiltPoss} of filter-descriptive possibility frames with p-morphisms is dually equivalent to the category \textbf{BA} of Boolean algebras with Boolean homomorphisms.}
\end{theorem}

As an example of how possibility frames can be used to prove facts about BAs, consider the following basic fact in Proposition \ref{ChainsProp}, the proof of which is almost copied verbatim from our proof using upper Vietoris spaces in \cite{BH2020} (note that the filter realization property in Proposition \ref{SepReal} is not needed in the proof). Recall that the Axiom of Dependent Choice (DC) states that if a binary relation $R$ on a set $X$ is serial, i.e., for every $x\in X$ there is a $y\in X$ with $xRy$, then there is an infinite sequence $x_0,x_1,\dots$ of elements of $X$ such that $x_iR_{i+1}$ for each $i\in\mathbb{N}$.

\begin{proposition}\label{ChainsProp} \textnormal{(ZF+DC) Every infinite BA has infinite chains and infinite anti-chains}.
\end{proposition}

\begin{proof} As the dual of any infinite BA is an infinite separative possibility frame, it suffices to show that for any infinite separative $\mathcal{F}=(S,\sqsubseteq, P)$, there is an infinite chain $U_0\supsetneq U_1\supsetneq\dots$ of sets from  $P$, as well as an infinite family of pairwise disjoint sets from $P$. For this it suffices to show that for every infinite  $U\in P$ (note that $S$ is such a $U$), there is an infinite $U'\in P$ with $U\supsetneq U'$ and $U\cap \neg U'\neq\varnothing$. For then by DC, there is an infinite chain $U_0\supsetneq U_1\supsetneq \dots$ of  sets from $P$ with $U_i\cap \neg U_{i+1}\neq\varnothing$ for each $i\in\mathbb{N}$, in which case $\{U_0\cap \neg U_1,U_1\cap \neg U_2,\dots\}$ is our antichain.

Assume $U\in P$ is infinite. Since $\sqsubseteq$ antisymmetric, take $x,y\in U$ such that $y\not\sqsubseteq x$. Then since $\mathcal{F}$ is separative, there is a $V\in P$ such that $x\in V$ and $y\not\in V$, which with $y\in U$ and $U,V\in P$ implies that there is a $z\sqsubseteq y$ such that $z\in U\cap \neg V$. Since $U,V\in P$, we have $U\cap V,U\cap \neg V\in P$; and since  $z\in U\cap \neg V$ and $x\in U\cap V$, we have $z\in U\cap \neg (U\cap V)\neq \varnothing$ and $x\in U\cap \neg (U\cap \neg V)\neq \varnothing$. Thus, if $U\cap V$ is infinite, then we can set $U':=U\cap V$, and otherwise we claim that $U\cap\neg V$ is infinite, in which case we can set $U':=U\cap\neg V$. Let $\sqsubseteq'$ be the restriction of $\sqsubseteq$ to $U$. Since $(S,\sqsubseteq)$ is a separative poset by Lemma \ref{FiltDLem}.\ref{FiltDLem1}, $(U,\sqsubseteq')$ is a separative poset by Proposition \ref{SepProp}.\ref{SepProp2} and the fact that $U\in \mathcal{RO}(S,\sqsubseteq)$. Given $V\in\mathcal{RO}(S,\sqsubseteq)$, we have  $U\cap V,U\cap \neg V\in \mathcal{RO}(U,\sqsubseteq')$ and $U\cap \neg V=\neg'(U\cap V)$, where $\neg'$ is the negation operation in $\mathcal{RO}(U,\sqsubseteq')$. Then since $U$ is infinite, by Proposition \ref{SepProp}.\ref{SepProp4} either $U\cap V$ or $\neg'(U\cap V)$ is infinite, as desired.\end{proof}

\subsection{Spaces of possibilities and Boolean algebras}\label{TopFrameSection}

Dualities between categories of algebras and categories of topological structures play a central role in the study of modal and nonclassical logics (see, e.g., \cite{Esakia2017}). Along these lines, in this section we ``topologize'' the possibility frames of \S~\ref{ArbitraryBooleanSection1} by using the family $P$ of admissible sets to generate a topology on $S$. This is analogous to the way that general frames in modal logic give rise to topological spaces (see \cite{Sambin1988}). We first briefly consider topologizing arbitrary possibility frames in \S~\ref{TopFrameSection1} and then filter-descriptive possibility frames in particular in \S~\ref{TopFrameSection2}. \S~\ref{TopFrameSection1} is based on \cite{BH2018}, which introduces relational topological possibility frames for modal logic.

\subsubsection{Topological possibility frames}\label{TopFrameSection1}

To describe the spaces that arise from topologizing a possibility frame, we need the following variation on the standard topological notion of a clopen set. 

\begin{definition}\label{NegClopDef} Let $(S,\sqsubseteq,\tau)$ be a triple such that $(S,\sqsubseteq)$ is a poset and $(S,\tau)$ is a topological space. For $U\subseteq S$:
\begin{enumerate}
\item $U\subseteq S$ is \textit{open} in $(S,\sqsubseteq,\tau)$ if $U$ is open in $(S,\tau)$; 
\item $U$ is \textit{neg-closed} in $(S,\sqsubseteq,\tau)$ if $\neg U=\{x\in S\mid \forall x'\sqsubseteq x \; x'\not\in U\}$ is open in~$(S,\tau)$; \item $U$ is \textit{neg-clopen} in $(S,\sqsubseteq,\tau)$ if $U$ is both open and neg-closed in $(S,\sqsubseteq,\tau)$;
\item $\mathsf{NegClop}(S,\sqsubseteq,\tau)$ is the family of neg-clopen sets;
\item $\mathsf{NegClop}\mathcal{RO}(S,\sqsubseteq,\tau)=\mathsf{NegClop}(S,\sqsubseteq,\tau)\cap \mathcal{RO}(S,\sqsubseteq)$.
\end{enumerate}
\end{definition}

\begin{definition} A \textit{topological possibility frame} is a triple $\mathcal{T}=(S,\sqsubseteq,\tau)$ such that $(S,\sqsubseteq)$ is a poset, $(S,\tau)$ is a topological space, and $\mathsf{NegClop}\mathcal{RO}(S,\sqsubseteq,\tau)$ is closed under intersection and forms a basis for $\tau$.
\end{definition}

Possibility frames may be turned into topological possibility frames, and vice versa, as follows. 

\begin{proposition}\label{TopPossProp} $\,$\textnormal{
\begin{enumerate}
\item\label{TopPossProp1} For any possibility frame $\mathcal{F}=(S,\sqsubseteq, P)$, the triple
$T(\mathcal{F})={(S,\sqsubseteq, \tau)}$
where $\tau$ is the topology generated by taking the elements of $P$ as basic opens is a topological possibility frame. 
\item\label{TopPossProp2} For any topological possibility frame $\mathcal{T}=(S,\sqsubseteq,\tau)$, the triple \\
$F(\mathcal{T})={(S,\sqsubseteq, P)}$ where $P=\mathsf{NegClop}\mathcal{RO}(S,\sqsubseteq,\tau)$ is a possibility frame.
\item\label{TopPossProp3} If $\mathcal{F}=(S,\sqsubseteq, P)$ is a separative possibility frame, then $\sqsubseteq$ is the converse of the specialization order of $(S,\tau)$ in $T(\mathcal{F})$.\footnote{In \cite{BH2018} we define regular opens of a poset to be regular opens in the \textit{upset} topology on the poset, in which case $\sqsubseteq$ coincides with the specialization order of the topology generated by $P$.}
\item\label{TopPossProp4} If $\mathcal{F}$ is a possibility frame satisfying filter realization, then $F(T(\mathcal{F}))=\mathcal{F}$.
\item\label{TopPossProp5} For any topological possibility frame $\mathcal{T}$, $T(F(\mathcal{T}))=\mathcal{T}$.
\end{enumerate}}
\end{proposition}

\begin{proof} Parts \ref{TopPossProp1}, \ref{TopPossProp2}, \ref{TopPossProp3}, and \ref{TopPossProp5} are easy to check from the definitions.

For part \ref{TopPossProp4}, we must show $U\in P$ iff $U\in \mathsf{NegClop}\mathcal{RO}(S,\sqsubseteq,\tau)$, where $\tau$ is the topology generated by $P$. The left-to-right direction is straightforward. From right to left, suppose $U\in \mathsf{NegClop}\mathcal{RO}(S,\sqsubseteq,\tau)$. Then $U$ is open in $(S,\tau)$, so $U=\bigcup\{A\in P\mid A\subseteq U\}$, and $\neg U$ is open in $(S,\tau)$, so $\neg U=\bigcup\{B\in P\mid B\subseteq \neg U\}$. Since $U,\neg U\in \mathcal{RO}(S,\sqsubseteq)$, it follows that $U=\bigvee \{A\in P\mid A\subseteq U\}$ and $V=\bigvee\{B\in P\mid B\subseteq \neg U\}$, and then since $U\vee\neg U=S$, we have $S=\bigvee \{A\in P\mid A\subseteq U\}\vee \bigvee\{B\in P\mid B\subseteq \neg U\}$. It follows by Lemma \ref{FiltDLem}.\ref{FiltDLem2} that $S=A_1\vee\dots\vee A_n\vee B_1\vee\dots\vee B_m$ where $A_i\subseteq U$ and $B_i\subseteq \neg U$. Hence $A_1\cup\dots\cup A_n\subseteq U$, which with $U\in\mathcal{RO}(S,\sqsubseteq)$ implies $A_1\vee\dots\vee A_n\subseteq U$, and  $U\cap (B_1\cup\dots \cup B_m)=\varnothing$, which implies $U\wedge (B_1\vee\dots \vee B_m)=\varnothing$. Then:
\begin{eqnarray*}
U&=&U\wedge S = U\wedge (A_1\vee\dots\vee A_n\vee B_1\vee\dots\vee B_m) \\
&=& (U\wedge ( A_1\vee\dots\vee A_n))\vee (U\wedge (B_1\vee\dots\vee B_m))\\
&=& U\wedge ( A_1\vee\dots\vee A_n).
\end{eqnarray*}
Since $\wedge $ is intersection, $U= U\wedge ( A_1\vee\dots\vee A_n)$ implies $U\subseteq A_1\vee\dots\vee A_n$, and we already derived $A_1\vee\dots\vee A_n\subseteq U$ above. Thus, $U=A_1\vee\dots\vee A_n\in P$.\end{proof}

In light of part \ref{TopPossProp3}, there is little harm in dropping $\sqsubseteq$ from the signature of our structures and simply working with spaces, as we will do in the next section.

\subsubsection{UV-spaces}\label{TopFrameSection2}

In this section, we characterize the topological spaces corresponding (via Proposition \ref{TopPossProp}.\ref{TopPossProp1}) to the \textit{filter-descriptive} possibility frames of Proposition \ref{SepReal}.

Recall that Stone \cite{Stone1938} proved that every distributive lattice $L$ can be represented as the lattice of compact open sets of a spectral space, namely the spectral space $S(L)$ of prime filters of $L$ topologized with basic open sets of the form 
\[\beta(a)=\{F\in\mathsf{PrimeFilt}(L)\mid a\in F\}\] for $a\in L$. Recall that a \textit{spectral space} is a $T_0$ space $X$ such that:
\begin{enumerate}
\item $X$ is \textit{coherent}: $\mathsf{CO}(X)$, the family of compact open sets of $X$, is closed under finite intersections and forms a basis for the topology of $X$;
\item $X$ has \textit{enough points}: every completely prime filter in $\mathsf{O}(X)$, the lattice of open sets of $X$, is of the form $\mathsf{O}(x)=\{U\in \mathsf{O}(X)\mid x\in U\}$ for some $x\in X$.\footnote{Equivalently, a spectral space is a space that is coherent and \textit{sober} \cite[\S~II.1.6]{Johnstone1982}, where $X$ is sober if the map from $X$ to the set of completely prime filters in $\mathsf{O}(X)$ is not only a surjection but a bijection. Sobriety is equivalent to the conjunction of $T_0$ and having enough points.}
\end{enumerate}
A special case of this representation is Stone's \cite{Stone1936} representation of Boolean algebras: if $L$ is a BA, then the compact open sets are precisely the clopen sets, and $S(L)$ is a zero-dimensional compact Hausdorff space or \textit{Stone space} \cite{Johnstone1982}.

The use of prime filters in Stone's theorems requires the nonconstructive Prime Filter Theorem. However, as shown in \cite{BH2020}, a choice-free representation of BAs is possible using the following special spectral spaces.

\begin{definition}\label{UVDef} An \textit{upper Vietoris space} (\textit{UV-space}) is a $T_0$ space $X$ such that:
\begin{enumerate}
\item\label{UVDef1} $\mathsf{CRO}(X)$, the family of compact regular open sets of $X$, forms a basis and is closed under finite intersections and the operation $U\mapsto \mathsf{int}(X\setminus U )$;
\item\label{UVDef2} every proper filter in $\mathsf{CRO}(X)$ is of the form $\mathsf{CRO}(x)=\{U\in \mathsf{CRO}(X)\mid {x\in U}\}$ for some $x\in X$.
\end{enumerate}
\end{definition}
\noindent Note the similarity in logical form between conditions 1 and 2 of a UV-space and the conditions of coherence and having enough points of a spectral space. The choice of the name ``upper Vietoris space'' is explained in~\cite{BH2020}. In brief, it comes from the ``upper Vietoris'' variant of the Vietoris hyperspace construction \cite{Vietoris1922}. Assuming the Prime Filter Theorem, every upper Vietoris space as in Definition \ref{UVDef} is isomorphic to the upper Vietoris hyperspace of a Stone space.

\begin{proposition}[\cite{BH2020}]\label{UVProp} $\,$\textnormal{
\begin{enumerate}
\item Every UV-space is a spectral space.
\item For any UV-space, $\mathsf{CRO}(X)$ forms a Boolean algebra with $\cap$ as meet and $U\mapsto \mathsf{int}(X\setminus U )$ as complement.
\item\label{UVProp3} For any Boolean algebra $B$, the space $UV(B)$ whose underlying set is $\mathsf{PropFilt}(B)$ and whose topology is generated by $\{\widehat{a}\mid a\in B\}$, with $\widehat{a}$ defined as in Proposition \ref{FiltFrameDef}, is a UV-space. 
\item A space $X$ is homeomorphic to $UV(\mathsf{CRO}(X))$ iff $X$ is a UV-space.
\end{enumerate}}
\end{proposition} 

\begin{remark}Though we defined UV-spaces above using the family $\mathsf{CRO}(X)$, we could have equivalently done so using the family $\mathsf{CO}\mathcal{RO}(X)$ of sets (as in \cite{BH2020}) that are both compact open in $X$ and regular open in the poset $(X,\sqsubseteq)$, where $\sqsubseteq$ is the converse of the specialization order of $X$,\footnote{In \cite{BH2020} we define regular opens of a poset to be regular opens in the \textit{upset} topology on the poset, in which case $\sqsubseteq$ coincides with the specialization order of $X$.} or using the family $\mathsf{NegClop}\mathcal{RO}(X)$ (as in \cite{BH2018}) defined as in Definition \ref{NegClopDef} in terms of $\sqsubseteq$ and the topology of $X$.
\end{remark}

\begin{proposition}[\cite{BH2020,BH2018}]\label{3UV} For any UV-space $X$,
\[\mathsf{CRO}(X)=\mathsf{CO}\mathcal{RO}(X)=\mathsf{NegClop}\mathcal{RO}(X).\]
\end{proposition}

UV-spaces and filter-descriptive possibility frames as in Definition \ref{SepReal} are in one-to-one correspondence via analogues of the $T$ and $F$ maps of Definition \ref{TopPossProp}.

\begin{proposition}\textnormal{For any UV-space $X$, if $S$ is the underlying set of $X$ and $\sqsubseteq$ the converse of the specialization order of $X$, then $(S,\sqsubseteq,\mathsf{CRO}(X))$ is a filter-descriptive possibility frame. For any filter-descriptive possibility frame $(S,\sqsubseteq, P)$, the space $(S,\tau)$ where $\tau$ is the topology generated by $P$ is a UV-space.}
\end{proposition}

\begin{proof} For the first claim, $(S,\sqsubseteq)$ is a poset, and $\mathsf{CRO}(X)$ is a nonempty subset of $\mathcal{RO}(S,\sqsubseteq)$ closed under $\neg$ and $\cap$, as $\mathsf{CRO}(X)=\mathsf{CO}\mathcal{RO}(X)$ by Proposition \ref{3UV}. That $\mathsf{CO}\mathcal{RO}(X)$ forms a basis for $X$, given by part \ref{UVDef1} of Definition \ref{UVDef}, implies the separation property of filter-descriptive possibility frames in Proposition \ref{SepReal}, and part \ref{UVDef2} of Definition \ref{UVDef} implies the filter realization property. The second claim follows from Propositions \ref{SepReal} and \ref{UVProp}.\ref{UVProp3}.\end{proof}

To go beyond representation of BAs to a full duality for BAs, we introduce morphisms similar to those in Definition \ref{PossPMorph}. A \textit{spectral map} \cite{Hochster1969} between spectral spaces $X$ and $X'$ is a map $f\colon X\to X'$ such that $f^{-1}[U]\in\mathsf{CO}(X)$ for each $U\in\mathsf{CO}(X')$, which implies that $f$ is continuous.

\begin{definition} Let $X$ and $X'$ be UV-spaces with specialization orders $\leqslant$ and $\leqslant'$, respectively. A \textit{UV-map} from $X$ to $X'$ is a spectral map $f\colon X\to X'$ that also satisfies the \textit{p-morphism condition}:
\[\mbox{if $f(x) \leqslant' y'$, then $\exists y: x\leqslant y$ and $f(y)=y'$}.\]
\end{definition}
\noindent This condition matches the p-morphic \SqBack{} condition of Definition \ref{PossPMorph}.\ref{PossPMorph2} after adjusting for the fact that $\leqslant$ is the converse of $\sqsubseteq$. That $f$ preserves the specialization order follows from $f$ being continuous.

We can now state a duality between UV-spaces and BAs analogous to the duality between filter-descriptive possibility frames and BAs (Theorem \ref{FiltPossDuality}).

\begin{theorem}[\cite{BH2020}] \textnormal{(ZF) The category \textbf{UV} of UV-spaces with UV-maps is dually equivalent to the category \textbf{BA} of Boolean algebras with Boolean homomorphisms.}
\end{theorem}

This duality leads to a dictionary for translating between BA concepts and {UV-concepts}, given in \cite{BH2020}, making it possible to prove facts about BAs using point-set topological reasoning. Here we highlight only one entry in the dictionary. For this we stress the distinction between $\mathsf{RO}(X)$, the family of regular open sets in the UV-space $X$, and $\mathcal{RO}(X)$, the family of regular open sets in the poset $(X,\sqsubseteq)$ where $\sqsubseteq$ is the converse of the specialization order of $X$. While $\mathcal{RO}(X)$ realizes the canonical extension of the BA dual to $X$ (recall Proposition \ref{FiltFrameDef}.\ref{FiltFrameDef2}), $\mathsf{RO}(X)$ realizes the MacNeille completion of the BA dual to $X$.

\begin{proposition}[\cite{BH2020}]\label{MacNeilleProp2} \textnormal{Let $B$ be a BA and $X$ its dual UV-space. Then $\mathsf{RO}(X)$ is (up to isomorphism) the MacNeille completion of $B$.}
\end{proposition}
\noindent Thus, we have seen two ways of realizing the MacNeille completion of a BA $A$: $\mathcal{RO}(B_+,\leq_+)$, as in Proposition \ref{MacNeilleProp}, or $\mathsf{RO}(UV(B))$, as in Proposition \ref{MacNeilleProp2}. 

\begin{remark}\label{MacNeilleCanonicalComparison} Compare the choice-free possibilistic approaches to realizing the MacNeille completion and constructive canonical extension of a BA to the classical methods based on Stone duality, assuming the Boolean Prime Filter Theorem: the MacNeille completion of a BA $B$ can be realized as $\mathsf{RO}(X)$ where $X$ is the Stone space of $B$; and the canonical extension of a BA $B$ can be realized as the powerset of the Stone space of $B$. Figure \ref{MacNeilleCanonical} summarizes the foregoing points.
\end{remark}

\begin{figure}
\begin{tabular}{l|l|l}
& world approach & possibility approach \\
\hline
MacNeille  of $B$ & regular opens of Stone space & $\mathcal{RO}(B_+,\leq_+)$ or regular    \\
&&  opens of UV-space \\
canonical extension & powerset of Stone space  & regular opens of poset \\
of $B$ && of proper filters  \\
\end{tabular}
\caption{realizations of MacNeille completion and canonical extension}\label{MacNeilleCanonical}
\end{figure}

\section{First-order case}\label{FOSection}

In this section, we give possibility semantics for classical first-order logic, a topic previously studied by van Benthem \cite{Benthem1981a,Benthem2016b}, Harrison-Trainor \cite{HT2019}, and Massas \cite{Massas2016}.\footnote{These versions of possibility semantics for first-order logic differ in their details. Van Benthem and Massas consider first-order languages without identity or function symbols, and their models are ``varying domain'' models, whereas our models below are ``constant domain'' models, like Harrison-Trainor's. Our models differ from the others in adding ``equality'' relations.} Fine \cite{Fine1975} also gave what can be considered a kind of possibility semantics for classical first-order logic with applications to vague languages (see \cite[\S~8.1]{Holliday2018} for comparison). As we briefly show in \S~\ref{ForcingSection}, first-order possibility semantics can be seen as generalizing an interpretation of the language of  first-order set theory used in forcing (see, e.g., \cite{Bell2005}) to arbitrary first-order languages. In the terminology of topos theory, first-order possibility semantics can be seen as a variant of sheaf semantics for first-order logic using the Cohen topos \cite[p.~318]{MacLane1992}.

After defining the semantics, our main goal is a completeness theorem. The traditional completeness theorem for first-order logic for uncountable languages \cite{Malcev1936}, stating that every consistent set of first-order sentences has a \textit{Tarskian model}, is not provable in ZF, as it is equivalent in ZF to the Boolean Prime Filter Theorem \cite{Henkin1954} (for a proof, see \cite[p.~104]{Bell1974}). By contrast, we will prove in ZF that for arbitrary languages, every consistent set of first-order sentences has a \textit{possibility model}. This suggest the possibility of a program of ``choice-free model theory.''

We assume familiarity with the definition of a first-order language $L$, including identity and function symbols. The only detail to note is that we treat constant symbols as $0$-ary function symbols for simplicity (as in, e.g., \cite[\S~2.4]{Shoenfield2010}). Let $\mathsf{Var}(L)$ and $\mathsf{Term}(L)$ be the set of variables and set of terms of $L$, respectively. We also assume familiarity with first-order logic as a proof system. 

\subsection{First-order possibility models}

We begin with possibility models for a first-order language $L$. 

In addition to a poset $(S,\sqsubseteq)$ of possibilities, these models have a set $D$ of what we call \textit{guises of objects}. A possibility $s\in S$ may settle that two guises $a$ and $b$ are guises of the same object: $a\asymp _s b$.\footnote{The use of a family of equivalence relations $\{\asymp_s\}_{s\in S}$ appears in the standard definition of \textit{Kripke models with equality} in \cite[\S~3.5]{Gabbay2009}.} On the other hand, a partial possibility $s$ may not yet settle that they are guises of the same object; borrowing Frege's \cite{Frege1892} famous example, $s$ may not settle that the object known under the guise of \textit{The Morning Star} is the same as the object known under the guise of \textit{The Evening Star}. Then by the standard refinability property in possibility semantics (see Definition \ref{FOmodel}.\ref{FOmodel3}), there will be a possibility $s'\sqsubseteq s$ that settles that they are \textit{not} guises of the same object: for all $s''\sqsubseteq s'$, $a\not\asymp_{s''}b$. Note the distinction between $a\not\asymp_{s'}b$, which means $s'$ does not settle that $a$ and $b$ are guises of the same object, and the stronger condition that for all $s''\sqsubseteq s'$, $a\not\asymp_{s''}b$, which means $s'$ settles that $a$ and $b$ are not guises of the same object. 

A similar distinction applies to the interpretation of predicate symbols by an interpretation function $V$: if $(a,b)\in V(R,s)$, then $s$ settles that the pair $(a,b)$ is in the interpretation of $R$, which we take to mean that the objects of which $a$ and $b$ are guises, respectively, stand in the relation to which $R$ corresponds; if $(a,b)\not\in V(R,s)$, then $s$ does not settle that $(a,b)$ is in the interpretation of  $R$; and if for all $s'\sqsubseteq s$, $(a,b)\not\in V(R,s')$, then $s$ settles that $(a,b)$ is \textit{not} in the interpretation of  $R$. We comment on the interpretation of function symbols below.

\begin{definition}\label{FOmodel} Let $L$ be a first-order language. A \textit{first-order possibility model for $L$} is a tuple $\mathfrak{A}=(S,\sqsubseteq, D, \asymp, V)$ where:
\begin{enumerate}
\item $(S,\sqsubseteq)$ is a poset;
\item $D$ is a nonempty set;
\item\label{FOmodel3} $\asymp$ is a function assigning to each $s\in S$ an equivalence relation $\asymp_s$ on $D$ satisfying:

\item[$\bullet$] \textit{persistence} for $\asymp$: if $a\asymp_{s}b$ and $s'\sqsubseteq s$, then $a\asymp_{s'}b$;
\item[$\bullet$] \textit{refinability} for $\asymp$: if $a\not\asymp_{s}b$, then $\exists s'\sqsubseteq s$ $\forall s''\sqsubseteq s'$ $a\not\asymp_{s''}b$.

\item $V$ is a function assigning to each pair of an $n$-ary predicate $R$ of $L$ and $s\in S$ a set $V(R,s)\subseteq D^n$ and to each $n$-ary function symbol $f$ of $L$ and $s\in S$ a set $V(f,s)\subseteq D^{n+1}$ satisfying:

\item[$\bullet$] \textit{persistence} for $R$: if $\overline{a}\in V(R,s)$, $s'\sqsubseteq s$, and $\overline{a}\asymp_{s'}\overline{b}$, then $\overline{b}\in V(R,s')$;\footnote{Here $\overline{a}\asymp_{s'}\overline{b}$ means that $a_1\asymp_{s'}b_1,\dots,a_n\asymp_{s'}b_n$.}
\item[$\bullet$] \textit{refinability} for $R$: if $\overline{a}\not\in V(R,s)$, then $\exists s'\sqsubseteq s$ $\forall s''\sqsubseteq s'$ $\overline{a}\not\in V(R,s'')$;
\item[$\bullet$] \textit{persistence} for $f$: if $\overline{a}\in V(f,s)$, $s'\sqsubseteq s$, and $\overline{a}\asymp_{s'}\overline{b}$, then $\overline{b}\in V(f,s')$;\
\item[$\bullet$] \textit{quasi-functionality} for $f$: if $(\overline{a},b)\in V(f,s)$ and $(\overline{a},b')\in V(f,s)$, then $b\asymp_s b'$;
\item[$\bullet$] \textit{eventual definedness} for $f$:  $\forall \overline{a}\in D^n$  $\exists s'\sqsubseteq s$ $\exists b\in D$: $(\overline{a},b)\in V(f,s)$.
\end{enumerate}

We say $\mathfrak{A}$ \textit{has total functions} if for each $n$-ary function symbol $f$ of $L$, $s\in S$, and $\overline{a}\in D^n$, there is some $b\in D$ such that $(\overline{a},b)\in V(f,s)$.

A \textit{Tarskian model} is a first-order possibility model with total functions in which $S$ contains only one possibility $s$, and $\asymp_s$ is the identity relation.

A \textit{pointed model} is a pair $\mathfrak{A},s$ of a possibility model $\mathfrak{A}$ and possibility $s$ in $\mathfrak{A}$.\end{definition}

\begin{example} The following example of a non-Tarskian first-order possibility model comes from the development of choice-free nonstandard analysis using possibility models by Massas \cite{Massas2020b}. Let $\mathfrak{M}$ be a Tarskian model with domain $M$. The Fr\'{e}chet power $\mathfrak{M}_\mathscr{F}$ of $\mathfrak{M}$ is the tuple $\mathfrak{M}_\mathscr{F} = (\mathscr{F},\supseteq, D, \asymp, V)$ where:
\begin{enumerate}
\item $\mathscr{F}$ is the set of all proper filters of $\wp(\omega)$ extending the Fr\'{e}chet filter (i.e., the set of all cofinite subsets of $\omega$), which we order by reverse inclusion $\supseteq$;
\item $D=M^\omega$, i.e., the set of all functions $f:\omega\to M$;
\item for $a,b\in D$ and $F\in\mathfrak{F}$, $a\asymp_F b$ iff $\{i\in \omega\mid a(i)=b(i)\}\in F$;
\item for any $n$-ary predicate $R$ of $L$, $\overline{a}\in D^n$, and $F\in\mathfrak{F}$, $\overline{a}\in V(R,F)$ iff $\{i\in \omega\mid M\models_{g[\overline{x}:=\overline{a(i)}]} R(\overline{x})\}\in F$, where $\models _{g[\overline{x}:=\overline{a(i)}]}$ is the standard Tarskian satisfaction relation with the variables $\overline{x}$ mapped to $\overline{a(i)}$.
\end{enumerate} 
Then one can show that $\mathfrak{M}_\mathscr{F}$ is a first-order possibility model.
\end{example}

Let us note two aspects of our interpretation of an $n$-ary function symbol $f$. First, in traditional semantics, $f$ is interpreted as an $n$-ary function on the domain of objects, i.e.,  an $n+1$-ary  relation such that if $(\overline{a},b)$ and $(\overline{a},b')$ are both in the relation, then $b=b'$. By contrast, we ask only that $b$ and $b'$ be guises of the same object, i.e., $b\asymp_s b'$, not that they are one and the same guise. For $(\overline{a},b)\in V(f,s)$ means that when the function to which $f$ corresponds takes in objects of which $\overline{a}$ are guises, then it outputs an object of which $b$ is a guise. Hence if $b\asymp_s b'$, we should have $(\overline{a},b')\in V(f,s)$ as well. Thus, we do not interpret $f$ at $s$ as an $n$-ary function on the domain of guises. Instead, our approach is equivalent to interpreting $f$ at $s$ as an $n$-ary function on $D/\mathord{\asymp_s}$,\footnote{There is a mathematically viable alternative version of the semantics that interprets $f$ at $s$ as an $n$-ary function on $D$, which for a given input $\overline{a}\in D^n$ chooses a unique output $b$ from the equivalence class $\{b'\in D\mid b\asymp_s b'\}$. Although directly transforming one of our quasi-functional models to such a properly functional model by choosing representatives of equivalence classes may require the Axiom of Choice, the canonical model construction below can be adapted to show that properly functional models would suffice for completeness. A possible convenience of the properly functional approach is that one can take the denotation of a term to be a single element of $D$ (provided one then takes the truth clause for $t_1=t_2$ at $s$ to be that the denotations of $t_1$ and $t_2$ stand in the $\asymp_s$ relation), rather than an equivalence class of elements of $D$, as we do in Definition \ref{DenoteDef}. On the other hand, an attraction of the quasi-functional approach is that it treats relations and function in a uniform manner. Consider the binary relation symbol $Square$ such that $Square(t_1,t_2)$ means that $t_2$ is the square of $t_1$ vs.~the unary function symbol $square$ such that $square(t_1)=t_2$ means that $t_2$ is the square of $t_1$. If $(a,b)$ is in the interpretation of the $Square$ relation at $s$, and $b\asymp_s b'$, then $(a,b')$ should be in the interpretation of $Square$ at $s$, since we take $(a,b)\in V(Square,s)$ (resp.~$(a,b')\in V(Square,s)$) to mean that \textit{the objects} of which $a,b$ (resp.~$a,b'$) are guises stand in the relevant relation, and $b$ and $b'$ are guises of the same object. The quasi-functional approach applies the same idea to the interpretation of $square$.} as well as interpreting an $n$-ary relation symbol $R$ at $s$ as an $n$-ary relation on $D/\mathord{\asymp_s}$ (see Definition \ref{InterpretEquiv}). However, our conditions on $V$ are somewhat simpler to state when officially taking the interpretations of symbols at $s$ to be relations on~$D$.

Second, we have allowed the ``quasi-function'' interpreting a function symbol at a given partial possibility to be undefined for a given argument, i.e., there is some $\overline{a}\in D^n$ such that for no $b\in D$ do we have $(\overline{a},b)\in V(f,s)$, as the possibility $s$ has not yet settled any guise of the output of the function. Allowing such undefinedness is not strictly necessary. We will see that (in ZF) every pointed model has the same first-order theory as a pointed model with total functions. However, allowing partiality of functions may be convenient in constructing specific models---though functions must eventually become defined in order to validate theorems such that $\forall x\exists y f(x)=y$. Of course, it will turn out that in ZF + Boolean Prime Filter Theorem, every pointed model has the same first-order theory as a Tarskian model with no partiality at all, thanks to the soundness of first-order logic with respect to possibility models and the standard completeness theorem for first-order logic (using the Prime Filter Theorem) with respect to Tarskian models.

Next we define the denotation of a term relative to a possibility and variable assignment. The denotation of a term will be an equivalence class of $\asymp_s$ or $\varnothing$. 

\begin{definition}\label{DenoteDef} Given a first-order possibility model $\mathfrak{A}=(S,\sqsubseteq, D, \asymp, V)$, ${s\in S}$, and variable assignment $g:\mathsf{Var}(L)\to D$, we define a function \\ $\llbracket \;\rrbracket_{\mathfrak{A},s,g}:{\mathsf{Term}(L)\to \wp(D)}$ recursively as follows:
\begin{enumerate}
\item\label{DenoteDef1} $\llbracket x\rrbracket_{\mathfrak{A},s,g}=\{a\in D\mid a\asymp_s g(x)\}$ for $x\in\mathsf{Var}(L)$;
\item\label{DenoteDef2} for an $n$-ary function symbol $f$ and $t_1,\dots,t_n\in \mathsf{Term}(L)$,
\begin{eqnarray*}\llbracket f(t_1,\dots,t_n)\rrbracket_{\mathfrak{A},s,g}&=& \{b \in D\mid \exists a_1,\dots,a_n: a_i\in \llbracket t_i\rrbracket_{\mathfrak{A},s,g}\mbox{ and }\\
&& \qquad\qquad\qquad\qquad\quad\, (a_1,\dots,a_n,b)\in V(f,s)\}.
\end{eqnarray*}
\end{enumerate}
\end{definition}

Given terms $t,u$ and a variable $x$, let $u^x_t$ be the result of replacing $x$ in $u$ by $t$. Given a variable assignment $g$, variable $x$, and $a\in D$, let $g[x:=a]$ be the variable assignment that differs from $g$ at most in assigning $a$ to $x$. 

\begin{lemma}\label{TermEquiv} \textnormal{For any model $\mathfrak{A}=(S,\sqsubseteq, D, \asymp, V)$, $s\in S$, variable assignment $g:\mathsf{Var}(L)\to D$,   $t,u\in\mathsf{Term}(L)$, and $x\in\mathsf{Var}(L)$:
\begin{enumerate}
\item\label{TermEquiv1} $\llbracket t\rrbracket_{\mathfrak{A},s,g}$ is an $\asymp_s$-equivalence class or $\varnothing$;
\item\label{TermEquiv1.25} if $g(x)\asymp_s a$, then $\llbracket t\rrbracket_{\mathfrak{A},s,g}=\llbracket t\rrbracket_{\mathfrak{A},s,g[x:=a]}$;
\item\label{TermEquiv2} if $a\in \llbracket t\rrbracket_{\mathfrak{A},s,g}$ and $s'\sqsubseteq s$, then $a\in \llbracket t\rrbracket_{\mathfrak{A},s',g}$;
\item\label{TermEquiv1.5} there is some $s'\sqsubseteq s$ such that $\llbracket t\rrbracket_{\mathfrak{A},s',g}\neq\varnothing$;
\item\label{TermEquiv2.5} if $\mathfrak{A}$ has total functions, then $\llbracket t\rrbracket_{\mathfrak{A},s,g}\neq\varnothing$;
\item\label{TermEquiv3} $\llbracket u^x_t\rrbracket_{\mathfrak{A},s,g}=\llbracket u\rrbracket_{\mathfrak{A},s,g[x:=a]}$ for any $a\in \llbracket t\rrbracket_{\mathfrak{A},s,g}$.
\end{enumerate}}
\end{lemma}

\begin{proof} For part \ref{TermEquiv1}, the case for a variable $x$ is immediate from Definition \ref{DenoteDef}.\ref{DenoteDef1}. For terms with function symbols, assuming $b\in \llbracket f(t_1,\dots,t_n)\rrbracket_{\mathfrak{A},s,g}$, we claim that ${b\asymp_s b'}$ iff $b'\in \llbracket f(t_1,\dots,t_n)\rrbracket_{\mathfrak{A},s,g}$. The left-to-right direction follows from \textit{persistence} for $f$, while the right-to-left direction follows from \textit{quasi-functionality} for $f$. Part \ref{TermEquiv1.25} is an obvious induction on $t$.  Part \ref{TermEquiv2} follows from part \ref{TermEquiv1} and \textit{persistence} for $\asymp$ and function symbols. Part \ref{TermEquiv1.5} follows from \textit{eventual definedness} for function symbols. Parts \ref{TermEquiv2.5} and \ref{TermEquiv3} are straightforward inductions on $t$ and $u$, respectively.\end{proof}

We can also view the interpretation of a relation symbol at a possibility $s$ as a relation on the set of $\asymp_s$-equivalence classes, as follows.

\begin{definition}\label{InterpretEquiv} Given a first-order possibility model $\mathfrak{A}=(S,\sqsubseteq, D, \asymp, V)$, define a function $I_\asymp$ that assigns to each pair of an $n$-ary predicate $R$ of $L$ and $s\in S$ a set $I_\asymp(R,s)\subseteq (D/\mathord{\asymp_s})^n$ by
\[(\xi_1,\dots,\xi_n )\in I_\asymp(R,s)\mbox{ iff }\exists a_1,\dots,a_n: a_i\in \xi_i\mbox{ and }(a_1,\dots,a_n)\in V(R,s).\]
\end{definition}

Finally, we define the satisfaction of a formula by a possibility and variable assignment.

\begin{definition}\label{FOPossSat} Given a first-order possibility model $\mathfrak{A}=(S,\sqsubseteq, D, \asymp, V)$ for $L$, formula $\varphi$ of $L$, $s\in S$, and variable assignment $g:\mathsf{Var}(L)\to D$, we define the satisfaction relation $\mathfrak{A},s\Vdash_g \varphi$ recursively as follows:
\begin{enumerate}
\item\label{FOPossSat1} $\mathfrak{A},s\Vdash_g t_1=t_2$ iff for all $s'\sqsubseteq s$, if each $\llbracket t_i\rrbracket_{\mathfrak{A},s',g}$ is nonempty,  then \\ $\llbracket t_1\rrbracket_{\mathfrak{A},s',g}=\llbracket t_2\rrbracket_{\mathfrak{A},s',g}$;
\item\label{FOPossSat2} $\mathfrak{A},s\Vdash_g R(t_1,\dots,t_n)$ iff for all $s'\sqsubseteq s$, if each $\llbracket t_i\rrbracket_{\mathfrak{A},s',g}$ is nonempty, then $(\llbracket t_1\rrbracket_{\mathfrak{A},s',g},\dots ,\llbracket t_n\rrbracket_{\mathfrak{A},s',g})\in I_\asymp(R,s')$;
\item\label{FOPossSat3} $\mathfrak{A},s\Vdash_g \neg\varphi$ iff for all $s'\sqsubseteq s$, $\mathfrak{A},s'\nvDash_g \varphi$;
\item $\mathfrak{A},s\Vdash_g \varphi\wedge\psi$ iff $\mathfrak{A},s\Vdash_g\varphi$ and $\mathfrak{A},s\Vdash_g\psi$;
\item\label{FOPossSatForall} $\mathfrak{A},s\Vdash_g \forall x \varphi$ iff for all $a\in D$, $\mathfrak{A},s\Vdash_{g[x:=a]} \varphi$.
\end{enumerate}
We then say that:
\begin{itemize}
\item a set $\Gamma$ of formulas is \textit{satisfiable in $\mathfrak{A}$} if there is some possibility $s$ in $\mathfrak{A}$ and variable assignment $g$ such that $\mathfrak{A},s\Vdash_g\varphi$ for all $\varphi\in\Gamma$;

\item if $\varphi$ is a sentence, i.e., a formula with no free variables, then  $\mathfrak{A},s\Vdash\varphi$ if $\mathfrak{A},s\Vdash_g\varphi$ for some (equivalently, every) variable assignment $g$;

\item two pointed models $\mathfrak{A},s$ and $\mathfrak{A}',s'$ are \textit{elementarily equivalent} if for all sentences $\varphi$, $\mathfrak{A},s\Vdash \varphi$ iff $\mathfrak{A}',s'\Vdash \varphi$.
\end{itemize}
\end{definition}

In possibility models with total functions, the definition of satisfaction simplifies as follows.

\begin{lemma} \textnormal{If $\mathfrak{A}=(S,\sqsubseteq, D, \asymp, V)$ has total functions, then for any $s\in S$ and variable assignment $g:\mathsf{Var}(L)\to D$:
\begin{enumerate}
\item $\mathfrak{A},s\Vdash_g t_1=t_2$ iff $\llbracket t_1\rrbracket_{\mathfrak{A},s,g}=\llbracket t_2\rrbracket_{\mathfrak{A},s,g}$;
\item $\mathfrak{A},s\Vdash_g R(t_1,\dots,t_n)$ iff  $(\llbracket t_1\rrbracket_{\mathfrak{A},s,g},\dots ,\llbracket t_n\rrbracket_{\mathfrak{A},s,g})\in I_\asymp(R,s)$.
\end{enumerate}}
\end{lemma}
\begin{proof} If $\mathfrak{A}$ has total functions, then $\llbracket t_i\rrbracket_{\mathfrak{A},s',g}$ is always nonempty by Lemma \ref{TermEquiv}.\ref{TermEquiv2.5}, so the nonemptiness condition drops out of Definition \ref{FOPossSat}.\ref{FOPossSat1}-\ref{FOPossSat2}, and the quantification over $s'\sqsubseteq s$ can be dropped thanks to \textit{persistence} for $\asymp$ and $R$ (Definition \ref{FOmodel}).\end{proof}

Moreover, the following lemma, which is a consequence of Lemma \ref{CanModel} and Theorem \ref{StrongComp} below, shows that we can work with possibility models with total functions without loss of generality. 

\begin{lemma} \textnormal{For every pointed first-order possibility model $\mathfrak{A},s$, there is a pointed first-order possibility model $\mathfrak{A}',s'$ with total functions such that $\mathfrak{A},s$ and $\mathfrak{A}',s'$ are elementarily equivalent.}
\end{lemma}

Finally, if our only concern is which formulas are satisfied at a particular possibility $s$ in a possibility model, then we can delete from the model all possibilities that are not refinements of $s$ (cf.~the fact that in modal logic, satisfaction is invariant under taking generated submodels \cite[Prop.~2.6]{Blackburn2001}).

\begin{lemma} \textnormal{Given a  first-order possibility model $\mathfrak{A}$ with $\mathfrak{A}=(S,\sqsubseteq, D, \asymp, V)$ and $s\in S$, let $\mathfrak{A}_s$ be the obvious restriction of $\mathfrak{A}$ to the set $\{s'\in S\mid s'\sqsubseteq s\}$ of possibilities, called the \textit{generated submodel of $\mathfrak{A}$ generated by $s$}. Then $\mathfrak{A},s$ and $\mathfrak{A}_s,s$ are elementarily equivalent.}
\end{lemma}

\begin{proof}[Proof sketch.] An easy induction on the structure of formulas $\varphi$, noting that the definition of truth for $\varphi$ at $s$ only quantifies over states in $\{s'\in S\mid s'\sqsubseteq s\}$.\end{proof}

\subsection{Soundness of first-order logic with respect to possibility semantics}

In this section, we work up to a first main result concerning first-order possibility models: the soundness of first-order logic with respect to the semantics.

To that end, the first key point is that the set of possibilities satisfying a given formula always belongs to the Boolean algebra of regular open subset of the underlying poset of the possibility model.\footnote{For the last part of Lemma \ref{FORO} concerning $\forall x\varphi$, cf.~the algebraic interpretation of quantifiers \cite{RS63,Scott2008}.}

\begin{lemma}\label{FORO} \textnormal{For any first-order possibility model $\mathfrak{A}=(S,\sqsubseteq, D, \asymp, V)$ for $L$,  variable assignment $g:\mathsf{Var}\to D$, and formula $\varphi$ of $L$, define \[\|\varphi \|_{\mathfrak{A},g} = \{s\in S\mid \mathfrak{A},s\Vdash_g\varphi\}.\] Then for all formulas $\varphi,\psi$ and variables $x$:
\begin{eqnarray*}
\| \varphi \|_{\mathfrak{A},g} &\in & \mathcal{RO}(S,\sqsubseteq)\\
\| \neg\varphi \|_{\mathfrak{A},g} &=& \neg \| \varphi \|_{\mathfrak{A},g}  \\
\| \varphi \wedge\psi \|_{\mathfrak{A},g} &=&  \| \varphi \|_{\mathfrak{A},g}\wedge \| \varphi \|_{\mathfrak{A},g}   \\
\| \forall x\varphi \|_{\mathfrak{A},g} &=&\bigwedge \{ \| \varphi \|_{\mathfrak{A},g[x:=a]} \mid a\in D\}.
\end{eqnarray*}}
\end{lemma}
\begin{proof} We prove only that $\| \varphi \|_{\mathfrak{A},g}\in \mathcal{RO}(S,\sqsubseteq)$, by induction on $\varphi$. Persistence for the semantic values of $t=t'$ and $R(t_1,\dots,t_n)$ is immediate from Definition \ref{FOPossSat}.

For refinability for $t_1=t_2$, suppose $\mathfrak{A},s\nvDash_g t_1=t_2$, so there is an $s'\sqsubseteq s$ such that  each $\llbracket t_i\rrbracket_{\mathfrak{A},s',g}$ is nonempty and  $\llbracket t_1\rrbracket_{\mathfrak{A},s',g}\neq\llbracket t_2\rrbracket_{\mathfrak{A},s',g}$. It follows  by Lemma \ref{TermEquiv}.\ref{TermEquiv1} that there is some $a\in \llbracket t_1\rrbracket_{\mathfrak{A},s',g}$ and $b\in \llbracket t_2\rrbracket_{\mathfrak{A},s',g}$ such that $a\not\asymp_{s'}b$. Then by \textit{refinability} for $\asymp$, there is an $s''\sqsubseteq s'$ such that for all $s'''\sqsubseteq s''$, $a\not\asymp_{s'''}b$. It follows by Lemma \ref{TermEquiv}.\ref{TermEquiv2} that for all $s'''\sqsubseteq s''$,  $\varnothing\neq\llbracket t_1\rrbracket_{\mathfrak{A},s''',g}\neq\llbracket t_2\rrbracket_{\mathfrak{A},s''',g}\neq\varnothing$  and hence $\mathfrak{A},s'''\nvDash_g t_1=t_2$. By transitivity of $\sqsubseteq$, $s''\sqsubseteq s'\sqsubseteq s$ implies $s''\sqsubseteq s$. Thus, assuming $\mathfrak{A},s\nvDash_g t_1=t_2$, we have shown that there is an $s''\sqsubseteq s$ such that for all $s'''\sqsubseteq s''$, $\mathfrak{A},s'''\nvDash_g t_1=t_2$, as desired. 

Similarly, for $R(t_1,\dots,t_n)$, suppose $\mathfrak{A},s\nvDash_g R(t_1,\dots,t_n)$, so there is an  ${s'\sqsubseteq s}$ such that each $\llbracket t_i\rrbracket_{\mathfrak{A},s',g}$ is nonempty and $(\llbracket t_1\rrbracket_{\mathfrak{A},s',g},\dots ,\llbracket t_n\rrbracket_{\mathfrak{A},s',g})\not\in I_\asymp(R,s')$. It follows by Definition \ref{InterpretEquiv} that there are $a_1,\dots,a_n$ such that $a_i\in \llbracket t_i\rrbracket_{\mathfrak{A},s',g}$ and $(a_1,\dots, a_n)\not\in V(R,s)$. Then by \textit{refinability} for $R$, there is an $s''\sqsubseteq s'$ such that for all $s'''\sqsubseteq s''$, $(a_1,\dots, a_n)\not\in V(R,s''')$. By transitivity of $\sqsubseteq$, $s'''\sqsubseteq s''\sqsubseteq s'$ implies $s'''\sqsubseteq s'$. Then $a_i\in \llbracket t_i\rrbracket_{\mathfrak{A},s',g}$ implies $a_i\in \llbracket t_i\rrbracket_{\mathfrak{A},s''',g}$ by Lemma \ref{TermEquiv}.\ref{TermEquiv2}, whence  $(a_1,\dots, a_n)\not\in V(R,s''')$ implies that there are no $b_1,\dots,b_n$ such that $b_i\in \llbracket t_i\rrbracket_{\mathfrak{A},s''',g}$ and $(b_1,\dots, b_n)\not\in V(R,s''')$ by \textit{persistence} for $R$. Thus, $(\llbracket t_1\rrbracket_{\mathfrak{A},s''',g},\dots ,\llbracket t_n\rrbracket_{\mathfrak{A},s''',g})\not\in I_\asymp(R,s''')$ and hence $\mathfrak{A},s'''\nvDash_g R(t_1,\dots,t_n)$. By transitivity of $\sqsubseteq$ again, we have $s''\sqsubseteq s$. Thus, assuming $\mathfrak{A},s\nvDash_g R(t_1,\dots,t_n)$, we have shown that there is an $s''\sqsubseteq s$ such that for all $s'''\sqsubseteq s''$, $\mathfrak{A},s'''\nvDash_g R(t_1,\dots,t_n)$.

The cases for $\neg$, $\wedge$, and $\forall x$ are straightforward.\end{proof}

Given the standard definitions of $\vee$, $\to$, and $\leftrightarrow$ in terms of $\neg$ and $\wedge$ and of $\exists$ in terms of $\neg$ and $\forall$, we have the following derived semantic clauses.

\begin{lemma}\label{DerivedClauses} \textnormal{Given a first-order possibility model $\mathfrak{A}=(S,\sqsubseteq, D, \asymp, V)$ for $L$, formula $\varphi$ of $L$, $s\in S$, and variable assignment $g:\mathsf{Var}(L)\to D$:
\begin{enumerate}
\item $\mathfrak{A},s\Vdash_g \varphi\vee\psi$ iff for all $ s'\sqsubseteq s$ there is $ s''\sqsubseteq s'$: $\mathfrak{A},s''\Vdash_g \varphi$ or $\mathfrak{A},s''\Vdash_g \psi$.
\item\label{DerivedClauses2} $\mathfrak{A},s\Vdash_g \varphi\to\psi$ iff for all $ s'\sqsubseteq s$, if $\mathfrak{A},s'\Vdash_g \varphi$ then $\mathfrak{A},s'\Vdash_g \psi$;
\item $\mathfrak{A},s\Vdash_g \varphi\leftrightarrow\psi$ iff for all $ s'\sqsubseteq s$,  $\mathfrak{A},s'\Vdash_g \varphi$ iff $\mathfrak{A},s'\Vdash_g \psi$;
\item $\mathfrak{A},s\Vdash_g \exists x\varphi $ iff for all $ s'\sqsubseteq s$ there is $ s''\sqsubseteq s'$  and $a\in D$: $\mathfrak{A},s''\Vdash_{g[x:=a]} \varphi$.
\end{enumerate}}
\end{lemma}

\begin{proof} We define $\varphi\vee\psi$ as $\neg(\neg\varphi\wedge\neg\psi)$. By the semantic clauses for $\neg$ and $\wedge$, it is easy to see that $\mathfrak{A},s\Vdash_g \neg(\neg\varphi\wedge\neg\psi)$ iff for all $ s'\sqsubseteq s$ there is $ s''\sqsubseteq s'$: $\mathfrak{A},s''\Vdash_g \varphi$ or $\mathfrak{A},s''\Vdash_g \psi$. The proof for $\exists x\varphi$ as $\neg\forall x\neg\varphi$ is just as easy.  The proof for $\varphi\to \psi$ as $\neg (\varphi\wedge\neg\psi)$ requires slightly more. First, observe that $\mathfrak{A},s\Vdash_g \neg(\varphi\wedge\neg\psi)$ iff 
\begin{equation}\mbox{for all }s'\sqsubseteq s\mbox{, if $\mathfrak{A},s'\Vdash_g \varphi$, then there is } s''\sqsubseteq s': \mathfrak{A},s''\Vdash_g \varphi.\label{ImpEq1}\end{equation}
In fact (\ref{ImpEq1}) is equivalent to 
\begin{equation}\mbox{for all }u\sqsubseteq s\mbox{, if $\mathfrak{A},u\Vdash_g \varphi$, then $\mathfrak{A},u\Vdash_g \psi$}.\label{ImpEq2}\end{equation}
The implication from (\ref{ImpEq2}) to (\ref{ImpEq1}) is obvious. From (\ref{ImpEq1}) to (\ref{ImpEq2}), suppose $u\sqsubseteq s$, $\mathfrak{A},u\Vdash_g \varphi$, and $\mathfrak{A},u\not\Vdash_g \psi$. As $\{v\in S\mid \mathfrak{A},v\Vdash_g\varphi\}\in \mathcal{RO}(S,\sqsubseteq)$ by Lemma \ref{FORO}, $\mathfrak{A},u\not\Vdash_g \psi$ implies that there is a $u'\sqsubseteq u$ such that for all $u''\sqsubseteq u'$, $\mathfrak{A},u''\not\Vdash_g\psi$. Then as $\mathfrak{A},u\Vdash_g \varphi$ and $u'\sqsubseteq u$, we have $\mathfrak{A},u'\Vdash_g \varphi$. But then (\ref{ImpEq1}) is false with $s':=u'$. This completes the proof for $\to$, from which the $\leftrightarrow$ case follows.\end{proof}

The next lemma is used to prove the validity of the key first-order axiom $\forall x\varphi\to \varphi^x_t$ when $t$ is substitutable for $x$ in $\varphi$ and $\varphi^x_t$ is the result of replacing all free occurrences of $x$ in $\varphi$ by $t$.

\begin{lemma}[Substitution Lemma]\label{SubLem} \textnormal{For any model $\mathfrak{A}=(S,\sqsubseteq, D, \asymp, V)$ for $L$, $s\in S$,  variable assignment $g:\mathsf{Var}\to D$, formula $\varphi$ of $L$, and $x\in \mathsf{Var}(L)$:
\begin{enumerate}[label=\arabic*.,ref=\arabic*]
\item\label{SubLem1} if $\mathfrak{A},s\Vdash_g\varphi$ and $g(x)\asymp_s a$, then $\mathfrak{A},s\Vdash_{g[x:=a]}\varphi$;
\end{enumerate}
and for any $t\in\mathsf{Term}(L)$  substitutable for $x$ in $\varphi$, the following are equivalent:
\begin{enumerate}[label=\arabic*.,ref=\arabic*,resume]
\item\label{SubLem2} $\mathfrak{A},s\Vdash_g\varphi_t^x$;
\item\label{SubLem3} for all $s'\sqsubseteq s$, $\mathfrak{A},s'\Vdash_{g[x:=a]}\varphi$ for all $a\in \llbracket t\rrbracket_{\mathfrak{A},s',g}$.
\end{enumerate}
If $\mathfrak{A}$ has total functions, then \ref{SubLem3} may be replaced with:
\begin{enumerate}[label=\arabic*$'$.,ref=\arabic*$'$,start=3]
\item\label{SubLem3'} $\mathfrak{A},s\Vdash_{g[x:=a]}\varphi$ for all $a\in \llbracket t\rrbracket_{\mathfrak{A},s,g}$.
\end{enumerate}}
\end{lemma}
\begin{proof} The proof of part \ref{SubLem1} is an easy induction on $\varphi$ using Lemma \ref{TermEquiv}.\ref{TermEquiv1.25}. The equivalence of \ref{SubLem3} and \ref{SubLem3'} under the assumption of total functions follows from parts \ref{TermEquiv1}, \ref{TermEquiv2}, and \ref{TermEquiv2.5} of Lemma \ref{TermEquiv}.

For the equivalence of \ref{SubLem2} and \ref{SubLem3}, the proof is by induction on $\varphi$ using Lemma \ref{TermEquiv}.\ref{TermEquiv3} in the base cases. For the inductive step, we give only the $\neg\varphi$ case, which is the most involved. We have the following equivalences:
\begin{itemize}
\item[] $\mathfrak{A},s\Vdash_g (\neg\varphi)_t^x$;
\item[$\Leftrightarrow$] (by the recursive definition of substitution) $\mathfrak{A},s\Vdash_g \neg \varphi_t^x$;
\item[$\Leftrightarrow$] (by the semantics of $\neg$) for all $s_*\sqsubseteq s$, $\mathfrak{A},s_* \nvDash_g\varphi_t^x$;
\item[$\Leftrightarrow$] (by the inductive hypothesis) for all $s_*\sqsubseteq s$ there is an $s_{**}\sqsubseteq s_*$ such that  $\mathfrak{A},s_{**}\nvDash_{g[x:=b]}\varphi$ for some $b\in \llbracket t\rrbracket_{\mathfrak{A},s_{**},g}$;
\item[$\Leftrightarrow$]  for all $s'\sqsubseteq s$,  $\mathfrak{A},s'\Vdash_{g[x:=a]}\neg \varphi$ for all $a\in \llbracket t\rrbracket_{\mathfrak{A},s',g}$. 
\end{itemize}
For the last equivalence, assume the right-hand side. Suppose $s_*\sqsubseteq s$. Then by Lemma \ref{TermEquiv}\ref{TermEquiv1.5}, there is an $s_{**}\sqsubseteq s_*$ such that $\llbracket t\rrbracket_{\mathfrak{A},s_{**},g}\neq\varnothing$. By the right-hand side, we have $\mathfrak{A},s_{**}\Vdash_{g[x:=a]}\neg \varphi$ for all $a\in \llbracket t\rrbracket_{\mathfrak{A},s_{**},g}$. It follows that $\mathfrak{A},s_{**}\nvDash_{g[x:=b]}\varphi$ for some $b\in \llbracket t\rrbracket_{\mathfrak{A},s_{**},g}$ by the semantics of $\neg$ and the fact that $\llbracket t\rrbracket_{\mathfrak{A},s_{**},g}\neq\varnothing$. Conversely, assume the left-hand side. Suppose $s'\sqsubseteq s$ and pick any $a\in \llbracket t\rrbracket_{\mathfrak{A},s',g}$. To show that $\mathfrak{A},s'\Vdash_{g[x:=a]}\neg \varphi$, consider any $s'_*\sqsubseteq s'$, so $s'_*\sqsubseteq s$ by transitivity of $\sqsubseteq$. Then by the left-hand side, there is an $s'_{**}\sqsubseteq s'_{*}$ such that $\mathfrak{A},s'_{**}\nvDash_{g[x:=b]}\varphi$ for some $b\in \llbracket t\rrbracket_{\mathfrak{A},s'_{**},g}$. By Lemma \ref{TermEquiv}.\ref{TermEquiv2}, $a\in \llbracket t\rrbracket_{\mathfrak{A},s',g}$ and $s_{**}'\sqsubseteq s'$ together imply $a\in \llbracket t\rrbracket_{\mathfrak{A},s'_{**},g}$, so $\mathfrak{A},s'_{**}\nvDash_{g[x:=b]}\varphi$ implies $\mathfrak{A},s'_{**}\nvDash_{g[x:=a]}\varphi$ by part \ref{SubLem1}. Given $s'_{**}\sqsubseteq s'_{*}$, it follows by persistence for $\varphi$ (Lemma \ref{FORO}) that $\mathfrak{A},s'_*\nvDash_{g[x:=a]}\varphi$. Thus, $\mathfrak{A},s'\Vdash_{g[x:=a]}\neg \varphi$, establishing the right-hand side.\end{proof}

Lemmas \ref{FORO}--\ref{SubLem} are the key parts of the proof that first-order logic is sound with respect to possibility semantics. Given a set $\Gamma$ of formulas and formula $\varphi$, let $\Gamma \vdash\varphi$ mean that $\varphi$ is provable in first-order logic from assumptions in $\Gamma$.

\begin{theorem}[Soundness] \textnormal{For any set $\Gamma$ of formulas and formula $\varphi$ of $L$, if $\Gamma\vdash\varphi$, then for every pointed model $\mathfrak{A},s$ and variable assignment $g$, if $\mathfrak{A},s\Vdash_g\psi$ for all $\psi\in\Gamma$, then $\mathfrak{A},s\Vdash_g\varphi$.}
\end{theorem}
\begin{proof}[Proof sketch] Consider a Hilbert-style proof system for first-order logic as in \cite{Enderton2001}. Use Lemma \ref{FORO} to verify that all first-order substitution instances of tautologies are valid, Lemma \ref{SubLem} to verify that $\forall x\varphi\to \varphi^x_t$ is valid when $t$ is substitutable for $x$ in $\varphi$,  and Lemma \ref{DerivedClauses}.\ref{DerivedClauses2} to verify that the other axioms involving $\forall$ and $\to$ are valid and that modus ponens preserves validity.\end{proof}

\subsection{Completeness of first-order logic with respect to possibility semantics}

In this section, we prove the completeness of first-order logic---for an arbitrary first-order language $L$---with respect to first-order possibility semantics. 

We first recall how to extend any consistent set of first-order formulas to one with witnesses for existential formulas---a standard construction not specific to possibility semantics.

\begin{definition} Given a set $\Gamma$ of formulas of $L$, let
\[\mathsf{Cn}_L(\Gamma)=\{\varphi \mbox{ a formula of }L\mid \Gamma\vdash \varphi \}.\]

$\Gamma$ is \textit{deductively closed} if $\Gamma=\mathsf{Cn}_L(\Gamma)$.

$\Gamma$ is \textit{Henkinized} if for every formula of $L$ of the form $\exists x\varphi$, there is a constant symbol $c$ of $L$ such that $\exists x\varphi\to \varphi^x_c\in \Gamma$.
\end{definition}

There is sometimes confusion about whether extending a consistent set to a Henkinized one requires choice principles or transfinite recursion. In fact, no matter the cardinality of $L$, no choice or transfinite recursion is needed. 

\begin{definition}\label{Lomega}
Given any first-order language $L$, we define a countable sequence of first-order languages as follows:
\begin{eqnarray*}
L_0&=& L \\
L_{n+1}&=&\mbox{extension of $L_n$ with new constant $c_{\exists x \varphi}$ for each formula $\exists x\varphi$ of $L_n$}\\
L_\omega &= &\underset{n\in\omega}{\bigcup}L_n .
\end{eqnarray*}
\end{definition}

\begin{lemma}\label{Henkinization}\label{HenkinLem} \textnormal{(ZF) Let $\Gamma$ be a consistent set of formulas of $L$. Then 
\[H(\Gamma)=\mathsf{Cn}_{L_\omega}(\Gamma\cup \{\exists x\varphi \to \varphi^x_{c_{\exists x\varphi}}\mid \exists x\varphi \mbox{ a formula of }L_\omega\})\]
 is a consistent, deductively closed, and Henkinized set of formulas of $L_\omega$.}
\end{lemma}
\noindent For proofs of Lemma \ref{Henkinization}, see, e.g., \cite[pp.~46-7]{Shoenfield2010}, \cite[pp.~82-4]{Ebbinghaus1994}, \cite[pp.~98-9]{Dalen2013}.

Now we begin the specifically possibility-semantic part of the completeness proof, by defining a single first-order possibility model in which all consistent sets of formulas are satisfiable.

\begin{definition}\label{CanModelDef} The \textit{canonical model for $L_\omega$} is the tuple $\mathfrak{A}_{L_\omega}=(S,\sqsubseteq, D, \asymp, V)$ where:
\begin{enumerate}
\item\label{CanModelDef1} $S$ is the set of all consistent, deductively closed, and Henkinized sets of formulas of $L_\omega$;
\item $\Gamma'\sqsubseteq \Gamma$ iff $\Gamma'\supseteq \Gamma$;
\item $D$ is the set of terms of $L_\omega$;
\item $t\asymp_\Gamma t'$ iff $t=t'\in \Gamma$;
\item for any $n$-ary predicate symbol $R$ and $\Gamma\in S$, \[V(R,\Gamma)=\{(t_1,\dots,t_n)\mid R(t_1,\dots t_n)\in \Gamma\};\]
\item\label{CanModelDefFunc} for any $n$-ary function symbol $f$ of $L_\omega$ and $\Gamma\in S$, 
\[V(f,\Gamma)=\{(t_1,\dots,t_{n+1}) \mid  f(t_1,\dots,t_n)=t_{n+1}\in \Gamma\}.\]
\end{enumerate}
The \textit{canonical model for $L$}, $\mathfrak{A}_L$, is the reduct of $\mathfrak{A}_{L_\omega}$ to $L$, i.e., the model exactly like  $\mathfrak{A}_{L_\omega}$ except it does not interpret constants that are in $L_\omega$ but not~$L$.
\end{definition}

In Definition \ref{CanModelDef}.\ref{CanModelDef1}, one may replace `Henkinized' with `canonically Henkinized', where a set of formulas of $L_\omega$ is \textit{canonically Henkinized} if for every formula of $L_\omega$ of the form $\exists x\varphi$, we have $\exists x\varphi\to \varphi^x_{c_{\exists x\varphi}}\in \Gamma$ for $c_{\exists x\varphi}$  in particular, where $c_{\exists x\varphi}$ is the constant introduced for $\exists x\varphi$ in Definition \ref{Lomega}.

\begin{lemma}\label{CanModel} \textnormal{The canonical model for $L_\omega$ (resp.~$L$) is a first-order possibility model for $L_\omega$ (resp.~$L$) with total functions.}
\end{lemma}

\begin{proof} Conditions 1 and 2 of Definition \ref{FOmodel} are immediate. The other conditions correspond to the following easily verified syntactic facts:
\begin{itemize}
\item that $\asymp_\Gamma$ is an equivalence relation follows from the facts that $\vdash t=t$, $\vdash t=t'\to t'=t$, and $\vdash (t=t'\wedge t'=t'')\to t=t''$;
\item persistence for $\asymp$: if $a=b\in \Gamma$ and $\Gamma'\sqsubseteq \Gamma$, so $\Gamma'\supseteq \Gamma$, then $a=b\in \Gamma'$;
\item refinability for $\asymp$: if $a=b\not\in \Gamma$, then there is a $\Gamma'\sqsubseteq \Gamma$ such that $\neg a=b\in \Gamma'$, namely $\Gamma ' = \mathsf{Cn}_{L_\omega}(\Gamma\cup \{\neg a=b\})$, and hence $ a=b\not\in \Gamma''$ for all $\Gamma''\sqsubseteq\Gamma'$;
\item persistence for $R$: if $R(t_1,\dots,t_n)\in \Gamma$, $\Gamma'\sqsubseteq \Gamma$, and for each $i$, \\ ${t_i=t_i' \in \Gamma'}$, then $R(t_1',\dots,t_n')\in \Gamma'$;
\item refinability for $R$: if $R(t_1,\dots,t_n)\not\in \Gamma$, then there is a  $\Gamma'\sqsubseteq \Gamma$ such that $\neg R(t_1,\dots,t_n)\in \Gamma'$ and hence $R(t_1,\dots,t_n)\not\in \Gamma''$ for all $\Gamma''\sqsubseteq \Gamma'$;
\item persistence for $f$: if $f(t_1,\dots,t_n)=t_{n+1}\in \Gamma$, $\Gamma'\sqsubseteq \Gamma$, and for each $i$, $t_i=t_i' \in \Gamma'$, then  $f(t_1',\dots,t_n')=t_{n+1}'\in \Gamma'$.
\item quasi-functionality for $f$: if $f(t_1,\dots,t_n)=t_{n+1}\in \Gamma$, then \\ $f(t_1,\dots,t_n)=t'\in \Gamma$ implies $t_{n+1}=t'\in \Gamma$.
\end{itemize}
Finally, that the model has total functions follows from the fact that for every term $f(t_1,\dots,t_n)$ and $\Gamma\in S$, we have $f(t_1,\dots,t_n)=f(t_1,\dots,t_n)\in \Gamma$, so $(t_1,\dots, t_n,f(t_1,\dots,t_n))\in V(f,\Gamma)$.\end{proof}

For the next two lemmas, let $g$ be the variable assignment for $\mathfrak{A}_{L_\omega}$ such that $g(x)=x$.

\begin{lemma}[Term Lemma]\label{TermLemma} \textnormal{For every term $t$ of $L_\omega$ and $\Gamma\in S$, $t\in \llbracket t\rrbracket_{\mathfrak{A}_{L_\omega},\Gamma,g}$.}
\end{lemma}

\begin{proof}[Proof sketch.] By induction on the structure of $t$, using the semantics of terms in Definition \ref{DenoteDef}.\ref{DenoteDef2} and interpretation of function symbols in Definition \ref{CanModelDef}.\ref{CanModelDefFunc}.\end{proof}

\begin{lemma}[Truth Lemma]\label{TruthLemma} \textnormal{For every formula $\varphi$ of $L_\omega$ and $\Gamma\in S$, \[\mathfrak{A}_{L_\omega},\Gamma\Vdash_g \varphi\mbox{ iff } \varphi\in \Gamma.\]}
\end{lemma}
\begin{proof} As usual, the proof is by induction on the complexity of $\varphi$, and the most interesting part is the proof that $\mathfrak{A}_{L_\omega},\Gamma\Vdash_g \forall x\varphi$ implies $\forall x\varphi\in \Gamma$, which uses that $\Gamma$ is Henkinized, so there is a constant $c$ such that $\exists x\neg \varphi\to \neg\varphi^x_c\in \Gamma$:
\begin{eqnarray*}
\mathfrak{A}_{L_\omega},\Gamma\Vdash_g \forall x\varphi & \Rightarrow & \mathfrak{A}_{L_\omega},\Gamma\Vdash_{g[x:=c]} \varphi\mbox{ by Definition \ref{FOPossSat}.\ref{FOPossSatForall}} \\
& \Rightarrow & \mathfrak{A}_{L_\omega},\Gamma\Vdash_{g} \varphi^x_{c}\mbox{ by Lemmas \ref{SubLem} and \ref{TermLemma}} \\
& \Rightarrow & \varphi^x_{c}\in \Gamma\mbox{ by the inductive hypothesis} \\
&\Rightarrow& \forall x\varphi\in \Gamma\mbox{ since }\varphi^x_{c}\to  \forall x \varphi\in \Gamma.\qedhere
\end{eqnarray*} 
\end{proof}

\begin{theorem}[Strong Completeness]\label{StrongComp} \textnormal{(ZF) Every consistent set $\Gamma$ of formulas of $L$ is satisfiable in the canonical first-order possibility model for $L$.}
\end{theorem}
\begin{proof} By Lemma \ref{TruthLemma}, $\Gamma$ is satisfied at the possibility $H(\Gamma)$ (recall Lemma \ref{Henkinization}) in $\mathfrak{A}_{L_\omega}$ and hence at the same possibility in the reduct $\mathfrak{A}_L$.
\end{proof}

Completeness is only the first step in a model theory for first-order languages based on possibility models.  Using his definition of first-order possibility models, van Benthem \cite{Benthem1981a,Benthem2016b} defines several operations on models: generated submodels, disjoint unions, zigzag images, filter products, and filter bases. He then proves a Keisler-type definability result, stating that a class of possibility models is definable by a set of first-order sentences iff the class is closed under the listed operations. His notion of \textit{filter product} is meant to play the role in possibility semantics of ultraproducts in traditional semantics. While taking a product $\{\mathfrak{A}_i\}_{i\in I}$ for an arbitrary $I$ requires the Axiom of Choice, one can handle a countable $I$ quasiconstructively, as Dependent Choice implies Countable Choice. It remains to be seen how other concepts and results of traditional first-order model theory may be developed for possibility models.

\subsection{First-order possibility models and forcing}\label{ForcingSection}

In this section, we briefly explain how special first-order possibilities models are essentially already implicit in forcing in set theory (cf.~\cite{Fitting1969}). We follow the presentation of forcing using Boolean-valued models in \cite{Bell2005}.

Assume there is a standard transitive Tarskian model $M$ of ZFC. Suppose $(S,\sqsubseteq)$ is a poset in $M$, which we will use as the underlying poset of a first-order possibility model. Next we will add a set $D$ of \textit{guises of sets}. The basic idea is this:  for a partial possibility $s$ in $S$ and guises $a$ and $b$ of sets, it may be that $s$ does not settle that \[\mbox{the set of which $a$ is a guise is a member of the set of which $b$ is a guise},\] but also $s$ does not settle that this is \textit{not} the case. Given this idea, we take the guise $b$ to be a \textit{function} that outputs, for a given guise $a$ in its domain, a regular open set   $b(a)$ of possibilities such that every $s\in b(a)$ settles that the set of which $a$ is a guise is a member of the set of which $b$ is a guise. Formally, let $B$ be the BA $\mathcal{RO}(S,\sqsubseteq)$ in $M$, and define the following sequence of sets by recursion on~$\alpha$:

\begin{itemize}
\item $M^{(B)}_\alpha$ is the set of all functions in $M$ whose domain is a subset of $M^{(B)}_\xi$ for some $\xi<\alpha$ and whose codomain is $B$;
\item $M^{(B)}=\underset{\alpha}{\bigcup}M^{(B)}_\alpha$.
\end{itemize}
Let the set $D$ of guises be $M^{(B)}$. Let $L_D$ be the language of first-order set theory, with $\dot{\in}$ the membership symbol, expanded with a constant $c_a$ for every $a\in D$. 

Next we define by a joint recursion\footnote{To make this rigorous in terms of recursion on an appropriate well-founded relation, see \cite[p.~23]{Bell2005}.} the functions $\asymp$ and $V(\dot{\in},\cdot)$ such that for all possibilities $s\in S$ and guises $a,b\in D$:
\begin{itemize}
\item $a\asymp_s b$ iff for all $c\in \mathrm{dom}(a)$ and $s'\sqsubseteq s$, if $s'\in a(c)$, then $(c,b)\in V(\dot{\in},s')$, and for all $c\in \mathrm{dom}(b)$ and $s'\sqsubseteq s$, if $s'\in b(c)$, then $(c,a)\in V(\dot{\in},s')$;
\item $(a,b)\in V(\dot{\in},s)$ iff  $\forall s'\sqsubseteq s$ $\exists s''\sqsubseteq s'$ $\exists c\in\mathrm{dom}(b)$: $s''\in b(c)$ and $a\asymp_{s''}c$. 
\end{itemize}
Finally, we interpret the constants in the obvious way:
\begin{itemize}
\item $b\in V(c_a,s)$ iff $a\asymp_s b$.
\end{itemize}

\begin{proposition} \textnormal{The tuple $(S,\sqsubseteq,D, \asymp, V)$ constructed above is a first-order possibility model for $L_D$ with total functions.}
\end{proposition}
\begin{proof} First observe that for each $s\in S$, $\asymp_s$ is an equivalence relation. 

 \textit{Persistence} for $\asymp$ is immediate from the definition of $\asymp$ due to the quantification over all $s'\sqsubseteq s$. For \textit{persistence} for $\dot{\in}$, suppose $(a,b)\in V(\dot{\in},s)$, $u\sqsubseteq s$, $a\asymp_{u}a^*$, and $b\asymp_{u}b^*$. From $(a,b)\in V(\dot{\in},s)$, we have (i) $\forall s'\sqsubseteq s$ $\exists s''\sqsubseteq s'$ $\exists c\in\mathrm{dom}(b)$: $s''\in b(c)$ and $a\asymp_{s''}c$. To show $(a^*,b^*)\in V(\dot{\in},u)$, let $u'\sqsubseteq u$. Then $u' \sqsubseteq s$, so by (i), there are $u''\sqsubseteq u'$ and $c\in\mathrm{dom}(b)$ with $u''\in b(c)$ and $a\asymp_{u''}c$. Since $b\asymp_{u}b^*$ and $u''\sqsubseteq u$, we have  $b\asymp_{u''}b^*$ by persistence for $\asymp$, which with $u''\in b(c)$ implies $(c,b^*)\in V(\dot{\in},u'')$. By definition of $V(\dot{\in},u'')$, it follows that for some  $u'''\sqsubseteq u''$ and $d\in \mathrm{dom}(b^*)$, we have $u'''\in b^*(d)$ and $c\asymp_{u'''}d$. Since $a\asymp_{u}a^*$, $a\asymp_{u''}c$, $c\asymp_{u'''}d$, and $u'''\sqsubseteq u''\sqsubseteq u$, we have  $a^*\asymp_{u'''}d$ by persistence for $\asymp$ and the fact that $\asymp_{u'''}$ is an equivalence relation. Thus, we have shown that for all $u'\sqsubseteq u$ there is a $u'''\sqsubseteq u'$ and $d\in \mathrm{dom}(b^*)$ such that $u'''\in b^*(d)$ and $a^*\asymp_{u'''}d$. Hence $(a^*,b^*)\in V(\dot{\in},u)$.

 \textit{Refinability} for $\dot{\in}$ is immediate from the definition of $V(\dot{\in},s)$ due to the $\forall s'\sqsubseteq s$  $\exists s''\sqsubseteq s'$ quantification pattern. For \textit{refinability} for $\asymp$, suppose $a\not\asymp_s b$. Without loss of generality, suppose there is a $c\in \mathrm{dom}(a)$ and $s'\sqsubseteq s$ such that $s'\in a(c)$ but $(c,b)\not\in V(\dot{\in},s')$. Then by \textit{refinability} for $\dot{\in}$, there is an $s''\sqsubseteq s'$ such that for all $s'''\sqsubseteq s''$, $(c,b)\not\in V(\dot{\in},s''')$. From $s'\in a(c)$ and $s'''\sqsubseteq s'$, we have $s'''\in a(c)$, which with $(c,b)\not\in V(\dot{\in},s''')$ implies $a\not\asymp_{s'''} b$. Thus, we have an $s''\sqsubseteq s$ such that for all $s'''\sqsubseteq s''$, $a\not\asymp_{s'''} b$, which establishes \textit{refinability} for $\asymp$.
 
 \textit{Persistence}, \textit{quasi-functionality}, and \textit{totality} for constant symbols follows from the properties of $\asymp$.\end{proof}

One may now verify that possibility semantics for $L_D$ using $(S,\sqsubseteq,D, \asymp, V)$ results in the same semantic values for formulas in the Boolean algebra $\mathcal{RO}(S,\sqsubseteq)$ (recall Lemma \ref{FORO}) as the algebraic semantics for $L_D$ in \cite[Ch.~1]{Bell2005}.

\section{Modal case}\label{ModalSection}

In this section, we cover possibility semantics for modal logic. We assume some previous familiarity with possible world semantics using Kripke frames \cite{Chagrov1997,Blackburn2001} and neighborhood frames \cite{Pacuit2017}, to which we will compare possibility semantics.

 We will consider four ways of interpreting the modal language using possibility frames:
\begin{enumerate}
\item (\S~\ref{PropQuantSection}) Interpret $\Box$ as the universal modality in possibility frames. When combined with the device of \textit{propositional quantification}, this simple semantics already allows us to characterize naturally occurring logics that cannot be characterized by Kripke frames.
\item (\S~\ref{NeighSection}) Add neighborhood functions to possibility frames to interpret arbitrary congruential modalities. This allow us to characterize simple modal logics (without propositional quantification) that cannot be characterized by possible world frames with neighborhood functions.
\item (\S~\ref{RelationalFrameSection}) Add accessibility relations to possibility frames to interpret normal modalities. This was Humberstone's \cite{Humberstone1981} original setting, though our definition of relational possibility frames is more general. Again this allows us to characterize modal logics (without propositional quantification) that cannot be characterized by Kripke frames or neighborhood world frames.
\item (\S~\ref{FunctionalPoss}) Add ``accessibility functions'' to possibility frames to interpret normal modalities: a modal formula $\Box\varphi$ is true at a possibility $x$ iff $\varphi$ is true at the functionally determined possibility $f(x)$. While this functional semantics is severely limited when paired with possible world frames, where $f(x)$ must be a complete world, it is quite general when paired with possibility frames.
\end{enumerate}

After covering these four ways of interpreting modalities in the context of propositional modal logic, in \S~\ref{QuantModalSection} we briefly consider \textit{first-order modal logic}, extending the first-order possibility frames of \S~\ref{FOSection} with accessibility relations as in \S~\ref{RelationalFrameSection}. We show there are philosophically controversial first-order modal inferences valid over possible world frames but invalid over possibility frames.

\subsection{Universal modality and propositional quantification}\label{PropQuantSection}

The language $\mathcal{L}$ of unimodal propositional logic is defined by the grammar
\[\varphi::= p \mid \neg\varphi\mid (\varphi\wedge\varphi)\mid \Box\varphi\]
where $p$ belongs to a countably infinite set $\mathsf{Prop}$ of propositional variables. We define $\vee$, $\rightarrow$, and $\leftrightarrow$ in terms of $\neg$ and $\wedge$ as usual, and
\[\Diamond \varphi:=\neg\Box\neg\varphi.\]
A simple possibility semantics for this language using posets $(S,\sqsubseteq)$ with $\Box$ interpreted as the universal modality (see Definition \ref{PossForcing1} below) provides a semantics for the logic $\mathsf{S5}$, the smallest normal modal logic (see Definition \ref{NormalLogic}) containing  \[\Box p\to p,\,\Box p\to \Box\Box p,\, \neg \Box p\to \Box\neg\Box p,\] 
which is a standard logic for the ``necessity'' interpretation of $\Box$ in philosophy and the ``knowledge'' interpretation of $\Box$ in computer science and game theory. As $\mathsf{S5}$ is already complete with respect to sets of worlds with $\Box$ interpreted as the universal modality, so far there is no gain concerning completeness. However, as we will see, if we enrich the language sufficiently, then there are extensions of $\mathsf{S5}$ that are incomplete with respect to a semantics based on worlds but complete with respect to possibility~semantics. 

Consider the language $\mathcal{L}\Pi$ that extends $\mathcal{L}$ with propositional quantifiers:
\[\varphi::= p \mid \neg\varphi\mid (\varphi\wedge\varphi)\mid \Box\varphi\mid \forall p \varphi\]
where $p\in\mathsf{Prop}$. We define $\exists p\varphi:=\neg\forall p\neg\varphi$. From a syntactic perspective, perhaps the most natural extension of $\mathsf{S5}$ with principles for propositional quantification is the logic \textsf{S5}$\Pi$ studied by Bull \cite{Bull1969} and Fine \cite{Fine1970}, which extends the axioms and rules of \textsf{S5} with the following axioms and rule for the propositional quantifiers:
\begin{itemize}
\item Universal distribution axiom: $\forall p (\varphi\to \psi) \to (\forall p\varphi\to\forall p\psi)$.
\item Universal instantiation axiom: $\forall p\varphi \to \varphi^p_\psi$ where $\psi$ is substitutable for $p$ in $\varphi$, and $\varphi^p_\psi$ is the result of replacing all free occurrences of $p$ in $\varphi$ by $\psi$ (defining substitutability and free occurences of propositional variables just as we do for individual variables in first-order logic \cite{Enderton2001}).
\item Vacuous quantification axiom: $\varphi\to\forall p\varphi$ where $p$ is not free in $\varphi$.
\item Rule of universal generalization: if $\varphi$ is a theorem, then $\forall p\varphi$ is a theorem.
\end{itemize}
The set of theorems of $\mathsf{S5}\Pi$ is the smallest set of sentences of $\mathcal{L}\Pi$ containing the stated axioms and closed under the stated rules.

A simple possibility semantics for the logic \textsf{S5}$\Pi$ is as follows.

\begin{definition}\label{PossForcing1} For a full possibility frame $\mathcal{F}=(S,\sqsubseteq, \mathcal{RO}(S,\sqsubseteq))$, a \textit{possibility model based on $\mathcal{F}$} is a pair $\mathcal{M}={(\mathcal{F},\pi)}$ such that $\pi: \mathsf{Prop}\to \mathcal{RO}(S,\sqsubseteq)$. For $\varphi\in\mathcal{L}\Pi$ and $x\in S$, we define the forcing relation $\mathcal{M},x\Vdash \varphi$ recursively as follows:
\begin{enumerate}
\item $\mathcal{M},x\Vdash p$ iff $x\in \pi(p)$;
\item $\mathcal{M},x\Vdash \neg\varphi$ iff for all $x'\sqsubseteq x$, $\mathcal{M},x'\nVdash\varphi$;
\item $\mathcal{M},x\Vdash (\varphi\wedge\psi)$ iff $\mathcal{M},x\Vdash\varphi$ and $\mathcal{M},x\Vdash\psi$;
\item $\mathcal{M},x\Vdash \Box\varphi$ iff for all $y\in S$, $\mathcal{M},y\Vdash \varphi$;
\item\label{PossForcing15} $\mathcal{M},x\Vdash \forall p \varphi$ iff for all $\mathcal{M}'\sim_p\mathcal{M}$, $\mathcal{M}',x\Vdash \varphi$,
\end{enumerate}
where $\mathcal{M}'\sim_p\mathcal{M}$ means that $\mathcal{M}'$ is a possibility model based on the same $\mathcal{F}$ as $\mathcal{M}$ whose valuation $\pi'$ differs from the valuation $\pi$ in $\mathcal{M}$ at most at $p$.

A formula $\varphi$ is \textit{valid on} $\mathcal{F}$ if for every model $\mathcal{M}$ based on $\mathcal{F}$ and every $x\in S$, $\mathcal{M},x\Vdash \varphi$; and $\varphi$ is \textit{valid on a class $\mathsf{K}$} of full possibility frames if it is valid on every $\mathcal{F}$ in $\mathsf{K}$. The \textit{logic of $\mathsf{K}$} is the set of all formulas valid on $\mathsf{K}$.
\end{definition}

\begin{remark}\label{PropQuantAlg}Algebraically, we are interpreting $\mathcal{L}\Pi$ in a complete Boolean algebra, namely ${\mathcal{RO}(S,\sqsubseteq)}$, taking the semantic value of $\Box \varphi$ to be 1 if the  value of $\varphi$ is 1, and 0 otherwise, and taking the semantic value of $\forall p\varphi$ to be the meet of all  values of $\varphi$ as we vary the  value of $p$ to be any element of the algebra (cf.~\cite{Holliday2019}).
\end{remark}

\begin{remark}\label{OtherModels} We can also define models based on arbitrary---not only full---possibility frames. For a frame $\mathcal{F}=(S,\sqsubseteq, P)$, a \textit{possibility model  based on $\mathcal{F}$} is a pair $\mathcal{M}=(\mathcal{F},\pi)$ with $\pi:\mathsf{Prop}\to P$. However, for interpreting $\mathcal{L}\Pi$, we restrict attention to full possibility frames so that $P=\mathcal{RO}(S,\sqsubseteq)$ and hence the semantic value of $\forall p\varphi$ belongs to $P$ by Remark \ref{PropQuantAlg}. This is used to prove the validity of the universal instantiation axiom for $\psi$ containing propositional quantifiers.\end{remark}

The derived semantic clauses for $\vee$, $\to$, $\leftrightarrow$, and $\exists$ are analogous to those in Lemma \ref{DerivedClauses}. In addition, we have \[\mbox{$\mathcal{M},x\Vdash \Diamond\varphi$ iff for some $y\in S$, $\mathcal{M},y\Vdash\varphi$.}\]

Recall that a \textit{world frame} is a possibility frame in which $\sqsubseteq$ is the identity relation (Definition \ref{PosFrameDef}).

\begin{proposition}\label{AtomNotThm} \textnormal{$\mathsf{S5}\Pi$ is incomplete with respect to the class of all full world frames, as
\begin{equation} \exists q (q\wedge \forall p(p\to \Box (q\to p)))\label{BadAx2}\tag{W}\end{equation}
 is valid on every full world frame but is not a theorem of $\mathsf{S5}\Pi$.}
\end{proposition}
\begin{proof} Suppose $\mathcal{M}=(\mathcal{F},\pi)$ is a model based on a world frame $\mathcal{F}$, and $w$ is a world in $\mathcal{F}$. Since $\mathcal{F}$ is a world frame, we have $\mathcal{M},w\Vdash (\mathrm{W})$ iff there is some $\mathcal{M}'\sim_q\mathcal{M}$ such that $\mathcal{M}',w\Vdash q\wedge \forall p(p\to \Box (q\to p))$. To see that the right-hand side holds, let $\mathcal{M}'=(\mathcal{F},\pi')$ be such that $\pi'$ differs from $\pi$ only in that  $\pi'(q)=\{w\}$.

To see that (\ref{BadAx2}) is not a theorem of $\mathsf{S5}\Pi$, first it is easy to check that $\mathsf{S5}\Pi$ is valid on any full possibility frame. Second, we observe that (\ref{BadAx2}) is refuted by the full possibility frame based on the full infinite binary tree in Examples \ref{BinaryTree0} and \ref{BinaryTree}. For any possibility $x$ and formula $\varphi$, we have $\mathcal{M},x\Vdash \exists q\varphi$ iff  for all $ x'\sqsubseteq x$ there is $x''\sqsubseteq x'$ and $\mathcal{M}'\sim_q\mathcal{M}$ such that $\mathcal{M}',x''\Vdash\varphi$. But for any $x'\sqsubseteq x$,  $x''\sqsubseteq x'$, and regular open set $U$   chosen to interpret $q$ in $\mathcal{M}'$, we claim that \[\mathcal{M}',x''\nVdash q\wedge \forall p(p\to \Box(q\to p)).\]
Assume $x''\in U$, so the first conjunct is true. Where $y$ is a child of $x''$, interpret $p$ as the regular open set $\mathord{\downarrow}y\subsetneq U$. Then $p\to \Box(q\to p)$ is not true at $x''$, since the antecedent is true at $y\sqsubseteq x''$ but the consequent is not true at any possibility.\end{proof}

\noindent Stated algebraically, (\ref{BadAx2}) is valid on any complete \textit{and atomic} Boolean algebra. 

\begin{remark}More generally, if instead of interpreting $\mathcal{L}\Pi$ in complete Boolean algebras as in Remark \ref{PropQuantAlg}, we interpret $\mathcal{L}\Pi$ in complete Boolean algebra expansions, defined in the next section (Definition \ref{BAEdef}), then (\ref{BadAx2}) is valid on any complete \textit{and atomic} Boolean algebra expansion in which $\Box\top$ is valid.
\end{remark}

Although $\mathsf{S5}\Pi$ is incomplete with respect to possible world semantics, it is complete with respect to possibility semantics in virtue of its completeness with respect to complete Boolean algebras, shown in \cite{Holliday2019}, and the fact that we can transform any complete Boolean algebra refuting a formula into a possibility frame refuting the formula using Theorem \ref{FirstThm}.\ref{FirstThm2}.

\begin{theorem}[\cite{Holliday2019}]\label{S5PiComp} \textnormal{$\mathsf{S5}\Pi$ is the logic of the class of all full possibility frames.}
\end{theorem}

\begin{theorem}[\cite{Ding2018}]\label{S5PiExtComp} \textnormal{There are infinitely many normal extensions $\mathsf{L}$ of $\mathsf{S5}\Pi$ such that $\mathsf{L}$ is the logic of a class of full possibility frames but not of any class of full world frames.}
\end{theorem}

Below we will see that when we consider modal logics weaker than $\mathsf{S5}$, there are many open problems concerning propositional quantification.

\begin{remark} The approach to propositional quantification in the references above relies on the completeness of the Boolean algebra of propositions for the interpretation of $\forall$ and $\exists$. For a way of interpreting propositional quantification in possibly incomplete algebras, see \cite{H&L2016}.
\end{remark}

\subsection{Neighborhood frames}\label{NeighSection}

To go beyond semantics for the universal modality and cover all manner of other modalities, in this section we introduce the possibility semantic version of a very general style of modal semantics known as \textit{neighborhood semantics} \cite{Pacuit2017}, due to Scott \cite{Scott1970} and Montague \cite{Montague1970}. We give neighborhood possibility semantics for a language containing a family of modalities indexed by some nonempty set $I$. 

The grammar of the polymodal language $\mathcal{L}(I)$ is given by
\[\varphi::= p \mid \neg\varphi\mid (\varphi\wedge\varphi)\mid \Box_i\varphi\]
where $p\in\mathsf{Prop}$ and $i\in I$. $\mathcal{L}(I)\Pi$ adds the clause for $\forall p\varphi$. When $I$ is a singleton set, identify $\mathcal{L}$ from \S~\ref{PropQuantSection} with $\mathcal{L}(I)$.

The language $\mathcal{L}(I)$ has a straightforward interpretation in Boolean algebras equipped with functions interpreting the modalities.

\begin{definition}\label{BAEdef} A \textit{Boolean algebra expansion} (BAE) is a pair $\mathbb{B}=( B, \{f_i\}_{i\in I} )$ where $B$ is a Boolean algebra and $f_i:B\to B$. A \textit{valuation} on $\mathbb{B}$ is a function $\theta:\mathsf{Prop}\to B$, which extends to an \textit{$\mathcal{L}(I)$-valuation} $\tilde{\theta}:\mathcal{L}(I)\to B$ by: $\tilde{\theta}(p)=\theta(p)$; $\tilde{\theta}(\neg \varphi)=\neg \tilde{\theta}(\varphi)$; $\tilde{\theta}(\varphi\wedge\psi)=\tilde{\theta}(\varphi)\wedge \tilde{\theta}(\psi)$; and $\tilde{\theta}(\Box_i\varphi)=f_i(\tilde{\theta})$. We say $\mathbb{B}$ \textit{validates} a formula $\varphi$ if for every valuation $\theta$ on $\mathbb{B}$,  $\tilde{\theta}(\varphi)=1_B$ where $1_B$ is the top element of $B$; and a class of BAEs validates $\varphi$ if every BAE in the class validates $\varphi$.
\end{definition}

\noindent The set of formulas validated by any class of BAEs is a congruential modal logic.

\begin{definition}\label{CongLog} A \textit{congruential modal logic} for $\mathcal{L}(I)$  is a set $\mathsf{L}$ of formulas satisfying the following conditions:
\begin{enumerate}
\item $\mathsf{L}$ contains every tautology of classical propositional logic;
\item rule of uniform substitution: if $\varphi\in\mathsf{L}$ and $\varphi'$ is obtained from $\varphi$ by uniformly substituting formulas for propositional variables, then $\varphi'\in\mathsf{L}$;
\item rule of modus ponens: if $\varphi\to\psi\in\mathsf{L}$ and $\varphi\in\mathsf{L}$, then $\psi\in\mathsf{L}$;
\item congruence rule: if $\varphi\leftrightarrow\psi\in\mathsf{L}$, then $\Box_i\varphi\leftrightarrow\Box_i\psi\in\mathsf{L}$.
\end{enumerate}
\end{definition}

\subsubsection{Basic frames}\label{BasicNeighSection}

The first step toward neighborhood possibility semantics is to add to a poset a collection of neighborhood functions for interpreting the modalities $\Box_i$. Recall that in possible world semantics, a neighborhood frame  is a pair $(W,\{N_i\}_{i\in I})$ where $W$ is a nonempty set and $N_i$ is a function assigning to each $w\in W$ a set $N_i(w)$ of propositions, i.e., of subsets of $W$ (see \cite{Pacuit2017}). The modal operation $\Box_i$ is defined as follows: given a proposition $U$, a world $w$ belongs to the proposition $\Box_i U$ iff $U$ belongs to $N(w)$. Now we will replace $W$ by a poset of possibilities and require that $N_i$ assigns to each possibility a set of propositions in the sense of possibility semantics, i.e., regular open subsets of the poset.

\begin{definition} A \textit{neighborhood possibility foundation} is a triple \\ $F=( S,\sqsubseteq, \{N_i\}_{i\in I})$ such that:
\begin{enumerate}
\item $( S,\sqsubseteq)$ is a poset;
\item $N_i:S\to \wp(\mathcal{RO}(S,\sqsubseteq))$.
\end{enumerate}
For $i\in I$ and $U\in \mathcal{RO}(S,\sqsubseteq)$, we define 
\[\Box_{N_i}U=\{x\in S\mid U\in N_i(x)\}.\]
\end{definition}
These structures are not yet neighborhood possibility \textit{frames}, because without some interaction conditions relating $\sqsubseteq$ and $N_i$, there is no guarantee that $\Box_{N_i}U\in \mathcal{RO}(S,\sqsubseteq)$. The following conditions are necessary and sufficient to guarantee that for all $U\in \mathcal{RO}(S,\sqsubseteq)$, we have $\Box_{N_i}U\in \mathcal{RO}(S,\sqsubseteq)$:
\begin{itemize}
\item $\boldsymbol{N_i}$-\textbf{persistence}: if $x'\sqsubseteq x$, then $N_i(x')\supseteq N_i(x)$;
\item $\boldsymbol{N_i}$-\textbf{refinability}: if $U\not\in N_i(x)$, then $\exists x'\sqsubseteq x$ $\forall x''\sqsubseteq x'$ $U\not\in N_i(x'')$.
\end{itemize}

\begin{proposition}\label{NeighBoxClosure} \textnormal{For any foundation $F=( S,\sqsubseteq, \{N_i\}_{i\in I})$, the following are equivalent:
\begin{enumerate}
\item $\mathcal{RO}(S,\sqsubseteq)$ is closed under $\Box_{N_i}$;
\item\label{NeighBoxClosure2} $F$ satisfies  $\boldsymbol{N_i}$-\textbf{persistence} and $\boldsymbol{N_i}$-\textbf{refinability} for each $i\in I$.
\end{enumerate}}
\end{proposition}
\begin{proof} The claim that for all $U\in \mathcal{RO}(S,\sqsubseteq)$, $\Box_{N_i}U$ satisfies persistence and refinability (recall \S~\ref{CompleteBooleanSection}) is clearly equivalent to $\boldsymbol{N_i}$-\textbf{persistence} and $\boldsymbol{N_i}$-\textbf{refinability}.\end{proof}

Now if we only wish to represent \textit{complete} BAEs, i.e., BAEs whose underlying Boolean algebras are complete, then there is no need for the distinguished family $P$ of admissible sets in a possibility frame (recall Definition \ref{PosFrameDef}), so it suffices to work with the following ``basic'' frames that simply expand posets with appropriate neighborhood functions.

\begin{definition}\label{NeighPossFrame} A \textit{basic neighborhood possibility frame} is a neighborhood possibility foundation $F=( S,\sqsubseteq, \{N_i\}_{i\in I})$ satisfying, for each $i\in I$, $\boldsymbol{N_i}$-\textbf{persistence} and $\boldsymbol{N_i}$-\textbf{refinability}.

A \textit{basic neighborhood world frame} is a neighborhood possibility foundation in which $\sqsubseteq$ is the identity relation.
\end{definition}

\begin{proposition}\label{BasicNeighToBAE} \textnormal{For any basic neighborhood possibility frame $F$, the pair $F^\mathsf{b}=( \mathcal{RO}(S,\sqsubseteq), \{\Box_{N_i}\}_{i\in I})$ is a complete BAE.}
\end{proposition}

\begin{proof} Follows from Theorem \ref{FirstThm}.\ref{FirstThm1} and Proposition \ref{NeighBoxClosure}.
\end{proof}

\begin{remark}\label{BasicNeighWorld}Basic neighborhood \textit{world} frames may be identified with the standard ``neighborhood frames'' of possible world semantics mentioned at the beginning of this section. The key point about these world frames is that they can realize only \textit{atomic} algebras as $F^\mathsf{b}$.\end{remark}

\begin{example}\label{ImpreciseEx} This example is inspired by the theory of \textit{imprecise probability} \cite{BradleySEP}. Fix a measurable space $(W,\Sigma)$, i.e., $W$ is a nonempty set and $\Sigma$ is an algebra of subsets of $W$ closed under countable unions. A set $\mathcal{P}$ of probability measures on $(W,\Sigma)$ is \textit{convex} if for all $\mu_1,\mu_2\in\mathcal{P}$ and $\alpha\in [0,1]$, we have $\alpha\mu_1+(1-\alpha)\mu_2\in\mathcal{P}$, where $\alpha\mu_1+(1-\alpha)\mu_2$ is the probability measure defined for each $A\in\Sigma$ by $(\alpha\mu_1+(1-\alpha)\mu_2)(A)=\alpha\mu_1(A)+(1-\alpha)\mu_2(A)$. We call a convex set $\mathcal{P}$ of measures \textit{everywhere imprecise} if for all $A\in\Sigma$, there are $\mu,\mu'\in\mathcal{P}$ with $\mu(A)\neq\mu'(A)$. Let $S$ be the set of all pairs $(w,\mathcal{P})$ where $w\in W$ and $\mathcal{P}$ is an everywhere imprecise convex set of measures on $(W,\Sigma)$. We think of $w$ as the state of the world and $\mathcal{P}$ as an imprecise representation of an agent's uncertainty, which could be further refined, though never to the point of providing precise subjective probabilities. Thus, let $(w',\mathcal{P}')\sqsubseteq (w,\mathcal{P})$ iff $w'=w$ and $\mathcal{P}'\subseteq\mathcal{P}$. Hence $(S,\sqsubseteq)$ is a poset. For $A\in\Sigma$, let $\psi(A)=\{(w,\mathcal{P})\in S\mid w\in A\}$.  Observe that $\psi$ is a Boolean embedding of $\Sigma$ into $\mathcal{RO}(S,\sqsubseteq)$. For each $r\in [0,1]$, we define a neighborhood function by
\[N_{\geq r}(w,\mathcal{P})=\{\psi(A)\mid \mbox{for all }\mu\in\mathcal{P},\mu(A)\geq r \}.\]
We claim that $N_{\geq r}$ satisfies the two conditions of Proposition \ref{NeighBoxClosure}.\ref{NeighBoxClosure2}:
\begin{itemize}
\item $\boldsymbol{N_{\geq r}}$-\textbf{persistence}: suppose $(w,\mathcal{P}')\sqsubseteq (w,\mathcal{P})$ and $\psi(A)\in N_{\geq r}(w,\mathcal{P})$. Hence for all $\mu\in\mathcal{P}$, we have $\mu(A)\geq r$. Since $(w,\mathcal{P}')\sqsubseteq (w,\mathcal{P})$, we have $\mathcal{P}'\subseteq\mathcal{P}$, so for all $\mu\in\mathcal{P}'$, we have $\mu(A)\geq r$. Thus, $\psi(A)\in N_{\geq r}(w,\mathcal{P}')$.
\item $\boldsymbol{N_{\geq r}}$-\textbf{refinability}: suppose $\psi(A)\not\in N_{\geq r}(w,\mathcal{P})$, so there is some $\mu\in\mathcal{P}$ such that $\mu(A)< r$. Then we claim that the set $\mathcal{P}'=\{\nu\in\mathcal{P}\mid \nu(A)<r\}$ is an everywhere imprecise convex set of probability measures. For convexity, if $\nu_1,\nu_2\in \mathcal{P}'$, then $\alpha\nu_1(A)+(1-\alpha)\nu_2(A)< \alpha r+ (1-\alpha)r = r$, so that $(\alpha\nu_1+(1-\alpha)\nu_2)(A)\in \mathcal{P}'$. For everywhere imprecision, since $\mathcal{P}$ is everywhere imprecise, for each $B\in\Sigma$, there are $\mu_1,\mu_2\in\mathcal{P}$ with $\mu_1(B)\neq\mu_2(B)$. Hence for some $i\in \{1,2\}$, $\mu(B)\neq \mu_i(B)$. Since $\mu(A)<r$, taking $\alpha$ sufficiently large but less than 1, $\alpha\mu(A)+(1-\alpha)\mu_i(A)<r$. Thus, $\alpha\mu+(1-\alpha)\mu_i\in \mathcal{P}'$. Moreover, $\mu_i(B)\neq \mu(B)$ and $\alpha<1$ together imply that $\mu(B)\neq \alpha\mu(B)+(1-\alpha)\mu_i(B)=(\alpha \mu+(1-\alpha)\mu_i)(B)$. It follows that $\mathcal{P}'$ is everywhere imprecise. Thus, $(w,\mathcal{P}')\in S$ and $(w,\mathcal{P}')\sqsubseteq (w,\mathcal{P})$. Moreover, for all $(w,\mathcal{P}'')\sqsubseteq (w,\mathcal{P}')$, $(w,\mathcal{P}'')\not\in N_{\geq r}(w,\mathcal{P}'')$. This establishes $\boldsymbol{N_{\geq r}}$-\textbf{refinability}.
\end{itemize}
Hence $(S,\sqsubseteq, \{N_{\geq r}\}_{r\in [0,1]})$ is a basic neighborhood possibility frame. The modality $\Box_{\geq r}$ associated with $N_{\geq r}$, applied to the image $\psi(A)$ of an event $A\in\Sigma$, expresses the proposition that the probability of $A$ is settled to be at least $r$.\footnote{This model does not attribute to the agent subjective probabilities for higher-order propositions about her own uncertainty, i.e., propositions in $\mathcal{RO}(S,\sqsubseteq)$ not of the form $\psi(A)$ for $A\in\Sigma$, but richer models representing uncertainty about one's own uncertainty can also be defined.} Similarly, we can define neighborhood functions $N_{\leq r}$ and modalities $\Box_{\leq r}$ for expressing that the probability of $A$ is settled to be at most $r$. For  logics with such modalities interpreted using possible world semantics, see \cite[\S~4.1]{Demey2019}.\end{example}

For a converse of Proposition \ref{BasicNeighToBAE}, we can go from any complete BAE to a basic neighborhood possibility frame by extending Theorems \ref{FirstThm}.\ref{FirstThm2}.

\begin{theorem}\label{BAEtoBasicNeigh} \textnormal{For any complete BAE $\mathbb{B}=(B,\{f_i\}_{i\in I})$, define 
\[\mathbb{B}_\mathsf{n}=(B_+,\leq_+,\{N_i\}_{i\in I})\]
where  $(B_+,\leq_+)$ is as in Theorem \ref{FirstThm}.\ref{FirstThm2} and for $x\in B_+$ and $i\in I$, 
\[N_i(x)=\{U\in \mathcal{RO}(B_+,\leq_+)\mid x\leq_+ f_i(\bigvee U)\}.\]
Then:
\begin{enumerate}
\item\label{BAEtoBasicNeigh1}  $\mathbb{B}_\mathsf{n}$ is a basic neighborhood possibility frame;
\item\label{BAEtoBasicNeigh2} $(\mathbb{B}_\mathsf{n})^\mathsf{b}$ is isomorphic to $\mathbb{B}$;
\item\label{BAEtoBasicNeigh3} a basic frame $F$ is isomorphic to $(F^\mathsf{b})_\mathsf{n}$ iff the underlying poset of $F$ is obtained from a complete Boolean lattice by deleting its bottom element.
\end{enumerate}}
\end{theorem} 

\begin{proof} For part \ref{BAEtoBasicNeigh1}, $\boldsymbol{N_i}$-\textbf{persistence} is immediate from the definition of $N_i$. For $\boldsymbol{N_i}$-\textbf{refinability}, if $U\not\in N_i(x)$, so $x\not \leq_+ f_i(\bigvee U)$, then where $x'= x\wedge \neg f_i(\bigvee U)$, we have that for all $x''\leq_+ x'$, $x''\not \leq_+ f_i(\bigvee U)$, so $U\not\in N_i(x'')$. The proof of part \ref{BAEtoBasicNeigh2} extends that of Theorem \ref{FirstThm}.\ref{FirstThm2}, using the map $\varphi: B\to \mathcal{RO}(B_+,\leq_+)$ given by $\varphi(b)=\mathord{\downarrow}_+b:=\{b'\in B_+\mid b'\leq_+b\}$. Observe that 
\begin{eqnarray*} 
\mathord{\downarrow}_+ f_i (b) &=& \{b'\in B_+\mid b'\leq_+ f_i (b)\} \\
 &=&  \{b'\in B_+\mid b'\leq_+ f_i (\bigvee \mathord{\downarrow}_+ b)\}  \\
 &=& \{b'\in B_+\mid \mathord{\downarrow}_+ b\in N_i(b')\} \\
 &=& \Box_{N_i} \mathord{\downarrow}_+ b.
\end{eqnarray*}
For part \ref{BAEtoBasicNeigh3}, the left-to-right direction is by the definition of  $(\cdot)_\mathsf{n}$. For the right-to-left direction, the poset reduct of $F=(S,\sqsubseteq, \{N_i\}_{i\in I})$ is isomorphic to the poset reduct of $(F^\mathsf{b})_\mathsf{n}$ by Theorem \ref{FirstThm}.\ref{FirstThm2}. The isomorphism $\psi: S\to \mathcal{RO}(S,\sqsubseteq)_+$ is given by $\psi(x)=\mathord{\downarrow}x:=\{x'\in S\mid x'\sqsubseteq x\}$. Finally, we claim that for all $U\in \mathcal{RO}(S,\sqsubseteq)$, we have $U\in N_i^F(x)$ iff $\psi[U]\in N_i^{(F^\mathsf{b})_\mathsf{n}}(\psi(x))$.  Indeed, $U\in N_i^{F}(x)$ iff $\mathord{\downarrow}x\subseteq \Box_i U$ iff $\mathord{\downarrow}x\leq_+^{F^\mathsf{b}} \Box_i (\bigvee_{F^\mathsf{b}} \psi[U])$ iff $\psi[U]\in N_i^{(F^\mathsf{b})_\mathsf{n}}(\psi(x))$.\end{proof}

When constructing neighborhood frames, we will typically not build frames satisfying the condition of part \ref{BAEtoBasicNeigh3}. The whole point of using neighborhood possibility frames, as opposed to just working with complete BAEs, is that we can use simple posets---such as trees---equipped with neighborhood functions to realize a BAE that may be less intuitive to define algebraically and reason about directly. 

If we already know how to realize a complete Boolean algebra $B$ as the regular open sets of a nice poset, then we can apply the following.

\begin{proposition}\label{NeighDef} \textnormal{Suppose $B$ is a complete Boolean algebra isomorphic to  \\$\mathcal{RO}(S,\sqsubseteq)$  for some poset $(S,\sqsubseteq)$. Then for any BAE $\mathbb{B}=(B,\{f_i\}_{i\in I})$, there is a basic neighborhood possibility frame $\mathbb{B}_{(S,\sqsubseteq)}=(S,\sqsubseteq, \{N_i\}_{i\in I})$ such that $(\mathbb{B}_{(S,\sqsubseteq)})^\mathsf{b}$ is isomorphic to $\mathbb{B}$.}
\end{proposition}
\begin{proof}We use an isomorphism $\sigma$ from  $\mathcal{RO}(S,\sqsubseteq)$ to $B$ to define neighborhood functions on $S$: for any $x\in S$ and $Z\in \mathcal{RO}(S,\sqsubseteq)$, set \[\mbox{$Z\in N_i(x)$ iff $x\in \sigma^{-1}(f_i (\sigma(Z)))$.}\]
Then it is easy to check that $\boldsymbol{N_i}$-\textbf{persistence} and $\boldsymbol{N_i}$-\textbf{refinability} hold, and $(S,\sqsubseteq,\{N_i\}_{i\in I})^\mathsf{b}$ is isomorphic to $(B,\{f_i\}_{i\in I})$.\end{proof}

\begin{remark} Theorem \ref{BAEtoBasicNeigh} can be extended to an obvious categorical duality when morphisms are defined (see Definition \ref{NeighMorph}), but we omit the details. See \S~\ref{DualEquiv} for a more interesting duality for relational possibility frames. 
\end{remark}

The term ``neighborhood frame'' comes from the topological notion of a \textit{neighborhood system} on a set $S$, which is a function $N:S\to\wp(\wp(S))$ satisfying the following conditions for all $x\in S$ and $U,V\subseteq S$ \cite[p.~56]{Kelley1975}:
\begin{enumerate}
\item filter condition: $N(x)\neq\varnothing$, and $U\cap V\in N(x)$ iff $U\in N(x)$ and $V\in N(x)$;
\item neighborhood of a point condition: if $U\in N(x)$, then $x\in U$;
\item open neighborhood condition: for each $U\in N(x)$, there is a $V\in N(x)$ such that $V\subseteq U$ and for all $y\in V$, $V\in N(y)$.
\end{enumerate}
Recall the correspondence between neighborhood systems on a set and topological spaces on a set: given a neighborhood system, define an open set to be a set which is a neighborhood of each of its points; given a topological space, define a neighborhood of a point to be a superset of an open set containing the point.

\begin{definition} A \textit{neighborhood possibility system} is a basic neighborhood possibility frame $\mathcal{F}=(S,\sqsubseteq,N)$ such that $N$ satisfies conditions 1, 2, and 3 of a neighborhood system for every $U,V\in \mathcal{RO}(S,\sqsubseteq)$.
\end{definition}

Recall that an \textit{interior algebra} \cite{MT44,RS63,Blok1976,Esakia85} is a pair $(B,\Box)$ where $B$ is a Boolean algebra and $\Box$ is an operation on $B$ such that for all $a,b\in B$:
\[\Box a\wedge \Box b = \Box(a\wedge b),\Box a\leq a, \Box a\leq \Box\Box a.\]
Recall the correspondence between topological spaces and \textit{complete and atomic} interior algebras, i.e., whose underlying BA $B$ is complete and atomic: given a topological space $(S,\tau)$, we form a complete and atomic interior algebra $(\wp(S),\mathsf{int})$ where $\mathsf{int}$ is the interior operation in the space $(S,\tau)$; given a complete and atomic interior algebra, we represent the BA as a powerset and define a topology on the underlying set by taking the opens to be the fixpoints of the interior operation $\Box$.

By contrast, neighborhood possibility systems correspond to \textit{complete} interior algebras, i.e., whose underlying BA $B$ is complete but not necessarily atomic.

\begin{proposition} \textnormal{For any basic neighborhood possibility frame \\ $F=(S,\sqsubseteq,N)$, the following are equivalent:
\begin{enumerate}
\item $F$ is a neighborhood possibility system;
\item $F^\mathsf{b}$ is a complete interior algebra.
\end{enumerate}
Moreover, for any complete interior algebra $\mathbb{B}=(B,\Box)$, $\mathbb{B}_\mathsf{n}$ is a neighborhood possibility system.}
\end{proposition}
\begin{proof}[Proof sketch.] It is easy to check that $F$ (resp.~$\mathbb{B}_\mathsf{n}$)  satisfying the first (resp. second, third) condition of neighborhood possibility systems corresponds to $F^\mathsf{b}$ (resp.~$\mathbb{B}$) satisfying the first (resp.~second third) defining equation of interior algebras.\end{proof}

\noindent Thus, neighborhood possibility systems can be regarded as a generalization of neighborhood systems and hence as a generalization of topological spaces (for investigation of interior algebras that are not necessarily atomic, see \cite{Naturman1990}). This generalization of topological spaces can be compared with another generalization of topological spaces from pointfree topology \cite{Picado2012}: \textit{locales}, defined as complete lattices in which binary meet distributes over arbitrary joins (see \S~\ref{IntuitionisticSection} for another characterization of locales as complete Heyting algebras). The set $\{a\in B\mid a=\Box a\}$ of ``open elements'' of any complete interior algebra $(B,\Box)$ forms a locale \cite{MT46}, and every locale can be represented in this way,\footnote{It is stated in \cite[p.~318]{Gabbay2009} as open whether every locale can be represented as the locale of open elements of a complete interior algebra. However, as Guram Bezhanishvili (personal communication) observed, if $L$ is a locale, then Funayama's Theorem (see, e.g., \cite{Bezhanishvili2013}) yields a localic $(\wedge,\bigvee)$-embedding of $L$ into a complete Boolean algebra $B$.  Then the image of $L$ is a relatively complete sublattice of $B$, so it defines an interior operation $\Box$ on $B$ whose algebra of fixpoints is isomorphic to $L$. Thus, $L$ is isomorphic to the algebra of open elements of the complete interior algebra $(B,\Box)$.} it follows that the set \[\{U\in\mathcal{RO}(S,\sqsubseteq)\mid U=\Box_{N} U\}\] of ``open sets'' of any neighborhood possibility system forms a locale, and every locale can be represented in this way.
\begin{theorem} \textnormal{A lattice $L$ is a locale iff $L$ is isomorphic to the lattice of open sets of a neighborhood possibility system.}
\end{theorem}

In order to discuss logics that are incomplete with respect to neighborhood world semantics but complete with respect to neighborhood possibility semantics, let us give the official formal semantics for  $\mathcal{L}(I)$.

\begin{definition} Given a basic neighborhood frame $F=(S,\sqsubseteq, \{N_i\}_{i\in I})$, a \textit{neighborhood possibility model based on $F$} is a pair $M=(F,\pi)$ where $\pi: \mathsf{Prop}\to \mathcal{RO}(S,\sqsubseteq)$. For $\varphi\in\mathcal{L}(I)\Pi$ and $x\in S$, we define $M,x\Vdash \varphi$ and $\llbracket \varphi\rrbracket^M=\{y\in S\mid M,y\Vdash\varphi\}$  as in Definition \ref{PossForcing1} but with the following new clause for $\Box_i$:
\begin{itemize}
\item $M,x\Vdash \Box_i\varphi$ iff $\llbracket \varphi\rrbracket^M\in N_i(x)$.
\end{itemize}
A formula $\varphi$ is \textit{valid} on the frame $F$ if for every model $M$ based on $F$ and every $x\in S$, $M,x\Vdash \varphi$; and $\varphi$ is \textit{valid on a class $\mathsf{K}$} of frames if it is valid on every frame in $\mathsf{K}$. The $\mathcal{L}(I)$-\textit{logic} (resp.~$\mathcal{L}(I)\Pi$-\textit{logic}) \textit{of} $\mathsf{K}$ is the set of all formulas valid on $\mathsf{K}$.\end{definition}

\begin{proposition} \textnormal{The $\mathcal{L}(I)$-logic of any class of basic neighborhood possibility frames is a congruential modal logic.}
\end{proposition}
\begin{proof} Follows from Proposition \ref{BasicNeighToBAE}.
\end{proof}

As in \S~\ref{PropQuantSection}, we can easily detect the difference between world frames and possibility frames using propositional quantification, but now below \textsf{S5} using neighborhood functions. Consider, for example, the standard logic of introspective belief, \textsf{KD45}, which is the smallest normal modal logic containing the axioms
\[\neg\Box\bot,\, \Box p \to \Box\Box p,\,\neg\Box p\to \Box \neg \Box p.\]
\textsf{KD45} is valid on a basic neighborhood possibility frame $F=(S,\sqsubseteq, N)$ iff $N$ satisfies the filter condition and open neighborhood conditions of a neighborhood possibility system, plus the following for all $x\in S$:
\begin{itemize}
\item proper condition: $\varnothing\not\in N(x)$;
\item for each $U\not\in N(x)$, there is a $V\in N(x)$ such that for all $y\in V$, $U\not\in N(y)$.
\end{itemize}
In fact, we will see below that simpler frames suffice for completeness.

\begin{definition} A \textit{uniform normal neighborhood possibility frame} is a basic neighborhood possibility frame $F=(S,\sqsubseteq,N$) such that $N$ satisfies the filter condition and proper condition, and for all $x,y\in S$, $N(x)=N(y)$.
\end{definition}

Let $\mathsf{KD4}^\forall \mathsf{5}\Pi$ \cite{Ding2020} be the extension of $\mathsf{KD45}$ with the $\Pi$ principles for propositional quantification from \S~\ref{PropQuantSection}, plus the $\mathsf{4}^\forall$ axiom:
\[\forall p\Box\varphi\to \Box\forall p\Box\varphi.\]
$\mathsf{KD4}^\forall \mathsf{5}\Pi$ is not complete with respect to any class of basic neighborhood \textit{world} frames for the same reason that $\mathsf{S5}\Pi$ was not: the formula (\ref{BadAx2}) from \S~\ref{PropQuantSection} is valid on every such world frame validating $\mathsf{KD45}$ (in fact, every such world frame validating $\Box\top$), but it is not a theorem of $\mathsf{S5}\Pi$, which extends $\mathsf{KD4}^\forall \mathsf{5}\Pi$. By contrast, Ding \cite{Ding2020} proved that $\mathsf{KD4}^\forall \mathsf{5}\Pi$ is complete with respect to complete proper filter algebras, i.e., complete BAs equipped with a  proper filter for interpreting $\Box$ (i.e., $\Box a=1$ if $a$ is in the filter, and otherwise $\Box a=0$), as well as sound with respect to all complete $\mathsf{KD45}$ BAEs. Combining this with Proposition \ref{BasicNeighToBAE} and Theorem \ref{BAEtoBasicNeigh}.\ref{BAEtoBasicNeigh1}-\ref{BAEtoBasicNeigh2}, we obtain the following.

\begin{theorem}[\cite{Ding2020}]\label{DingThm} \textnormal{$\mathsf{KD4}^\forall \mathsf{5}\Pi$ is the logic of all basic neighborhood possibility frames validating $\mathsf{KD45}$ (resp.~uniform normal neighborhood possibility frames), but $\mathsf{KD4}^\forall \mathsf{5}\Pi$ is not the logic of any class of basic neighborhood world frames.}
\end{theorem}

It is open whether the $\mathcal{L}(I)\Pi$ logics of other classes of basic neighborhood possibility frames are recursively axiomatizable. Let us mention a few salient instances of the question.

\begin{question} Is the $\mathcal{L}\Pi$-logic of all neighborhood possibility systems (i.e., of complete interior algebras) recursively axiomatizable or enumerable? The logic of all basic neighborhood possibility frames? The logic of those \textit{normal} frames in which each $N(x)$ is a proper filter? The logics of other classes?
\end{question}

Let us now consider an example of a congruential modal logic \textit{without propositional quantifiers} that is the logic of a class of basic neighborhood possibility frames but not of any class of basic neighborhood world frames. Consider a bimodal language with modalities we will write as $\Box$ and $Q$ (instead of $\Box_1$ and $\Box_2$). Let $\mathsf{S}$ be the smallest congruential logic for this language containing $\Box\top$ and
\begin{equation}
p\to \big(\Diamond (p\wedge Qp)\wedge \Diamond (p\wedge\neg Qp)\big).\tag{\textsc{Split}}\label{Split}
\end{equation}
\begin{remark}\label{QArith}As shown in \cite[\S~7]{DH2020}, there is a natural \textit{arithmetic} interpretation of (\ref{Split}) where $\Diamond$ is interpreted as consistency in Peano arithmetic and $Q$ is interpreted as an arithmetic predicate defined by Shavrukov and Visser \cite{Shavrukov2014}.\end{remark} Let $\mathsf{ES}$ be the smallest congruential logic extending $\mathsf{S}$ with the following axioms:
\[ 
  \Box p \to p, \ \Box p \to \Box\Box p, \neg\Box p \to \Box\neg\Box p, \ \Box(p \leftrightarrow q) \to \Box(Qp \leftrightarrow Qq).
\]
Let $\mathsf{EST}$ be the smallest congruential modal logic extending $\mathsf{ES}$ by the T axiom for $Q$: $Qp \to p$. Then $\mathsf{EST}$ is sound with respect to the arithmetic interpretation mentioned in Remark \ref{QArith}. Yet basic neighborhood \textit{world} frames cannot give semantics for these logics.

\begin{proposition}\label{Snotvalid} \textnormal{$\mathsf{S}$ is not valid on any basic neighborhood world frame.}
\end{proposition}

\begin{proof} Suppose $F=\langle W, N_\Box,N_Q\rangle$ validates $\mathsf{S}$.
Define a model $M=( F,\pi )$ such that for some
$w\in W$, $\pi(p)=\{w\}$, so $M,w\Vdash p$. Then since $F$
validates (\ref{Split}), we have $M,w\Vdash \Diamond (p\wedge
Qp)\wedge \Diamond (p\wedge\neg Qp)$, i.e., \[\llbracket \neg (p\wedge
Qp)\rrbracket^M\not\in N_\Box(w)\mbox{ and }\llbracket \neg (p\wedge \neg
Qp)\rrbracket^M\not\in N_\Box(w).\] Since $\pi(p)$ is a singleton set, either
$\llbracket p\wedge Qp\rrbracket^M=\varnothing$ or $\llbracket p\wedge
\neg Qp\rrbracket^M=\varnothing$, so $\llbracket \neg( p\wedge
Qp)\rrbracket^M=W$ or $\llbracket \neg (p\wedge \neg
Qp)\rrbracket^M=W$. Combining the previous two steps, we have
$W\not\in N(w)$, which contradicts the fact that $F$ validates
$\Box\top$. 
\end{proof}
In fact, it is not hard to check the following stronger result.

\begin{proposition}[\cite{DH2020}]\textnormal{$\mathsf{S}$ is not valid on any BAE whose underlying BA is \textit{atomic}.}
\end{proposition}

By contrast, we have the following positive result when moving from basic neighborhood world frames to basic neighborhood possibility frames.

\begin{proposition}\label{ESTvalid} \textnormal{\textsf{EST} can be validated by a basic neighborhood possibility frame.}
\end{proposition}
\begin{proof}[Proof sketch] Consider from Examples \ref{BinaryTree0}-\ref{BinaryTree} the full infinite binary tree regarded as a poset ${(S,\sqsubseteq)}$. For $x\in S$, let $N_\Box(x)=\{S\}$. Let $\mathsf{Par}(x)$ be the parent of $x$ in the tree and $x0$ and $x1$  the two
  extensions of $x$ by $0$ and $1$, respectively. Define $N_Q$ inductively by:
  \begin{align*}
    N_Q(\epsilon) &=  \varnothing; \\
    N_Q(x0) &= N_Q(x)\cup \{U\in \mathcal{RO}(S,\sqsubseteq)\mid x\in U,\,\mathsf{Par}(x)\not\in U\}; \\
    N_Q(x1) &= N_Q(x).
  \end{align*}
  As shown in \cite{DH2020}, $(S,\sqsubseteq,N_\Box,N_Q)$ is a basic neighborhood possibility frame that validates \textsf{EST}.\end{proof}
  
  In fact, these logics are not only consistent but also possibility complete.

\begin{theorem}[\cite{DH2020}]\label{DHthm} \textnormal{Each of the world-inconsistent logics \textsf{S}, \textsf{ES}, and \textsf{EST} is complete with respect to the class of basic neighborhood possibility frames that validate the logic.}
\end{theorem}

We conclude with two open problems suggested by the preceding theorem.

\begin{question} Is there a consistent congruential modal logic that is not valid on any basic neighborhood world frame, or at least not the logic of any class of such frames, axiomatized with only  one modality, one propositional variable, and modal depth two \cite{DH2020}? What other logics are incomplete with respect to basic neighborhood world frames but complete with respect to basic neighborhood possibility frames?
\end{question}

\subsubsection{Tree completeness}\label{TreeSection}

While Theorem \ref{BAEtoBasicNeigh} tells us that any congruential modal logic complete with respect to complete BAEs is automatically complete with respect to basic neighborhood possibility frames, more interesting would be completeness with respect to basic neighborhood possibility frames based on special posets such as trees.

\begin{definition} A congruential modal logic $\mathsf{L}$ is \textit{neighborhood tree complete} if there is a class of basic neighborhood possibility frames, whose underlying posets are trees, such that $\mathsf{L}$ is the logic of the class.
\end{definition}

The following observation gives us tree completeness for many standard logics.

\begin{proposition}[\cite{DH2020}]\label{TreeThm} \textnormal{If a congruential modal logic $\mathsf{L}$ is the logic of a class of BAEs based on the MacNeille completion of the countable atomless Boolean algebra, then $\mathsf{L}$ is the logic of a class of basic neighborhood possibility frames based on the tree $2^{<\omega}$.}
\end{proposition}
\begin{proof} As in Example \ref{BinaryTree}, $\mathcal{RO}(2^{<\omega})$ is isomorphic to the MacNeille completion of the countable atomless Boolean algebra. Now apply Proposition~\ref{NeighDef}.\end{proof}

Taking the MacNeille completion of the Lindenbaum algebra of a logic is a standard technique for proving the completeness of a logic with respect to complete algebras. In the typical case, the Boolean reduct of the Lindenbaum algebra is (up to isomorphism) the countable atomless Boolean algebra. Thus, Theorem \ref{TreeThm} assures us that when this standard technique is applicable, we automatically have neighborhood tree completeness as well. For example, given the proof of Theorem \ref{DHthm} in \cite{DH2020}, Proposition \ref{TreeThm} yields the following.

\begin{theorem}[\cite{DH2020}]\label{DHthm2} \textnormal{Each of the world-inconsistent logics \textsf{S}, \textsf{ES}, and \textsf{EST} is neighborhood tree complete.}
\end{theorem}

\subsubsection{General frames}\label{GenFrame1}

As basic neighborhood possibility frames can only be used to represent complete BAEs, it is natural to introduce a more general notion of neighborhood possibility frames that can be used to represent arbitrary BAEs, as we do below. This is a straightforward development of the theory of possibility frames in \S~\ref{PropFilters}, but we state the main definitions and results for the sake of completeness.

\begin{remark}In possible world semantics, Do\v{s}en's \cite{Dosen1989} general frames, equivalent to our general neighborhood \textit{world} frames below, already allow the representation of arbitrary BAEs and indeed provide a category of frames dually equivalent to that of BAEs (morphisms are discussed in Definition \ref{NeighMorph} below). Once again an advantage of general \textit{possibility} frames is that they provide such representation and categorical duality theorems choice free.
\end{remark}

\begin{definition}\label{GenNeighPossFrame} A (\textit{general}) \textit{neighborhood possibility frame} is a tuple \\ $\mathcal{F}= ( S,\sqsubseteq,P,\{N_i\}_{i\in I}) $ such that:
\begin{enumerate}
\item $( S,\sqsubseteq, P)$ is a possibility frame as in Definition \ref{PosFrameDef};
\item $N_i:S\to \wp(P)$;
\item for all $X\in P$,  $\Box_{N_i}X\in P$.  
\end{enumerate}
When $P=\mathcal{RO}(S,\sqsubseteq)$, we say that $\mathcal{F}$ is a \textit{full neighborhood possibility frame}. When $\sqsubseteq$ is the identity relation, we say that $\mathcal{F}$ is a \textit{neighborhood world frame}.
\end{definition}

\begin{example} Recall from Example \ref{ImpreciseEx} the basic neighborhood possibility frame $(S,\sqsubseteq,\{N_{\geq r}\}_{r\in [0,1]})$ for a given measurable space $(W,\Sigma)$.  For $A\subseteq W$, let $\chi(A)=\{(w,\mathcal{P})\in S\mid w\in A\}$, and observe that $\chi$ is a Boolean embedding of $\wp(W)$ into $\mathcal{RO}(S,\sqsubseteq)$. Thus, each element of $\wp(W)$ yields a proposition in $\mathcal{RO}(S,\sqsubseteq)$. Yet if $\wp(W)\neq\Sigma$, then not every element of $\wp(W)$ is an admissible proposition in the sense of belonging to $\Sigma$. Thus, we may wish to restrict what counts as a proposition in our neighborhood frame accordingly: let $P$ be the smallest subset of $\mathcal{RO}(S,\sqsubseteq)$ that includes $\{\chi(A)\mid A\in\Sigma\}$ and is closed under $\cap$, $\neg$, and $\Box_{N_\geq r}$ for each $r\in [0,1]$. Then $(S,\sqsubseteq,\{N_\geq r\}_{r\in [0,1]},P)$ is a general possibility frame. In this frame, we have a proposition corresponding to each proposition in $\Sigma$, plus propositions generated using the $\Box_{N_\geq r}$ modalities and Boolean operations.\end{example}

One may identify full frames as in Definition \ref{GenNeighPossFrame} and basic frames as in Definition \ref{NeighPossFrame}, by the following consequence of Proposition \ref{NeighBoxClosure}.

\begin{proposition} \textnormal{The map $(\cdot)^\sharp$ sending each basic frame \[F=( S,\sqsubseteq, \{N_i\}_{i\in I})\] to its \textit{associated full frame} \[F^\sharp=( S,\sqsubseteq, \mathcal{RO}(S,\sqsubseteq), \{N_i\}_{i\in I})\] is a bijection between the class of basic neighborhood possibility frames (resp.~basic neighborhood world frames) and the class of full neighborhood possibility frames (resp.~full neighborhood world frames).}
\end{proposition}

\begin{remark}\label{FullNeighWorld} Thus, we may identify full neighborhood world frames \\ $(W,=,\wp(W),\{N_i\}_{i\in I})$ with basic neighborhood world frames and hence the usual neighborhood frames $(W,\{N_i\}_{i\in I})$ of possible world semantics as in Remark \ref{BasicNeighWorld}.
\end{remark}

Each general frame gives rise to a BAE in the obvious way.

\begin{proposition} \textnormal{Given $\mathcal{F}= ( S,\sqsubseteq,P,\{N_i\}_{i\in I}) $, a  neighborhood possibility frame, define  $\mathcal{F}^\mathsf{b}=( \mathsf{BA}(P), \{\Box_{N_i}\}_{i\in I})$ where $\mathsf{BA}(P)$ is the Boolean subalgebra of $\mathcal{RO}(S,\sqsubseteq)$ based on $P$. Then $\mathcal{F}^\mathsf{b}$ is a BAE.}\end{proposition}
Thus, we can give formal semantics for $\mathcal{L}(I)$ and congruential modal logics using general neighborhood possibility frames, where models interpret propositional variables as elements of $P$ instead of arbitrary elements of $\mathcal{RO}(S,\sqsubseteq)$.

We can go in the converse direction from BAEs to frames as follows.

\begin{proposition}\label{BAEstoFrames} \textnormal{Given any BAE $\mathbb{B}=( B, \{f_i\}_{i\in I})$, define  
\[\mathbb{B}_\mathsf{f}=(B_\mathsf{f}, \{N_i\}_{i\in I})\mbox{ and }\mathbb{B}_\mathsf{g}=(B_\mathsf{g}, \{N_i\}_{i\in I}) \]
where $B_\mathsf{f}$ and $B_\mathsf{g}$ are as in Proposition \ref{FiltFrameDef}, and for $F\in \mathsf{PropFilt}(B)$, \[N_i(F)=\{ \widehat{a}\mid \Box_i a\in F\}.\]
Then:
\begin{enumerate}
\item $\mathbb{B}_\mathsf{f}$ and $\mathbb{B}_\mathsf{g}$ are neighborhood possibility frames;
\item\label{BAEstoFrames2} $(\mathbb{B}_\mathsf{f})^\mathsf{b}$ is (up to isomorphism) the canonical extension of $\mathbb{B}$;
\item $(\mathbb{B}_\mathsf{g})^\mathsf{b}$ is isomorphic to $\mathbb{B}$.
\end{enumerate}}
\end{proposition}
\begin{proof}[Proof sketch.] The proof is an easy extension of the proof of Proposition \ref{FiltFrameDef} taking into account $\Box_i$ and $N_i$.
\end{proof}

Frames in the image of the $(\cdot)_\mathsf{g}$ map can be characterized as in Proposition~\ref{SepReal}.

\begin{proposition}\textnormal{Let $\mathcal{F}= ( S,\sqsubseteq,P, \{N_i\}_{i\in I} ) $ be a neighborhood possibility frame. Then  $\mathcal{F}$ is isomorphic to $(\mathcal{F}^\mathsf{b})_\mathsf{g}$ iff $( S,\sqsubseteq,P )$ is a filter-descriptive possibility frame (as in Proposition \ref{SepReal}). We also call such an $\mathcal{F}$ \textit{filter-descriptive}.}
\end{proposition}

To obtain a full categorical duality between filter-descriptive neighborhood frames and BAEs, we need morphisms that combine conditions from the standard p-morphisms in modal logic and neighborhood morphisms as in \cite{Dosen1989}.

\begin{definition}\label{NeighMorph} Given neighborhood possibility frames $\mathcal{F}= ( S,\sqsubseteq,P, \{N_i\}_{i\in I} ) $ and $\mathcal{F}'= ( S',\sqsubseteq',P', \{N'_i\}_{i\in I}) $, a \textit{neighborhood p-morphism} from $\mathcal{F}$ to $\mathcal{F}'$ is a map $h:S\to S'$ satisfying the following conditions for all $x,y\in S$ and $y'\in S'$:
\begin{enumerate}
\item for all $ Z'\in P'$, $h^{-1}[Z']\in P$;
\item\label{SqForthCon} \SqForth{}: if $y\sqsubseteq x$, then $h(y)\sqsubseteq h(x)$;
\item \SqBack{}: if $y'\sqsubseteq' h( x )$, then $\exists y$: $y\sqsubseteq x $ and $h(y)= y'$;
\item for all $Z'\in P'$ and $i\in I$, $h^{-1}[Z']\in N_i(x)$ iff $Z'\in N'_i(h(x))$.
\end{enumerate}
\end{definition}

\begin{proposition}\label{NeighBAE}$\,$\textnormal{
\begin{enumerate}
\item\label{NeighBAE1} Filter-descriptive possibility frames together with neighborhood \\ p-morphisms form a category, \textbf{FiltNeighPoss}.
\item\label{NeighBAE2} BAEs with BAE-homomorphisms form a category, \textbf{BAE}.
\end{enumerate}}
\end{proposition}
\begin{proof}[Proof sketch.] For part \ref{NeighBAE1}, it is easy to check that the composition of two neighborhood p-morphisms is  a neighborhood p-morphism (cf.~\cite{Holliday2018} for the analogous proof for p-morphisms between relational possibility frames and \cite{Dosen1989} for p-morphisms between neighborhood world frames). Part \ref{NeighBAE2} is well known (cf.~\cite{Dosen1989}).\end{proof}

\begin{theorem}\textnormal{(ZF) \textbf{FiltNeighPoss} is dually equivalent to \textbf{BAE}.}
\end{theorem}
\begin{proof}[Proof sketch.] The proof is an easy extension of the proof of Theorem \ref{FiltPossDuality} taking into account the  neighborhood functions and congruential modalities.
\end{proof}

\subsection{Relational frames}\label{RelationalFrameSection}

In this section, we narrow our focus from semantics for arbitrary congruential modal logics, as in \S~\ref{NeighSection}, to semantics for normal modal logics in particular.

\begin{definition}\label{NormalLogic} A \textit{normal modal logic} is a congruential modal logic $\mathsf{L}$ as in Definition \ref{CongLog} satisfying the following conditions:
\begin{enumerate}
\item $(\Box p\wedge \Box q)\leftrightarrow \Box(p\wedge q)\in \mathsf{L}$;
\item necessitation rule: if $\varphi\in\mathsf{L}$, then $\Box_i\varphi\in\mathsf{L}$.
\end{enumerate}
\end{definition}
\noindent If we restrict attention to normal modal logics, we can use a possibility-semantic version of a very intuitive style of modal semantics known as \textit{relational semantics} (cf.~\cite{Blackburn2001}). In this section---the longest of the chapter---we cover a number of aspects of relational possibility semantics: 

In \S~\ref{AccessSection}, we define relational possibility frames and discuss the key issue of the interaction between a modal accessibility relation $R$ and the refinement relation $\sqsubseteq$. Every Kripke frame gives rise to a possibility frame (by taking $\sqsubseteq$ to be the identity relation on worlds), but the converse does not hold, since possibility frames can give rise to non-atomic modal algebras.

In \S~\ref{PossKripkeModels}, we briefly focus on models---frames with a fixed valuation of propositional variables---and show how possibility models can be turned into Kripke models. We then return to our focus on frames for the rest of the section.

In \S\S~\ref{VtoFrame}-\ref{DualEquiv}, we relate relational possibility frames to Boolean algebras with $\Box_i$ operations that distribute over arbitrary, not only finite, conjunctions (meets). The moral is that from an algebraic point of view, the essence of relational semantics for modal logic  in its most basic form (i.e., before the step to general frames or topological frames that can realize all modal algebras)  is the distribution of $\Box_i$ over arbitrary conjunctions---or dually, the distribution of $\Diamond_i$ over arbitrary disjunctions---not the atomicity that comes from assuming that accessibility relations relate possible \textit{worlds}.  

In \S\S~\ref{QuasiSection}, we briefly introduce possibility semantics for \textit{quasi-normal} modal logics, a generalization of normal modal logics not requiring the necessitation rule (see Definition \ref{NormalLogic}), which have important applications in provability logic \cite{Boolos1993}.

Provability logic is also the theme of \S\S~\ref{FullRelFrameSection}-\ref{PrincRelFrame}. In \S~\ref{FullRelFrameSection}, we use examples inspired by provability logic to illustrate the greater generality of full relational possibility frames over Kripke frames, by identifying provability-logical principles that cannot be validated by the latter frames but can by the former frames. Then in \S~\ref{PrincRelFrame}, we use  possibility semantics to deepen a well-known Kripke incompleteness result for bimodal provability logic, showing that this incompleteness does not depend on the atomicity of the dual algebras of Kripke frames and arises only from the assumption that $\Box_i$ distributes over arbitrary conjunctions. Thus, possibility semantics can be used both to overcome some incompleteness results (\S~\ref{FullRelFrameSection}) and to strengthen other incompleteness results (\S~\ref{PrincRelFrame}).

Finally, in \S~\ref{GenRelFrame}, we briefly sketch the theory of \textit{filter-descriptive} relational possibility frames, the relational analogue of the filter-descriptive neighborhood possibility frames of \S~\ref{GenFrame1}, extending the theory of filter-descriptive (non-modal) possibility frames from \S~\ref{PropFilters}. These filter-descriptive relational frames provide a fully general, choice-free relational semantics for normal modal logics.

Before introducing relational possibility frames for normal modal logics in \S~\ref{AccessSection}, we first recall the relevant algebras for such logics.

\begin{definition}\label{BAOdef} A \textit{Boolean algebra with} (\textit{unary}) \textit{operators} (BAOs) is a BAE $\mathbb{B}=( B, \{\Box_i\}_{i\in I} )$ as in Definition~\ref{BAEdef} in which each $\Box_i$ distributes over all finite meets.\end{definition}

\begin{remark} The term `operator' usually refers to the $\Diamond_i$ operation that distributes over all finite joins, whereas the $\Box_i$ operation that distributes over all finite meets is called the `dual operator'. It turns out to be convenient for us to take the $\Box_i$ operation as primitive and $\Diamond_i$ defined by $\Diamond_ia :=\neg \Box_i\neg a$.
\end{remark}

\subsubsection{Accessibility relations}\label{AccessSection}

The restriction to normal modal logics allows us to replace neighborhood functions $N_i$ with more easily visualized \textit{accessibility relations} $R_i$. As in \S~\ref{BasicNeighSection}, we could first define relational possibility \textit{foundations},  then \textit{basic} relational possibility frames, and finally \textit{general} relational possibility frames---but instead we will cut straight to the most general notion.

\begin{definition}\label{RellPossFrame} A (\textit{general}) \textit{relational possibility frame} is a tuple \\ $\mathcal{F}={( S,\sqsubseteq,P, \{R_i\}_{i\in I})}$ such that:
\begin{enumerate}
\item $(S,\sqsubseteq, P)$ is a possibility frame as in Definition \ref{PosFrameDef};
\item $R_i$ is a binary relation on $S$;
\item for all $Z\in P$, we have $\Box_iZ\in P$, where
\[\Box_i Z=\{x\in S\mid R_i(x)\subseteq Z\}\mbox{ with }R_i(x)=\{y\in S\mid xR_iy\}.\]
\end{enumerate}
When $P=\mathcal{RO}(S,\sqsubseteq)$, we say that $\mathcal{F}$ is a \textit{full} relational possibility frame. When $\sqsubseteq$ is the identity relation, we say that $\mathcal{F}$ is a \textit{relational world frame}.

A \textit{possibility model} based on $\mathcal{F}$ is a pair $\mathcal{M}=(\mathcal{F},\pi)$ where $\pi:\mathsf{Prop}\to P$. For $\varphi\in\mathcal{L}(I)$ and $x\in S$, we define $\mathcal{M},x\Vdash\varphi$ recursively as in Definition \ref{PossForcing1} but with the following new clause for $\Box_i$:
\begin{itemize}
\item $\mathcal{M},x\Vdash\Box_i\varphi$ iff for all $y\in S$, if $xR_iy$ then $\mathcal{M},y\Vdash\varphi$.
\end{itemize}
Validity of a formula on a frame or class of frames is defined as in Definition \ref{PossForcing1}, as is the logic of a class of frames.
\end{definition}

\begin{proposition}\label{BAOfromFrame} \textnormal{For a relational possibility frame $\mathcal{F}$, define the algebra $\mathcal{F}^\mathsf{b}=(\mathsf{BA}(P),\{\Box_i\}_{i\in I})$ where $\mathsf{BA}(P)$ is the Boolean subalgebra of $\mathcal{RO}(S,\sqsubseteq)$ based on $P$. Then $\mathcal{F}^\mathsf{b}$ is a BAO.}
\end{proposition}
\begin{proof} Immediate from Definitions \ref{BAOdef} and \ref{RellPossFrame}, using the fact that $\mathcal{RO}(S,\sqsubseteq)$ is a Boolean algebra (Theorem \ref{FirstThm}.\ref{FirstThm1}).
\end{proof}

\begin{remark}\label{FullKripkeRemark} \textit{Full} relational \textit{world} frames $\mathcal{F}=(W,=,\wp(W),\{R_i\}_{i\in I})$ may be regarded as the standard Kripke frames \cite{Kripke1963} from possible world semantics. The key difference between full world frames and full possibility frames is that the former can realize only atomic algebras as $\mathcal{F}^\mathsf{b}$. In this chapter, we focus mainly on this difference at the level of frames and algebras. For relations between possibility \textit{models} and world \textit{models}, see \S~\ref{PossKripkeModels} and \cite{HT2017}.
\end{remark}

The following lemma gives convenient sufficient conditions for $\mathcal{RO}(S,\sqsubseteq)$ to be closed under $\Box_i$. The first two are familiar from semantics for intuitionistic modal logic \cite{Wolter1997}. We give necessary and sufficient conditions in Appendix \ref{NecSucc}.

\begin{lemma}[\cite{Holliday2018}]\label{SufficientForFull} \textnormal{For a poset $( S,\sqsubseteq )$ and family $\{R_i\}_{i\in I}$ of binary relations on $S$, suppose the following conditions hold for each $i\in I$:
\begin{itemize}
\item \upR{} -- if $x'\sqsubseteq x$ and $x'R_iy'$, then $xR_iy'$ (see Figure \ref{upRFig}); 
\item \Rdown{} -- if $xR_iy$ and $y'\sqsubseteq y$, then $xR_iy'$ (see Figure \ref{RdownFig});
\item \Rref{} -- if $xR_iy$, then $\exists x'\sqsubseteq x$ $\forall x''\sqsubseteq x'$ $\exists y'\sqsubseteq y$: $x'' R_i y'$ (see Figure~\ref{Rreffig}).
\end{itemize}
Then:
\begin{enumerate}
\item $\mathcal{RO}(S,\sqsubseteq)$ is closed under $\Box_i$;
\item $( S,\sqsubseteq, \mathcal{RO}(S,\sqsubseteq), \{R_i\}_{i\in I})$ is a full relational possibility frame;
\item $( S,\sqsubseteq, P, \{R_i\}_{i\in I})$ is a relational possibility frame for any $P\subseteq \mathcal{RO}(S,\sqsubseteq)$ closed under $\neg$ and~$\cap$.
\end{enumerate}}
\end{lemma}

 \begin{figure}[H]
\begin{center}
\begin{tikzpicture}[->,>=stealth',shorten >=1pt,shorten <=1pt, auto,node
distance=2cm,thick,every loop/.style={<-,shorten <=1pt}]
\tikzstyle{every state}=[fill=gray!20,draw=none,text=black]

\node (x-up) at (0,0) {{$x$}};
\node (x) at (0,-1.5) {{$x'$}};
\node (y) at (2,-1.5) {{$y'$}};

\node at (3,-1.5) {{\textit{$\Rightarrow$}}};

\path (x) edge[dashed,->] node {{}} (y);
\path (x-up) edge[->] node {{}} (x);

\node (x-up') at (4,0) {{$x$}};
\node (x') at (4,-1.5) {{$x'$}};
\node (y') at (6,-1.5) {{$y'$}};

\path (x') edge[dashed,->] node {{}} (y');
\path (x-up') edge[->] node {{}} (x');
\path (x-up') edge[dashed,->] node {{}} (y');

\end{tikzpicture}
\end{center}
\caption{the \upR{} condition. Given $x'R_iy'$, we may go \textit{up} in the first coordinate to any $x$ above $x'$ to obtain $xR_iy'$.}\label{upRFig}
\end{figure}
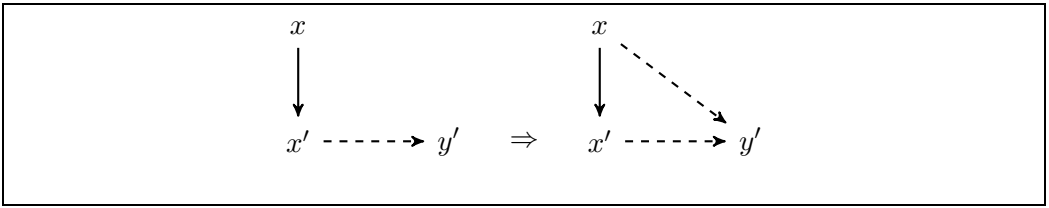

 \begin{figure}[H]
\begin{center}
\begin{tikzpicture}[->,>=stealth',shorten >=1pt,shorten <=1pt, auto,node
distance=2cm,thick,every loop/.style={<-,shorten <=1pt}]
\tikzstyle{every state}=[fill=gray!20,draw=none,text=black]

\node (x) at (0,-1.5) {{$x$}};
\node (y) at (2,-1.5) {{$y$}};
\node (y-down) at (2,-3) {{$y'$}};
\node at (3,-1.5) {{\textit{$\Rightarrow$}}};

\path (x) edge[dashed,->] node {{}} (y);

\path (y) edge[->] node {{}} (y-down);

\node (x') at (4,-1.5) {{$x$}};
\node (y') at (6,-1.5) {{$y$}};
\node (y-down') at (6,-3) {{$y'$}};

\path (x') edge[dashed,->] node {{}} (y');
\path (y') edge[->] node {{}} (y-down');
\path (x') edge[dashed,->] node {{}} (y-down');

\end{tikzpicture}
\end{center}
\caption{the \Rdown{} condition. Given $xR_iy$, we may go \textit{down} in the second coordinate to any $y'$ below $y$ to obtain $xR_iy'$.}\label{RdownFig}
\end{figure}
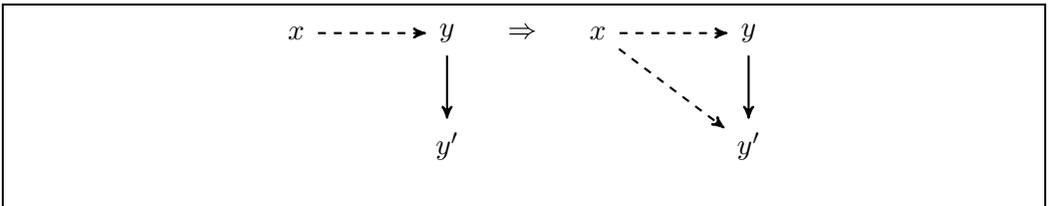

\begin{figure}[H]
 \begin{center}
\begin{tikzpicture}[->,>=stealth',shorten >=1pt,shorten <=1pt, auto,node
distance=2cm,thick,every loop/.style={<-,shorten <=1pt}]
\tikzstyle{every state}=[fill=gray!20,draw=none,text=black]

\node (x0) at (-5,0) {{$x$}};
\node (y0) at (-2,0) {{$y$}};

\path (x0) edge[dashed,->] node {{}} (y0);

\node at (-1,0) {{\textit{$\Rightarrow$}}};

\node (x) at (0,0) {{$x$}};
\node (y) at (3,0) {{$y$}};

\node (x') at (0,-1.5) {{$x'$}};
\node (x'') at (0,-3) {{$x''$}};
\node (y'') at (3,-3) {{$y'$}};

\path (x') edge[<-] node {{ a.$\,\exists\;$}} (x);
\path (x) edge[dashed,->] node {{}} (y);
\path (y) edge[->] node {{ c.$\,\exists$}} (y'');
\path (y'') edge[dashed,<-] node {{}} (x'');
\path (x'') edge[<-] node {{ b.$\,\forall\;$ }} (x');

\end{tikzpicture}
\end{center}
\caption{the \Rref{} condition. The quantifiers in the diagram correspond to the quantifiers in the definition of \Rref{}  in Lemma \ref{SufficientForFull}.}\label{Rreffig}
\end{figure}
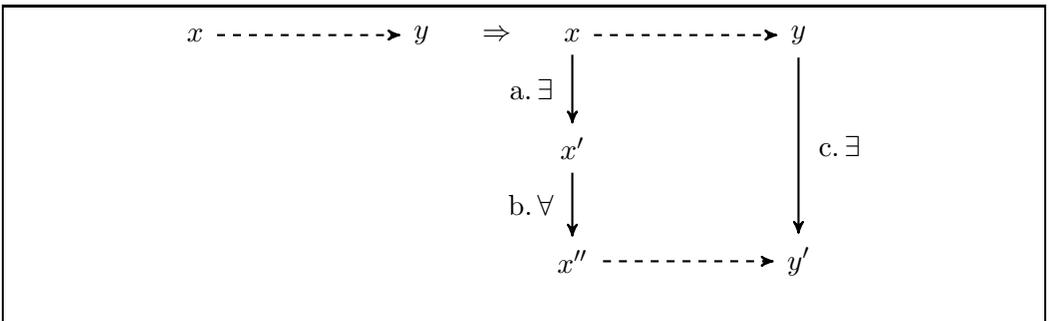

Note our naming convention: we use the name `\upR{}' instead of `$\boldsymbol{R}$\textbf{-up}' because we are going up in the position before `$R$': from $x'R_iy'$ to $xR_iy'$ where $x$ is above $x'$; by contrast, we use the name `\Rdown' because  we are going down in the position after `$R$': from $xR_iy$ to $xR_iy'$ where $y'$ is below $y$. 

A key fact about \Rdown{} is that it allows us to simplify the interpretation of $\Diamond$. Defining $\Diamond_i Z$ as $\neg\Box_i\neg Z$, we have
\begin{equation}\Diamond_iZ=\{x\in S\mid \forall x'\sqsubseteq x\, \exists y'\in R_i(x')\,\exists y''\sqsubseteq y': y''\in Z\}.\label{DiamondDef}\end{equation}
However, with \Rdown{} we achieve the following simplication.
\begin{lemma} \textnormal{Let $\mathcal{F}={( S,\sqsubseteq,P, \{R_i\}_{i\in I})}$ be a possibility frame satisfying \Rdown{}. Then for all $Z\in P$:
\[\Diamond_iZ=\{x\in S\mid \forall x'\sqsubseteq x\, \exists z'\in R_i(x'): z'\in Z\}.\]}
\end{lemma}
\begin{proof} Let $z'$ be $y''$ in (\ref{DiamondDef}). By \Rdown{}, together $y'\in R_i(x')$ and $y''\sqsubseteq y'$ imply $y''\in R_i(x')$.\end{proof}

\begin{example}\label{SeaBattle} We begin with a simple example of a finite possibility frame. Although finite frames do not generate any new algebras compared to possible world frames, the frames themselves may be interesting structures. We give a temporal frame inspired by Aristotle's Sea Battle argument (see, e.g., \cite[p.~35]{Fitting1998}):
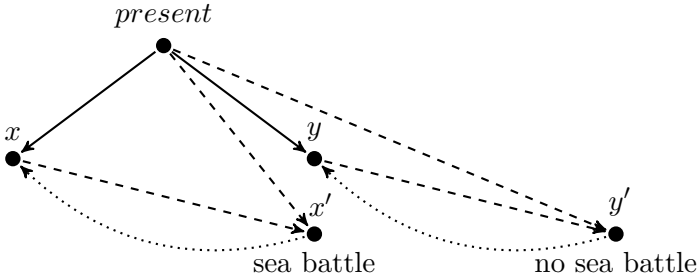
\begin{center}
\begin{tikzpicture}[->,>=stealth',shorten >=1pt,shorten <=1pt, auto,node
distance=2cm,thick,every loop/.style={<-,shorten <=1pt}]
\tikzstyle{every state}=[fill=gray!20,draw=none,text=black]

\node[circle,draw=black!100,fill=black!100, label=above:$present$,inner sep=0pt,minimum size=.175cm] (t) at (2,1.5) {{$$}};

\node[circle,draw=black!100,fill=black!100, label=above:$x$,inner sep=0pt,minimum size=.175cm] (x) at (0,0) {{$$}};
\node[circle,draw=black!100,fill=black!100, label=above:$y$,inner sep=0pt,minimum size=.175cm] (y) at (4,0) {{$$}};

\node[circle,draw=black!100,fill=black!100, label=above:\;\,$x'$,label=below:sea battle,inner sep=0pt,minimum size=.175cm] (x') at (4,-1) {{$$}};
\node[circle,draw=black!100,fill=black!100, label=above:\;$y'$,label=below:no sea battle,inner sep=0pt,minimum size=.175cm] (y') at (8,-1) {{$$}};

\path (t) edge[->] node {{}} (x);
\path (t) edge[->] node {{}} (y);

\path (t) edge[dashed,->] node {{}} (x');
\path (t) edge[dashed,->] node {{}} (y');

\path (x) edge[dashed,->] node {{}} (x');
\path (y) edge[dashed,->] node {{}} (y');

\path (x') edge[dotted,bend left,->] node {{}} (x);
\path (y') edge[dotted,bend left,->] node {{}} (y);

\end{tikzpicture}
\end{center}
Solid arrows represent the refinement relation; dashed arrows represent the future accessibility relation $R_{f}$ used to interpret future modalities $\Box_f$ and $\Diamond_f$ (\textit{henceforth} and \textit{sometime in the future}, respectively); dotted arrows represent the past accessibility relation $R_{p}$ used to interpret past modalities $\Box_p$ and $\Diamond_p$ (\textit{hitherto} and \textit{sometime in the past}, respectively). It is easy to check that \upR{}, \Rdown{}, and \Rref{} hold for $R_f$ and $R_p$. For example, since  $x\sqsubseteq present$ and $xR_fx'$, \upR{} requires that $presentR_fx'$, which indeed holds. The present moment is represented by the possibility labeled `$present$'. There are two possible refinements of the present, $x$ and $y$, but we suppose that neither is currently realized. There are two associated future possibilities for what happens tomorrow: one ($x'$) in which there is sea battle, and one ($y'$) in which there is no sea battle. 

Define a model $\mathcal{M}$ based on this frame where a propositional variable $s$, expressing that there is a sea battle, is true only at $x'$. As we are considering possibility semantics for \textit{classical} logic, either there will be a sea battle tomorrow or there won't be: $\mathcal{M},present\Vdash \Diamond_f s\vee\neg \Diamond_f s$. However, the present does not settle that there will be sea battle tomorrow, and the present does not settle that there won't be a sea battle tomorrow: $\mathcal{M},present\nVdash \Diamond_f s$, since $y\sqsubseteq present$ and $\mathcal{M},y\Vdash \neg\Diamond_f s$, and $\mathcal{M},present\nVdash \neg \Diamond_f  s$, since $x\sqsubseteq present$ and $\mathcal{M},x\Vdash\Diamond_f s$. Thus, the future is presently open. Yet if there is a sea battle, so $x'$ is realized, then the past will turn out to be $x$, in which there would be a future sea battle, whereas if there is no sea battle, so $y'$ is realized, then the past will turn out to be $y$, in which there would be no future sea battle. Come tomorrow, we might say, ``the past is not what it used to be.'' Arguably this picture resolves some puzzles about retrospective assessment of statements of future contingents (cf.~\cite{MacFarlane2003}).\end{example}

\begin{example}\label{TempEx} For another temporal example---but now with infinitely divisible time, as opposed to the discrete time of Example \ref{SeaBattle}---recall from Examples \ref{IntervalEx} and \ref{IntervalEx2} the poset $(S,\sqsubseteq)$ of non-empty open intervals of rational numbers ordered by inclusion. Given intervals $(a,b), (c,d)$, let $(a,b)R_f(c,d)$ if $a < c$. Then we claim that $R_f$ satisfies the three conditions in Lemma \ref{SufficientForFull}:
\begin{itemize}
\item \upR{}: suppose $(a',b')\subseteq (a,b)$ and $(a',b')R_f(c',d')$. From $(a',b')\subseteq (a,b)$, we have $a\leq a'$, and from $(a',b')R_f(c',d')$, we have $a'< c'$. Thus, $a < c'$, which yields $(a,b)R_f(c',d')$.
\item \Rdown{}: suppose $(a,b)R_f(c,d)$ and $(c',d')\subseteq (c,d)$. From $(a,b)R_f(c,d)$, we have $a< c$, and from $(c',d')\subseteq (c,d)$, we have $c\leq c'$. Thus, $a< c'$, which yields $(a,b)R_f(c',d')$.
\item \Rref{}: suppose $(a,b)R_f(c,d)$, so $a< c$. If $b\leq c$, let $(a',b')=(a,b)$. If $c<b$, let  $(a',b')=(a,c)$. Observe that $(a',b')\subseteq (a,b)$ and for all $(a'',b'')\subseteq (a',b')$, $a''<c$, so $(a'',b'')R_f(c,d)$.
\end{itemize}
Thus, by Lemma \ref{SufficientForFull}, $(S,\sqsubseteq, \mathcal{RO}(S,\sqsubseteq),R_f)$ is a full relational possibility frame. In line with the temporal interpretation of $(S,\sqsubseteq)$ mentioned in Example \ref{IntervalEx}, we may regard $R_f$ as the future accessibility relation for future modalities $\Box_f$ and $\Diamond_f$ (similarly, we can introduce a past accessibility relation by $(a,b)R_{p}(c,d)$ if $d < b$). Note that despite $R_f$ being irreflexive, $\Box_f p\to p$ (resp.~$ p\to\Diamond_f p$) is valid: supposing $\Box_f p$ is true at $(a,b)$, observe that $\forall (a',b')\subseteq (a,b)$ $\exists (a'',b'')\subseteq (a',b')$ such that $a<a''$ and hence $(a,b)R_f(a'',b'')$, which with $\Box_fp$ being true at $(a,b)$ implies that $p$ is true at $(a'',b'')$, which by refinability implies that $p$ is true at $(a,b)$ (cf.~\cite[p.~459]{Roper1980}).  For the close relation between this possibility semantics for temporal logic and  \textit{interval semantics} for temporal logic, see \S~\ref{Intervals}.
\end{example}

In Example \ref{TempEx} we  showed that the relation $R$ actually satisfies the following strengthening\footnote{Another condition dubbed \RrefPlus{} is considered in \cite{Holliday2018}.} of \Rref{}:
\begin{itemize}
\item \RrefPlusPlus{}: if $xRy$, then $\exists x'\sqsubseteq x$ $\forall x''\sqsubseteq x'$: $x''Ry$.
\end{itemize}
This was Humberstone's original condition for modal semantics instead of \Rref{}, but for reasons explained in \cite{Holliday2014,Holliday2018}, it proved too strong for the general theory of possibility semantics for modal logic. 

Below are intuitive, general explanations of each of the three conditions in Lemma \ref{SufficientForFull}, working with the standard idea of ``accessibility'' that $xR_iy$ iff for every proposition $Z$, if $\Box_i Z$ is true at $x$, then $Z$ is true at $y$:
\begin{itemize}
\item \upR{}. Assume $x'\sqsubseteq x$ and $x'R_iy'$. For $xR_iy'$, we must argue that whenever $\Box Z$ is true at $x$, $Z$ should be true at $y'$. Suppose $\Box Z$ is true at $x$. Then since $x'\sqsubseteq x$, by persistence, $\Box Z$ should be true at $x'$. Then since $x'R_iy'$, $Z$ should be true at $y'$.
\item \Rdown{}. Assume $xR_iy$ and $y'\sqsubseteq y$. For $xR_iy'$, suppose $\Box Z$ is true at $x$. Then since $xR_iy$, $Z$ should be true at $y$, and then since $y'\sqsubseteq y$, $Z$ should be true at $y'$ by persistence.
\item \Rref{}. Think of obtaining $x'$ by extending the description of $x$ with $\Diamond_i \mathord{\downarrow}y$, using the fact that $\mathord{\downarrow}y$ is regular open if $(S,\sqsubseteq)$ is separative (recall Definition \ref{SepDef} and Proposition \ref{SepProp}).
\end{itemize}

It will be convenient to have a term for frames satisfying the three properties above.

\begin{definition} A \textit{paradigm} relational possibility frame is a relational possibility frame satisfying \upR{}, \Rdown{}, and \Rref{}.
\end{definition}

The possibility frames we will define from algebras also satisfy an additional property that can be useful:
\begin{itemize}
\item \Rdense{} -- $xR_iy$ if $\forall y'\sqsubseteq y$ $\exists y''\sqsubseteq y'$: $xR_iy''$ (see Figure \ref{RdenseFig}).
\end{itemize}
Note that together \Rdown{} and \Rdense{} are equivalent to the condition that for each $x\in S$, we have $R_i(x)\in\mathcal{RO}(S,\sqsubseteq)$.

 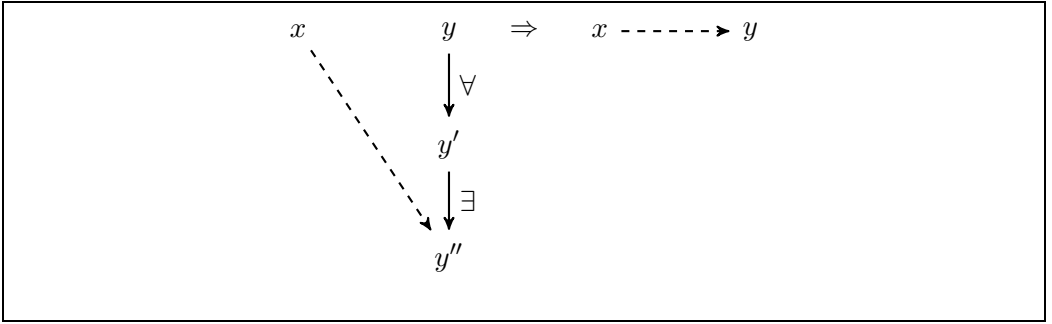
\begin{figure}
\begin{center}
\begin{tikzpicture}[->,>=stealth',shorten >=1pt,shorten <=1pt, auto,node
distance=2cm,thick,every loop/.style={<-,shorten <=1pt}]
\tikzstyle{every state}=[fill=gray!20,draw=none,text=black]

\node (x) at (0,0) {{$x$}};
\node (y) at (2,0) {{$y$}};
\node (y') at (2,-1.5) {{$y'$}};
\node (y'') at (2,-3) {{$y''$}};

\node at (3,0) {{\textit{$\Rightarrow$}}};

\path (x) edge[dashed,->] node {{}} (y'');
\path (y) edge[->] node {{$\forall$}} (y');
\path (y') edge[->] node {{$\exists$}} (y'');

\node (xR) at (4,0) {{$x$}};
\node (yR) at (6,0) {{$y$}};

\path (xR) edge[dashed,->] node {{}} (yR);

\end{tikzpicture}
\end{center}
\caption{the \Rdense{} condition. The quantifiers in the diagram correspond to the quantifiers in the definition of \Rdense{} in the text.}\label{RdenseFig}
\end{figure}

\begin{definition}\label{StrongDef} A \textit{strong} relational possibility frame is a relational possibility frame satisfying \upR{}, \Rdown{}, \Rref{}, and \Rdense{}.
\end{definition}

\begin{example} It is easy to check that the temporal frame in Examples \ref{SeaBattle} satisfies \Rdense{}. However, the temporal frame in Example \ref{TempEx} does not, for a reason raised in our earlier discussion: although $\forall (a',b')\subseteq (a,b)$ $\exists (a'',b'')\subseteq (a',b')$ such that $a<a''$ and hence $(a,b)R_f(a'',b'')$, it is not the case that $(a,b)R_f(a,b)$, because $a\not < a$. In fact, we obtain a frame that realizes the same BAO if we define $(a,b)R_f'(c,d)$ by $a\leq c$, instead of $a<c$, and then \Rdense{} is satisfied. We adopted the definition with $a<c$ to facilitate comparison with \cite{Roper1980} in \S~\ref{Intervals}.\end{example}

In Appendix \ref{NecSucc}, we give a single condition equivalent to the conjunction of the four in Definition \ref{StrongDef}. We can relate the notion of a strong frame to the standard notion of \textit{tightness} of general frames (see, e.g., \cite[p.~251]{Chagrov1997}) as follows.

\begin{definition}\label{RTightDef} A relational possibility frame $\mathcal{F}={( S,\sqsubseteq, P, \{R_i\}_{i\in I})}$ is \textit{$R$-tight} if for all $x,y\in S$, if for all $Z\in P$, $x\in \Box_i Z$ implies $y\in Z$, then $xR_iy$.
\end{definition}

\begin{lemma}[\cite{Holliday2018}]\label{TightStrong} \textnormal{For any relational possibility frame $\mathcal{F}$:
\begin{enumerate}
\item\label{TightStrong1} if $\mathcal{F}$ is $R$-tight, then $\mathcal{F}$ satisfies \upR{}, \Rdown{}, and \Rdense{};
\item\label{TightStrong1.5} if $\mathcal{F}$ is $R$-tight and satisfies \Rref{}, then $\mathcal{F}$ is strong;
\item\label{TightStrong1.75} if $\mathcal{F}$ is \textit{full}, then $\mathcal{F}$ is $R$-tight iff $\mathcal{F}$ is strong.
\end{enumerate}}
\end{lemma}

\begin{lemma}[\cite{Holliday2018}]\label{FullToStrong} \textnormal{For any possibility frame $\mathcal{F}={( S,\sqsubseteq, P,\{R_i\}_{i\in I})}$, define $\mathcal{F}={( S,\sqsubseteq, P, \{R_i^\Box\}_{i\in I})}$ by $xR_i^\Box y$ iff for all $Z\in P$, $x\in \Box_i Z$ implies $y\in Z$. Then:
\begin{enumerate}
\item $\mathcal{F}^\mathsf{b}=\mathcal{F}^{\Box\mathsf{b}}$, and $\mathcal{F}^{\Box}$ is $R$-tight;
\item if $\mathcal{F}$ is full, then $\mathcal{F}^\Box$ is strong.
\end{enumerate}}
\end{lemma}

Thus, may assume without loss of generality that any full relational possibility frame we are working with is in fact strong. Indeed, we may assume this without loss of generality for any relational possibility frame (not just a full one), but we postpone the reason to \S~\ref{GenRelFrame}. 

One of the appeals of Kripke frame semantics for normal modal logics is the correspondence theory between modal axioms and first-order properties of the $R_i$ relations \cite{Benthem1980,Benthem1983}. In the setting of full possibility frames, we also have an appealing correspondence between modal axioms and first-order properties of $R_i$ \textit{and} $\sqsubseteq$. Compare the following two theorems, where for a sequence $\alpha=(\alpha_1,\dots,\alpha_n)$ of modal indices from $I$, $\Diamond_\alpha\varphi :=\Diamond_{\alpha_1}\dots\Diamond_{\alpha_n}\varphi$, $\Box_\alpha\varphi:=\Box_{\alpha_1}\dots\Box_{\alpha_n}\varphi$, and $xR_\alpha y$ iff there are $x_0,\dots,x_n$ with $x_0=x$, $x_n=y$, and $x_0R_{\alpha_1}x_1$, $x_1R_{\alpha_2}x_2$, \dots, $x_{n-1}R_{\alpha_n}x_n$ (if $\alpha$ is the empty sequence, $\Diamond_\alpha\varphi=\Box_\alpha\varphi=\varphi$ and $xR_\alpha y$ iff $x=y$).

\begin{theorem}[\cite{Lemmon1977}]\label{LS} \textnormal{Let $\mathfrak{F}$ be a Kripke frame. Then for any sequences $\alpha$, $\beta$, $\delta$, and $\gamma$ of indices from $I$, $\Diamond_\alpha\Box_\beta p\rightarrow \Box_\delta\Diamond_\gamma p$ is valid on $\mathfrak{F}$ iff $\mathfrak{F}$ satisfies:
\begin{equation*}\forall {x}\forall {y}\forall {z} (({x} R_\delta{y}\wedge {x} R_\alpha {z})\rightarrow \exists {u} ( {y} R_\gamma {u} \wedge {z} R_\beta {u})).\label{LSeq}\end{equation*}}
\end{theorem}

\begin{theorem}[\cite{Holliday2018}]\label{Gklmn} \textnormal{Let $\mathcal{F}$ be a full paradigm relational possibility frame. Then for any sequences $\alpha$, $\beta$, $\delta$, and $\gamma$ of indices from $I$, $\Diamond_\alpha  \Box_\beta  p\rightarrow \Box_\delta   \Diamond_\gamma  p$ is valid on $\mathcal{F}$ iff $\mathcal{F}$ satisfies
\begin{equation*}\forall{x}\forall{y}\big({x}R_\delta {y}\rightarrow \exists {x'}\sqsubseteq{x} \;\forall {z}( {x'}R_\alpha   {z} \rightarrow\exists {u} ({y} R_\gamma  {u}\wedge {z} R_\beta  {u}))\big).\label{Gklmn-rel-con}\end{equation*} 
For the case where $\alpha$ is empty, $\Box_\beta  p\rightarrow \Box_\delta   \Diamond_\gamma  p$ is valid on $\mathcal{F}$ iff $\mathcal{F}$ satisfies
\begin{equation*}\forall{x}\forall{y}({x}R_\delta {y}\rightarrow \exists {u} ({y} R_\gamma  {u}\wedge x R_\beta  {u})).\label{Gklmn-rel-con3}\end{equation*}}
\end{theorem}
\begin{remark} Over \textit{strong} full possibility frames, the first-order conditions in Theorem \ref{Gklmn} may simplify: for example, $\Box_i p\to p$ corresponds just to reflexivity of $R_i$ and $\Box_ip\to\Box_ip\Box_ip$ just to transitivity. Diamond formulas also simplify but still involve $\sqsubseteq$. For example, $\Diamond_ip\to\Box_i\Diamond_ip$ corresponds over strong frames to $\forall{x}\forall{y}\big({x}R_i{y}\rightarrow \exists {x'}\sqsubseteq{x} \;\forall {z}( {x'}R_i {z} \rightarrow yR_i z )\big)$.\end{remark}

More generally, Yamamoto \cite{Yamamoto2017} has proved the analogue of the Sahlqvist correspondence theorem (see \cite[\S~3.6]{Blackburn2001}) for full possibility frames.

\begin{theorem}[\cite{Yamamoto2017}] \textnormal{For every Sahlqvist formula $\varphi$ of $\mathcal{L}(I)$, there is a formula $c_\varphi$ in the first-order language with relation symbols for $R_i$ and $\sqsubseteq$ such that a full possibility frame $\mathcal{F}=(S,\sqsubseteq, \mathcal{RO}(S,\sqsubseteq), \{R_i\}_{i\in I})$ validates $\varphi$ iff $(S,\sqsubseteq, \{R_i\}_{i\in I})$ as a first-order structure satisfies $c_\varphi$. Moreover, $c_\varphi$ is effectively computable from $\varphi$.}
\end{theorem}

\noindent This result was generalized to the larger class of \textit{inductive} formulas in \cite{Zhao2018,Zhao2016}.

\subsubsection{Relational possibility models and Kripke models}\label{PossKripkeModels}

As noted in Remark \ref{FullKripkeRemark}, our focus in this chapter is on possibility frames, frame validity of formulas, and algebras arising from frames. At this level, the limitation of Kripke frames is that they can realize only atomic algebras. However, here we will briefly discuss the connection between possibility \textit{models} and Kripke \textit{models}. One direction is obvious: every Kripke model can be viewed as a possibility model where $\sqsubseteq$ is the identity relation. In the other direction, given any possibility model $\mathcal{M}$ and formula $\varphi$, we can construct a Kripke model based on the \textit{$\varphi$-decisive possibilities} in $\mathcal{M}$, which decide the truth value of every subformula of $\varphi$.

\begin{definition} Given a possibility model $\mathcal{M}=(S,\sqsubseteq, \{R_i\}_{i\in I},\pi)$ and formula $\varphi$,\footnote{We have omitted the specification of the set $P$ of admissible propositions of the underlying frame of $\mathcal{M}$, as only the valuation $\pi$ matters for the model.} we define the Kripke model $\mathcal{M}_\varphi=(S_\varphi, \{R_{i,\varphi}\}_{i\in I},\pi_\varphi)$ as follows:
\begin{enumerate}
\item $S_\varphi = \{x\in S\mid \mbox{for all subformulas }\psi\mbox{ of }\varphi, \mathcal{M},x\Vdash\psi\mbox{ or }\mathcal{M},x\Vdash\neg\psi\}$.
\item $xR_{i,\varphi}y$ iff $\exists z\in R_i(x)$: $y\sqsubseteq z$;
\item $\pi_\varphi(p)=\pi(p)\cap S_\varphi$.
\end{enumerate}
\end{definition}

\begin{lemma}\label{PossToKripke}Let  $\mathcal{M}=(S,\sqsubseteq, \{R_i\}_{i\in I},\pi)$ be a possibility model and $\varphi$ a formula. Then:
\begin{enumerate} 
\item\label{PossToKripke1} for every $x\in S$, there is an $x'\in S_\varphi$ such that $x'\sqsubseteq x$;
\item\label{PossToKripke2} for every $x\in S_\varphi$ and subformula $\psi$ of $\varphi$, $\mathcal{M},x\Vdash\psi$ iff $\mathcal{M}_\varphi,x\Vdash\psi$.
\end{enumerate}
\end{lemma}

\begin{proof} For part \ref{PossToKripke1}, note that for any $y\sqsubseteq x$, if $\mathcal{M},y\nVdash \psi$, then by refinability there is a $z\sqsubseteq y$ such that $\mathcal{M},z\Vdash\neg\psi$; and if $\mathcal{M},y\Vdash\psi$, then by persistence, for any $z\sqsubseteq y$, $\mathcal{M},z\Vdash\psi$. Thus, starting with $x$, we can apply refinability and persistence for each subformula $\psi$ of $\varphi$ in turn until reaching the desired $x'\in S_\varphi$ with $x'\sqsubseteq x$.

Part \ref{PossToKripke2} is by induction. The base case and $\wedge$ case are obvious. For any subformula $\neg\psi$ of $\varphi$, we have $\mathcal{M},x\Vdash\neg\psi$ iff $\mathcal{M},x\nVdash\psi$ (since $x\in S_\varphi$) iff $\mathcal{M}_\varphi,x\nVdash\psi$ (by the inductive hypothesis) iff $\mathcal{M}_\varphi,x\Vdash\neg\psi$ (since $\mathcal{M}_\varphi$ is a Kripke model). For any subformula $\Box_i\psi$ of $\varphi$, if $\mathcal{M}_\varphi,x\nVdash\Box_i\psi$, then there is a $y\in S_\varphi$ such that $xR_{i,\varphi}y$ and $\mathcal{M}_\varphi,y\nVdash\psi$, which implies $\mathcal{M},y\nVdash\psi$ by the inductive hypothesis. Since $xR_{i,\varphi}y$, there is a $z\in R_i(x)$ with $y\sqsubseteq z$. Then by persistence, $\mathcal{M},y\nVdash\psi$ implies $\mathcal{M},z\nVdash\psi$, so $\mathcal{M},x\nVdash\Box_i\psi$. Conversely, if $\mathcal{M},x\nVdash\Box_i\psi$, so there is a $y\in R_i(x)$ such that $\mathcal{M},y\nVdash\psi$, then by refinability there is a $y'\sqsubseteq y$ such that $\mathcal{M},y'\Vdash\neg\psi$. Then by part \ref{PossToKripke1}, there is a $y''\in S_\varphi$ with $y''\sqsubseteq y'$, which implies $\mathcal{M},y''\nVdash\psi$ and $y''\sqsubseteq y$. It follows that $xR_{i,\varphi}y''$ and hence $\mathcal{M}_\varphi,x\nVdash \Box_i\psi$.\end{proof}

Note, however, that the underlying Kripke frame of $\mathcal{M}_\varphi$ may fail to validate the same formulas as the underlying full possibility frame of $\mathcal{M}$. Nonetheless, Lemma \ref{PossToKripke} and the soundness of the minimal normal modal logic $\mathsf{K}$ with respect to Kripke models yields soundness of $\mathsf{K}$ with respect to possibility models without any additional checking of Boolean or modal axioms.

\subsubsection{From $\mathcal{V}$-BAOs to full and principal frames}\label{VtoFrame}

We now return to our discussion of the relation between possibility frames and BAOs. In this subsection and the next, we will make precise the following claim: from an algebraic point of view, the essence of \textit{relational semantics} for modal logic  in its most basic form (i.e., before the step to general frames or topological frames that can realize all modal algebras) is that the $\Box_i$ modalities distribute over arbitrary conjunctions---or dually, the $\Diamond_i$ modalities distribute over arbitrary disjunctions.  The atomicity of the dual algebras of Kripke frames can be disassociated from the essence of the basic relational semantics.

Formally, following \cite{Litak2005b}, we call a Boolean algebra equipped with unary $\Box_i$ operations a $\mathcal{V}$-BAO if $\Box_i$ distributes over any existing meets:
 \[\mbox{if $\underset{a\in A}{\bigwedge}a$ exists, then $\Box_i (\underset{a\in A}{\bigwedge}a)= \underset{a\in A}{\bigwedge}{\Box_i a}$}.\]
 Then $\Box_i$ is said to be \textit{completely multiplicative}. This is equivalent to the $\Diamond_i$ operation related by $\Diamond_i a=\neg\Box_i\neg a$ being \textit{completely additive}:
 \[\mbox{if $\underset{a\in A}{\bigvee}a$ exists, then $\Diamond_i (\underset{a\in A}{\bigvee}a)= \underset{a\in A}{\bigvee}{\Diamond_i a}$}.\]
The `$\mathcal{V}$' is chosen to suggest the big join $\bigvee$, thinking in terms of $\Diamond_i$ as primitive (in contrast to our approach of taking $\Box_i$ as primitive). 
 
We can turn $\mathcal{V}$-BAOs into relational possibility frames as follows.

\begin{theorem}[\cite{Holliday2018}]\label{VtoPoss}\textnormal{For any $\mathcal{V}$-BAO $\mathbb{B}=(B,\{\Box_i\}_{i\in I})$, define:
\begin{eqnarray*}
\mathbb{B}_\mathsf{u}&=&(B_+,\leq_+,\mathcal{RO}(B_+,\leq_+),\{R_i\}_{i\in I}),\mbox{the \textit{full frame of $\mathbb{B}$}, and} \\
\mathbb{B}_\mathsf{p}&=&(B_+,\leq_+,\{\mathord{\downarrow}_+ x\mid x\in B_+\}\cup \{\varnothing\},\{R_i\}_{i\in I}),\mbox{the \textit{principal frame of $\mathbb{B}$}},
\end{eqnarray*}
where  $(B_+,\leq_+)$ is as in Theorem \ref{FirstThm}.\ref{FirstThm2} and for $x,y\in B_+$ and $i\in I$, 
\begin{equation}\mbox{$xR_iy$ iff $\forall y'\in B_+$, if $y'\leq_+ y$ then $x\wedge\Diamond_iy'\neq 0$ in $\mathbb{B}$}.\label{KeyRDef}\end{equation}
Then:
\begin{enumerate}
\item $\mathbb{B}_\mathsf{u}$ and $\mathbb{B}_\mathsf{p}$ are strong relational possibility frames, with $\mathbb{B}_\mathsf{u}$ being full; 
\item if $B$ is a complete Boolean algebra, then $\mathbb{B}_\mathsf{u}=\mathbb{B}_\mathsf{p}$;
\item $(\mathbb{B}_\mathsf{u})^\mathsf{b}$ is the lower MacNeille completion (or Monk completion) of $\mathbb{B}$;
\item\label{VtoPoss4} $(\mathbb{B}_\mathsf{p})^\mathsf{b}$ is isomorphic to $\mathbb{B}$.
\end{enumerate}}
\end{theorem}

\begin{remark} The proof of part \ref{VtoPoss4} takes advantage of the surprising fact---to which the study of possibility semantics led---that complete multiplicativity of $\Box_i$ (or complete additivity of $\Diamond_i$), an ostensibly \textit{second-order} condition on a BAO, is in fact equivalent to the following \textit{first-order} condition on a BAO, defined using the $R_i$ relation from (\ref{KeyRDef}):
\[\mbox{if $x\wedge \Diamond_iy\neq 0$, then $\exists y'\in B_+$: $y'\leq_+ y$ and $xR_iy'$}.\] 
This is proved in \cite{Holliday2015,Litak2019} for BAOs (for a generalization to posets, see \cite{Andreka2016}).\end{remark}

By Theorem \ref{VtoPoss}, any complete $\mathcal{V}$-BAO can be represented as the algebra associated with a \textit{full} possibility frame, while any $\mathcal{V}$-BAO can be represented as the algebra associated with a possibility frame whose admissible sets are precisely the principle downsets together with $\varnothing$. We will use these facts in what follows. They imply that any logic complete with respect to a class of complete $\mathcal{V}$-BAOs (resp. $\mathcal{V}$-BAOs) is complete with respect to a class of the corresponding frames.

A $\mathcal{T}$-BAO \cite{Litak2005b} is a BAO $\mathbb{B}=(B,\{\Box_i\}_{i\in I})$ in which for each $i\in I$, there is an $f_i:B\to B$ that is a residual of $\Box_i$,\footnote{$\mathcal{T}$ stands for `tense', as in tense logics the past diamond operator is a residual of the future box operator.} i.e., for all $a,b\in B$:
\begin{equation}a\leq \Box_ib\mbox{ iff }f_i(a)\leq b.\label{residuation}\end{equation}
A useful fact is that a \textit{complete} BAO is a $\mathcal{V}$-BAO iff it is a $\mathcal{T}$-BAO. The right to left direction is obvious. For the left to right, simply take
\[f_i(a)=\bigwedge \{ c\in B\mid a\leq \Box_i c \}.\] 
Now we have an analogue of Proposition \ref{NeighDef} for relational possibility frames. 

\begin{theorem}\label{RelDef} \textnormal{Suppose $B$ is a Boolean algebra isomorphic to  $\mathcal{F}^\mathsf{b}$  for some possibility frame $\mathcal{F}=(S,\sqsubseteq,P)$ in which every principal downset belongs to $P$.\footnote{E.g., $P=\mathcal{RO}(S,\sqsubseteq)$ for a separative poset $(S,\sqsubseteq)$.} Then for any $\mathcal{T}$-BAO $\mathbb{B}=(B,\{\Box_i\}_{i\in I})$, there is a strong relational  possibility frame $\mathbb{B}_\mathcal{F}=(S,\sqsubseteq, P, \{R_i\}_{i\in I})$ such that $(\mathbb{B}_\mathcal{F})^\mathsf{b}$ is isomorphic to $\mathbb{B}$.}
\end{theorem}

\begin{proof} We use an isomorphism $\sigma$ from $\mathcal{F}^\mathsf{b}$ to $B$ to define accessibility relations $R_i$ on $S$. Let $f_i$ be the residual of $\Box_i$ in the  $\mathcal{T}$-BAO based on $B$. Then for $x\in S$, let
\[R_i(x)= \sigma^{-1}(f_i(\sigma(\mathord{\downarrow}x))).\]
With this definition, we claim that for any $Z\in P$, $\sigma(\Box_iZ)=\Box_i\sigma(Z)$. Indeed:
\begin{eqnarray*}
\sigma(\Box_iZ) & = & \sigma (\{x\in S\mid R_i(x)\subseteq Z\}\\
& = & \sigma (\{x\in S\mid f_i(\sigma(\mathord{\downarrow}x))\leq \sigma(Z)\}\\
& = & \sigma (\{x\in S\mid \sigma(\mathord{\downarrow}x)\leq \Box_i\sigma(Z)\}\mbox{ by (\ref{residuation})}\\
& = & \sigma (\{x\in S\mid \mathord{\downarrow}x\subseteq \sigma^{-1}(\Box_i\sigma(Z))\}\\
&=& \sigma (\sigma^{-1}(\Box_i\sigma(Z)))\\
&=& \Box_i\sigma(Z).
\end{eqnarray*}
Hence $(\mathbb{B}_\mathcal{F})^\mathsf{b}$ is isomorphic to $\mathbb{B}$. Let $\mathsf{f}_i:P\to P$ be the residual of $\Box_i$ in $(\mathbb{B}_\mathcal{F})^\mathsf{b}$, so  \[R_i(x)=\mathsf{f}_i(\mathord{\downarrow}x).\] 

We now show that $\mathbb{B}_\mathcal{F}$ is strong. First, \upR{} is equivalent to the condition that $x'\sqsubseteq x$ implies $R_i(x')\subseteq R_i(x)$. Now if $x'\sqsubseteq x$,  so $\mathord{\downarrow}x'\subseteq\mathord{\downarrow}x$, then  ${\mathsf{f}_i(\mathord{\downarrow}x')\subseteq \mathsf{f}_i(\mathord{\downarrow}x)}$, so $R_i(x')\subseteq R_i(x)$. \Rdown{} and \Rdense{} are immediate from the fact that  $\mathsf{f}_i(\mathord{\downarrow}x)\in P\subseteq\mathcal{RO}(S,\sqsubseteq)$. Finally, for \Rref{}, suppose $xR_iy$. Then ${x\not\in \Box_i\neg\mathord{\downarrow}y}$,  so there is an $x'\sqsubseteq x$ such that $x'\in \Diamond_i \mathord{\downarrow}y $. Consider any $x''\sqsubseteq x'$. Then $\mathord{\downarrow}x''\not\subseteq \Box_i\neg\mathord{\downarrow}y$, so by residuation, $\mathsf{f}_i(x'')\not\subseteq \neg\mathord{\downarrow}y$, i.e., $R_i(x'')\not\subseteq \neg\mathord{\downarrow}y$. Hence there is some $z\in R_i(x'')$ such that $z\not\in \neg\mathord{\downarrow}y$, which implies there is a $y'\sqsubseteq z$ such that $y'\sqsubseteq y$. Then since $R_i(x'')$ is a downset, we have $x''R_iy'$, which establishes \Rref{}. 
\end{proof}

Theorem \ref{RelDef} has a tree completeness corollary directly analogous to Corollary \ref{TreeThm}, but now for relational possibility frames based on $2^{<\omega}$.

\subsubsection{Dual equivalence with complete $\mathcal{V}$-BAOs}\label{DualEquiv}

The representation theorem for complete $\mathcal{V}$-BAOs in Theorem \ref{VtoPoss} can be turned into a full categorical duality. On the algebra side, we have the following category.

\begin{proposition} \textnormal{Complete $\mathcal{V}$-BAOs with complete BAO-homomorphisms, i.e., Boolean homomorphisms preserving arbitrary meets and each $\Box_i$ operator, form a category, \textbf{$\mathcal{CV}$-BAO}.}
\end{proposition}

On the possibility side, we first characterize the possibility frames that arise from complete $\mathcal{V}$-BAOs via the $(\cdot )_\mathsf{u}$ map.

\begin{lemma}[\cite{Holliday2018}]\label{RichLem} \textnormal{Let $\mathcal{F}={(S,\sqsubseteq, \mathcal{RO}(S,\sqsubseteq), \{R_i\}_{i\in I})}$ be a full relational possibility frame. Then $\mathcal{F}$ is isomorphic to $(\mathcal{F}^\mathsf{b})_\mathsf{u}$ iff the underlying poset of $\mathcal{F}$ is obtained from a complete Boolean lattice by deleting its bottom element and for all $i\in I$ and $x\in S$, $R_i(x)$ is a principal downset in $(S,\sqsubseteq)$.}
\end{lemma}

\begin{definition} A \textit{rich possibility frame} is a full possibility frame satisfying the right-hand side of the equivalence in Lemma \ref{RichLem}.
\end{definition}

To define categories of full and rich possibility frames, we need to introduce morphisms. We use the notation $x\comp y$ for $\mathord{\downarrow}x\cap\mathord{\downarrow}y\neq\varnothing$.

\begin{definition}\label{PossMorphs} Given full possibility frames $\mathcal{F}$ and $\mathcal{F}'$, a \textit{strict possibility morphism} from $\mathcal{F}$ to $\mathcal{F}'$ is an $h\colon S\to S'$ such that for all $x,y\in S$ and $y'\in S'$:
\begin{enumerate}
\item\label{SqForthCon} \SqForth{} -- if $y\sqsubseteq x$, then $h(y)\sqsubseteq' h(x)$;
\item \SqBack{} -- if $y'\sqsubseteq' h( x )$, then $\exists y$: $y\sqsubseteq x $ and $h(y)\sqsubseteq' y'$;
\item \RForth{} -- if $xR  y$, then $h(x)R 'h(y)$;
\item\label{SRBackCon} \SRBack{} -- if $h(x)R' y'$ and $z'\sqsubseteq ' y'$, then $\exists y$: $xR  y$ and $h(y)\comp' z'$.
\end{enumerate}
A \textit{p-morphism} is defined in the same way as a strict possibility morphism, but with strengthened versions of the two back conditions:
\begin{enumerate}
\item[$2'.$] \pSqBack{} -- $y'\sqsubseteq' h( x )$, then $\exists y$: $y\sqsubseteq x $ and $h(y)= y'$;
\item[$4'.$] \pRBack{} -- if $h(x)R' y'$, then $\exists y$: $xR  y$ and $h(y)=y'$.
\end{enumerate}
\end{definition}

\begin{proposition}[\cite{Holliday2018}]$\,$\textnormal{
\begin{enumerate}
\item Full possibility frames with strict possibility morphisms form a category, \\ $\mathbf{FullPoss}$;
\item Rich possibility frames with p-morphisms form a category, $\mathbf{RichPoss}$.
\end{enumerate}}
\end{proposition}

Now we can make precise the categorical relationship between full possibility frames and complete $\mathcal{V}$-BAOs.

\begin{theorem}[\cite{Holliday2018}] $\,$\textnormal{
\begin{enumerate} 
\item \textbf{RichPoss} is a reflective subcategory of \textbf{FullPoss};
\item \textbf{RichPoss} is dually equivalent to \textbf{$\mathcal{CV}$-BAO}.
\end{enumerate}}
\end{theorem}

Recall, by contrast, that the category of Kripke frames with p-morphisms is dually equivalent to the category of complete \textit{and atomic} $\mathcal{V}$-BAOs ($\mathcal{CAV}$-BAOs) with complete BAO-homomorphisms \cite{Thomason1975}.

\subsubsection{Quasi-normal possibility frames}\label{QuasiSection}

In this section, we briefly comment on possibility semantics for quasi-normal modal logics. These logics have important applications in provability logic \cite{Boolos1993}, which is the subject of \S\S~\ref{FullRelFrameSection}-\ref{PrincRelFrame}.

A \textit{quasi-normal modal logic} is a set $\mathsf{L}$ of $\mathcal{L}(I)$ formulas that contains all theorems of the smallest normal modal logic, $\mathsf{K}$, and is closed under modus ponens and uniform substitution---but not necessarily under the \textit{necessitation rule} stating that if $\varphi\in\mathsf{L}$, then $\Box_i\varphi\in\mathsf{L}$. Algebraic semantic for quasi-normal modal logics can be given using matrices $(\mathbb{B},F)$ where $\mathbb{B}=(B,\{\Box_i\}_{i\in I})$ is a BAO and $F$ is a distinguished filter in $B$. Instead of defining a formula $\varphi$ to be valid if for any valuation $\theta$ on $\mathbb{B}$, $\tilde{\theta}(\varphi)=1_B$, we define $\varphi$ to be valid if for any valuation $\theta$ on $\mathbb{B}$, $\tilde{\theta}(\varphi)\in F$. Then it is possible for $\varphi$ to be valid on $\mathbb{B}$ while $\Box_i\varphi$ is not valid. Yet modus ponens and uniform substitution still preserve validity. 

One obvious way to give possibility semantics for quasi-normal modal logics is to equip possibility frames with a distinguished filter of propositions.

\begin{definition}\label{Quasi1} A \textit{relational possibility frame with distinguished filter} is a pair $(\mathcal{F},F)$ where $\mathcal{F}=( S,\sqsubseteq,P, \{R_i\}_{i\in I})$ is a relational possibility frame and $F$ is a filter in the algebra $P$ of propositions.\end{definition}

\begin{remark} When $\sqsubseteq$ is the identity relation, $P=\wp(S)$, and $F$ is a principal filter, we obtain the standard notion of a \textit{Kripke frame with distinguished worlds} as a special case of Definition \ref{Quasi1}.
\end{remark}

A perhaps more appealing approach, given our move from worlds to possibilities, is to distinguish a directed set $S_0$ of possibilities.

\begin{definition}\label{Quasi2} A \textit{quasi-normal relational possibility frame} is a tuple $\mathcal{Q}={( S,\sqsubseteq,P, \{R_i\}_{i\in I},S_0)}$ where $\mathcal{F}_\mathcal{Q}=( S,\sqsubseteq,P, \{R_i\}_{i\in I})$ is a relational possibility frame and $S_0\subseteq S$ is such that if $x,y\in S_0$, then there is a $z\in S_0$ such that $z\sqsubseteq x$ and $z\sqsubseteq y$.  A formula $\varphi$ is \textit{valid on} $\mathcal{Q}$ if for every model $\mathcal{M}$ based on  $\mathcal{F}_\mathcal{Q}$, we have $\mathcal{M},x\Vdash\varphi$ for some $x\in S_0$.

If $S_0$ contains a maximum element with respect to $\sqsubseteq$, we call $\mathcal{Q}$ a \textit{relational possibility frame with a distinguished possibility}.
\end{definition} 

Any such frame gives us a matrix $\mathcal{Q}^\mathsf{m}=((\mathcal{F}_\mathcal{Q})^\mathsf{b}, F)$ with $(\mathcal{F}_\mathcal{Q})^\mathsf{b}$ as in Proposition \ref{BAOfromFrame} and the filter $F$ defined by:
\[F=\{Z\in P\mid S_0\cap Z\neq\varnothing \}\]
(which is principal if $\mathcal{Q}$ has a distinguished possibility). This yields the following, which can also be easily confirmed directly.
\begin{proposition} \textnormal{The set of formulas valid on any class of quasi-normal relational possibility frames is a quasi-normal modal logic.}
\end{proposition}

Conversely, any matrix $(\mathbb{B},F)$ where $\mathbb{B}$ is a complete $\mathcal{V}$-BAO and $F$ a proper filter gives us a quasi-normal relational possibility frame $(\mathbb{B}_\mathsf{u},S_0)$ with $\mathbb{B}_\mathsf{u}$ as in Theorem \ref{VtoPoss} and $S_0=F$, such that $(\mathbb{B}_\mathsf{u},S_0)^\mathsf{m}$ is isomorphic to $(\mathbb{B},F)$.

In addition, we note that any \textit{rooted} relational possibility frame, i.e., in which the poset $(S,\sqsubseteq)$ has a maximum element $m$, may be regarded as a quasi-normal relational possibility frame in which $S_0=\{m\}$. There is no loss of generality in possibility semantics for $\mathcal{L}(I)$ in working with rooted frames (note that $\mathbb{B}_\mathsf{u}$ is rooted, as is $\mathbb{B}_\mathsf{g}$ in \S~\ref{GenRelFrame}), so in this sense quasi-normal relational possibility frames are a generalization of relational possibility frames.

\subsubsection{Full frames and world incompleteness in provability logic}\label{FullRelFrameSection}

In this section, we illustrate the greater generality of full relational possibility frames over Kripke frames using examples inspired by provability logic. Recall that in possibility semantics, Kripke frames may be regarded as full relational world frames (see Remark \ref{FullKripkeRemark}). Already for the basic unimodal logic, there are many Kripke incomplete but full-possibility-frame complete logics. The following result adapts results of \cite{Litak2005} and \cite{Kracht1999b} to the setting of possibility semantics. 

\begin{theorem}[\cite{Holliday2018}]\label{CVnotCAV} \textnormal{There are continuum-many normal unimodal logics $\mathsf{L}$ such that $\mathsf{L}$ is the logic of a class of full relational possibility frames but $\mathsf{L}$ is not the logic of any class of full relational world frames.}
\end{theorem}

If we go beyond the basic unimodal language, we can obtain very natural examples of Kripke incomplete but full-possibility-frame complete logics. Here we consider the language of \textit{propositional term modal logic} \cite{Fitting2001,Padmanabha2019,Padmanabha2019b,Padmanabha2019c}, which allows quantification over modalities, as well as what can be regarded as a fragment of that language---the language of ``someone believes'' from doxastic logic \cite{Halpern1990,Grove1993}. These languages have natural applications in multi-agent doxastic or epistemic logic, as well as (we will propose) polymodal provability logic. In addition to the devices for quantifying over modalities, we consider extensions of the languages with 0-ary modalities $G(i)$ expressing that \textit{agent $i$ belongs to the group of agents over whom we are quantifying}, and 0-ary $O(i)$ modalities expressing the ``opinionatedness'' (in doxastic logic) or ``omniscience'' (in epistemic logic) or ``negation-completeness'' (in provability logic) of $i$. Fixing a countably infinite set $\mathsf{Var}$ of variables, the languages are:
\begin{eqnarray*}
&\mathcal{TML}_{GO} &\qquad \varphi::= p\mid \neg\varphi \mid (\varphi\wedge\varphi)\mid \Box_\mathrm{v} \varphi\mid \forall \mathrm{v} \varphi \mid G(\mathrm{v})\mid O(\mathrm{v}) \\
&\mathcal{TML}_{G} &\qquad \varphi::= p\mid \neg\varphi \mid (\varphi\wedge\varphi)\mid \Box_\mathrm{v} \varphi\mid \forall \mathrm{v} \varphi \mid G(\mathrm{v}) \\
&\mathcal{TML}_{O} &\qquad \varphi::= p\mid \neg\varphi \mid (\varphi\wedge\varphi)\mid \Box_\mathrm{v} \varphi\mid \forall \mathrm{v} \varphi \mid O(\mathrm{v}) \\
&\mathcal{TML} &\qquad \varphi::= p\mid \neg\varphi \mid (\varphi\wedge\varphi)\mid \Box_\mathrm{v} \varphi\mid \forall \mathrm{v} \varphi \\
&\mathcal{L}_{SG}(I)&\qquad\varphi::= p\mid \neg\varphi\mid (\varphi\wedge\varphi)\mid \Box_i\varphi\mid S\varphi \mid G(i) \\
&\mathcal{L}_{S}(I)&\qquad\varphi::= p\mid \neg\varphi\mid (\varphi\wedge\varphi)\mid \Box_i\varphi\mid S\varphi 
\end{eqnarray*}
where $p\in\mathsf{Prop}$, $\mathrm{v}\in \mathsf{Var}$, and $i\in I$. We define $\exists\mathrm{v}\varphi$ as $\neg\forall\mathrm{v}\neg\varphi$.

As the semantics below will show, the languages with $S$ may be regarded as fragments of the term modal languages by translating $S\varphi$ as $\exists \mathrm{v} \Box_\mathrm{v}\varphi$.

We can interpret all of these languages using the following frames.

\begin{definition}\label{VarAgt} A \textit{varying agent-domain full relational possibility frame} is a tuple $(\mathcal{F},\mathrm{G})$ where $\mathcal{F}=(S,\sqsubseteq, \mathcal{RO}(S,\sqsubseteq),\{R_i\}_{i\in I})$ is a full relational possibility frame for $\mathcal{L}(I)$ and $\mathrm{G}:S\to \wp(I)\setminus\{\varnothing\}$ satisfies the following:
\begin{enumerate}
\item \textit{persistence for $\mathrm{G}$}: if $x'\sqsubseteq x$, then $\mathrm{G}(x')\supseteq\mathrm{G}(x)$;
\item \textit{refinability for $\mathrm{G}$}: if $i\in I\setminus \mathrm{G}(x)$, then $\exists x'\sqsubseteq x$ $\forall x''\sqsubseteq x'$  $i\not\in \mathrm{G}(x'')$.
\end{enumerate}
We say $\mathcal{F}$ has \textit{universal agent-domain} if $\mathrm{G}(x)=I$ for all $x\in S$.

Varying agent-domain domain full \textit{neighborhood} possibility frames (see \S~\ref{NeighSection}) are defined analogously. Quasi-normal versions of these frames are defined as in Definitions \ref{Quasi1}-\ref{Quasi2} in the obvious way.
\end{definition}

We first explain the simpler semantics for $\mathcal{L}_{SG}(I)$.  

\begin{definition}\label{VarAgtSem} Given a possibility model $\mathcal{M}$ based on a frame $\mathcal{F}$ as in Definition \ref{VarAgt} and $\varphi\in\mathcal{L}_{SG}(I)$, we define $\mathcal{M},x\Vdash\varphi$ with the usual clauses for $\neg$, $\wedge$, and $\Box_i$, plus the following: 
\begin{enumerate}
\item $\mathcal{M},x\Vdash S\varphi$ iff $\forall x'\sqsubseteq x$ $\exists x''\sqsubseteq x'$ $\exists i\in \mathrm{G}(x'')$: $\mathcal{M},x''\Vdash \Box_i\varphi$;
\item $\mathcal{M},x\Vdash G(i)$ iff $i\in \mathrm{G}(x)$.
\end{enumerate}
A formula $\varphi$ is valid on $\mathcal{F}$ if for every model $\mathcal{M}$ based on $\mathcal{F}$ and every possibility $x$ in $\mathcal{F}$, we have $\mathcal{M},x\Vdash\varphi$.
\end{definition}

The following is immediate from the semantic clause for $S$.
\begin{lemma} \textnormal{If $\mathcal{M}$ is a world model, i.e., $\sqsubseteq$ is the identity relation, then:
\[\mbox{$\mathcal{M},w\Vdash S\varphi$ iff $\exists i\in\mathrm{G}(w)$: $\mathcal{M},w\Vdash \Box_i\varphi$}.\]}
\end{lemma}

Given the semantic clauses of Definition \ref{VarAgtSem}, we have the validity of the  (free logical) existential generalization principle:
\[(\Box_i\varphi\wedge G(i))\to S\varphi.\]

Now we turn to the term modal languages.

\begin{definition} Given a varying agent-domain full relational or neighborhood possibility frame $\mathcal{F}$ for a language $\mathcal{L}(I)$, a model $\mathcal{M}$ based on $\mathcal{F}$, variable assignment $g:\mathsf{Var}\to I$, and $\varphi\in\mathcal{TML}_{GO}$, we define $\mathcal{M},x\Vdash_g \varphi$ and $\llbracket \varphi\rrbracket^\mathcal{M}_g=\{y\in S\mid \mathcal{M},y\Vdash_g \varphi\}$ as follows:
\begin{itemize}
\item $\mathcal{M},x\Vdash_g \Box_\mathrm{v} \varphi$ iff $x\in \Box_{g(\mathrm{v})}\llbracket \varphi\rrbracket^\mathcal{M}_g$;
\item $\mathcal{M},x\Vdash_g \forall v\varphi$ iff for all $i\in \mathrm{G}(x)$, $\mathcal{M},x\Vdash_{g[\mathrm{v}:=i]} \varphi$;
\item $\mathcal{M},x\Vdash_g G(\mathrm{v})$ iff $g(\mathrm{v})\in \mathrm{G}(x)$;
\item $\mathcal{M},x\Vdash_g O(\mathrm{v})$  iff for all $Z\in \mathcal{RO}(S,\sqsubseteq)$, $x\in \Box_{g(\mathrm{v})}Z\vee \Box_{g(\mathrm{v})}\neg Z$.
\end{itemize}
A formula $\varphi$ is \textit{valid with respect to $\mathcal{F},g$} if for every model $\mathcal{M}$ based on $\mathcal{F}$ and every possibility $x$ in $\mathcal{F}$, $\mathcal{M},x\Vdash_g \varphi$. A formula is \textit{valid on $\mathcal{F}$} if for every variable assignment $g$, $\varphi$ is valid with respect to $\mathcal{F},g$. 
\end{definition}

Note the validity  according to this semantics of the principle
\[O(\mathrm{v})\to (\Box_\mathrm{v} p\vee \Box_\mathrm{v} \neg p).\]

The inadequacy of Kripke frames for the study of these languages is shown by a series of related problems to follow. In fact, they afflict not only Kripke frames but more generally any full neighborhood world frame (recall Remark \ref{FullNeighWorld}) that is \textit{monotonic}, i.e.,  each $N_i(w)$ is closed under supersets, which is the condition corresponding to the validity of $\Box_i(p\wedge q)\to \Box_i p$.  The essential problem is that world frames commit us to the idea that if every truth is believed by \textit{some agent/theory or other}, formalized by the frame validity of
\begin{equation*}p\to \exists \mathrm{v} \Box_\mathrm{v} p,\end{equation*} then there is a \textit{single agent/theory} who is fully opinionated:
\begin{equation*}\exists \mathrm{v} O(\mathrm{v}).\end{equation*}
 To see that world frames have this commitment, consider for a given world $w$ the ``world proposition'' $\{w\}$; then the agent who believes $\{w\}$ is fully opinionated. This can easily be turned into a proof of Proposition \ref{OmProp} below.
 
 \begin{remark}First, we note a connection to ``Fitch's paradox'' of knowability \cite{Fitch1963,Brogaard2019}: under weak assumptions, the ``verifiability principle'' that every truth \textit{could in principle} be known (at some time) by some agent or other, $p\to \blacklozenge \exists \mathrm{v} \Box_\mathrm{v} p$ (where $\blacklozenge$ is some kind of possibility modal), entails the stricter verificationist principle that every truth \textit{is} known (at some time) by some agent or other, $p\to \exists \mathrm{v} \Box_\mathrm{v} p$. Now we add that world frames commit the strict verificationist to the implausible principle that there is (at some time) a single omniscient agent.
 \end{remark}
 
\begin{proposition}\label{OmProp} \textnormal{Any varying agent-domain full monotonic neighborhood world frame that validates  $p\to \exists \mathrm{v} \Box_\mathrm{v} p$ validates $\exists \mathrm{v} O(\mathrm{v})$.}
\end{proposition}

Yet it is straightforward to construct a full relational possibility frame without the unwanted consequence.

\begin{proposition}\label{NoOmniscience}\textnormal{There is a universal agent-domain full relational possibility frame that validates  $p\to \exists \mathrm{v} \Box_\mathrm{v} p$ but not $\exists \mathrm{v} O(\mathrm{v})$.}
\end{proposition}

\begin{proof} Where $I$ is countably infinite, we build a possibility frame for $\mathcal{L}(I)$ based on the full infinite binary tree $2^{<\omega}$ regarded as a poset $(2^{<\omega},\sqsubseteq)$ as in Examples \ref{BinaryTree0} and \ref{BinaryTree0b}. Fix a bijection $f$ from $I$ to $2^{<\omega}$, and for $x\in 2^{<\omega}$, let $R_i(x)=\mathord{\downarrow}f(i)$. Then it is easy to see that $R_i$ satisfies \upR{}, \Rdown{}, and \Rref{}, so by Lemma \ref{SufficientForFull}, $(2^{<\omega},\sqsubseteq)$ equipped with $\mathcal{RO}(2^{<\omega},\sqsubseteq)$ and $\{R_i\}_{i\in I}$ is a full relational possibility frame, and setting $\mathrm{G}(x)=I$ gives us a universal agent-domain frame $\mathcal{F}$. Let $\mathcal{M}$ be any model based on $\mathcal{F}$ and $g_0:\mathsf{Var}\to I$. To see that $p\to \exists \mathrm{v} \Box_\mathrm{v} p$ is globally true in $\mathcal{M}$, observe that if $\mathcal{M},x\Vdash_{g_0} p$, then $\mathcal{M},x\Vdash_g \Box_\mathrm{v}p$ where $g(\mathrm{v})=f^{-1}(x)$, so $\mathcal{M},x\Vdash_{g_0}\exists \mathrm{v}\Box_\mathrm{v}p$. Finally, since there is no $x\in X$ and $i\in I$ such that for all $Z\in \mathcal{RO}(2^{<\omega},\sqsubseteq)$, $x\in \Box_i Z\vee\Box_i\neg Z$, the formula $\exists \mathrm{v} O(\mathrm{v})$ is not true at any possibility in the model.\end{proof}

Let us now consider a related example in the language $\mathcal{L}_S(I)$. For this, we draw an explicit connection to provability logic \cite{Artemov2005}. The key idea is to use the formula corresponding to G\"{o}del's Second Incompleteness Theorem:
\[\Box_i\neg\Box_i\bot\to \Box_i\bot.\]
The Second Incompleteness Theorem entails, for the theories it covers, that if a \textit{single} theory can prove every truth, then it is inconsistent, so there is an inconsistent theory: $S\bot$. World frames commit us to the idea that if every truth is provable in some theory or other in a class, then there is a single theory that proves every truth---hence the class contains an inconsistent theory by the Second Incompleteness Theorem. But we know from G\"{o}del that this is not so for every class of theories: it can be that every truth is provable in some theory or other in a class, yet there is no inconsistent theory in the class. This reasoning inspires the following world-incompleteness result in Theorem \ref{ProvThm}. Recall that the  Second Incompleteness formula is a special case (substituting $\bot$ for $p$) of the formula corresponding to L\"{o}b's Theorem in provability logic: \[ \Box_i (\Box_ip\to p)\to \Box_i p.\]

\begin{theorem}\label{ProvThm}$\,$\textnormal{
\begin{enumerate}
\item\label{ProvThm1} Any varying agent-domain full monotonic neighborhood world frame that validates $p\to Sp$ and $ \Box_i\neg\Box_i\bot\to \Box_i \bot$ for each $i\in I$ also validates $S\bot$.
\item\label{ProvThm2} There are varying agent-domain full relational possibility frames validating $p\to Sp$ and $\Box_i (\Box_i p\to p)\to \Box_i p$ for each $i\in I$ but not $S\bot$.
\end{enumerate}}
\end{theorem}

\begin{proof} For part \ref{ProvThm1}, let $\mathcal{F}$ be a frame satisfying the hypothesis. To show that $\mathcal{F}$ validates $S\bot$, consider any world $w$ in $\mathcal{F}$. Let $\mathcal{M}$ be a model based on $\mathcal{F}$ such that $\pi(p)=\{w\}$. Then since $\mathcal{F}$ validates $p\to Sp$, we have $\mathcal{M},w\Vdash Sp$, so there is some $i\in \mathrm{G}(w)$ such that $\mathcal{M},w\Vdash \Box_ip$ and hence $\{w\}\in N_i(w)$. Then by monotonicity, $N_i(w)$ contains every $A\subseteq W$ with $w\in A$. Thus, if  $\mathcal{M},w\Vdash \neg\Box_i \bot$, then $\mathcal{M},w\Vdash \Box_i\neg\Box_i\bot$. But $\mathcal{F}$ validates $ \Box_i\neg\Box_i\bot\to \Box_i \bot$, so we conclude that $\mathcal{M},w\Vdash \Box_i \bot$. Hence $\mathcal{M},w\Vdash S\bot$. Since $S\bot$ contains no propositional variables, it follows that $\mathcal{F},w$ validates $S\bot$. Since $w$ was arbitrary, $\mathcal{F}$ validates $S\bot$.

For part \ref{ProvThm2}, let $T$ be any tree in which every node has at least 2, but only finitely many, children. For $x,y\in T$, we set $x\sqsubseteq y$ if $x=y$ or $x$ is a descendent of $y$. Fix a bijection $f$ from $I$ to $T$. Then for any $i\in I$ and $x\in T$, where $c_1,\dots,c_n$ are the children of $f(i)$, define $R_i(x)$ as follows (recalling that $\mathord{\downarrow}c_j=\{y\in S\mid y\sqsubseteq c_j\}$):

\[R_i(x)= \begin{cases}
\underset{1<j\leq n}{\bigcup}\mathord{\downarrow}c_j &\mbox{if }f(i)\sqsubseteq x \\
\underset{k<j\leq n}{\bigcup}\mathord{\downarrow}c_j &\mbox{if }\exists k: x\sqsubseteq c_k\\
\underset{k_{i}<j\leq n}{\bigcup}\mathord{\downarrow}c_j & \mbox{otherwise},\end{cases}\] 
where $k_{i}$ can be chosen to be any integer from $1$ to $n$. See Figure \ref{GLtree} below.

We claim that $R_i$ satisfies the following properties:
\begin{enumerate}
\item \upR{} -- if $x'\sqsubseteq x$ and $x'R_iy'$, then $xR_iy'$;
\item \Rdown{} -- if $xR_iy$ and $y'\sqsubseteq y$, then $xR_iy'$;
\item \Rref{} -- if $xR_iy$, then $\exists x'\sqsubseteq x$ $\forall x''\sqsubseteq x'$ $\exists y'\sqsubseteq y$: $x'' R_i y'$.
\end{enumerate}

For \upR{}, this is equivalent to the condition that $x'\sqsubseteq x$ implies $R_i(x')\subseteq R_i(x)$, which is clear from the definition of $R_i$. That \Rdown{} holds is immediate from the definition of $R_i(x)$ as a downset. Finally, for \Rref{}, suppose $xR_iy$. Then either (i) $f(i)\sqsubseteq x$, (ii) $x\sqsubseteq c_k$ for some child $c_k$ of $f(i)$, or (iii) neither (i) nor (ii) holds. In addition, $y\sqsubseteq c_j$ for some child $c_j$ of $f(i)$ with $j>1$ in cases (i) and (iii) and $j>k$ in case (ii). In case (i), let $x':=c_1$, while in cases (ii) and (iii), let $x':=x$. Then note that that for all $x''\sqsubseteq x'$, we have $x''R_iy$.

Thus, by Lemma \ref{SufficientForFull}, $(T,\sqsubseteq)$ equipped with $\mathcal{RO}(T,\sqsubseteq)$ and $\{R_i\}_{i\in I}$  is a full relational possibility frame. Now define $\mathrm{G}:T\to \wp(I)\setminus \{\varnothing\}$ as follows:
\[\mathrm{G}(x)=\{i\in I\mid \forall x'\sqsubseteq x\, R_i(x')\neq\varnothing\}.\]
Then we claim that persistence and refinability hold for $\mathrm{G}$. Persistence is immediate from the definition. For refinability, suppose $i\in I\setminus\mathrm{G}(x)$, so there is an $x'\sqsubseteq x$ such that $R_i(x')=\varnothing$. Then by \upR{}, for all $x''\sqsubseteq x'$ we have  $R_i(x'')=\varnothing$ and hence $i\not\in \mathrm{G}(x'')$. Thus, we have a varying agent-domain full relational possibility frame. Let $\mathcal{M}$ be any model based on the frame. 

First, we show that $\Box_i (\Box_i p\to p)\to \Box_i p$ is true at every $x$ in $\mathcal{M}$. Suppose $\mathcal{M},x\Vdash \Box_i(\Box_ip\to p)$. If $R_i(x)=\varnothing$, then immediately $\mathcal{M},x\Vdash \Box_ip$. Otherwise $xR_ic_n$ where $c_n$ is the last of the children of $f(i)$ in our enumeration of $f(i)$'s children, so from $\mathcal{M},x\Vdash \Box_i(\Box_ip\to p)$ we have $\mathcal{M},c_n\Vdash \Box_ip\to p$. Since $R_i(c_n)=\varnothing$, we have $\mathcal{M},c_n\Vdash\Box_ip$ and hence $\mathcal{M},{c_n}\Vdash p$. Then since $R_i(c_{n-1})=\mathord{\downarrow}c_n$, we have $\mathcal{M},c_{n-1}\Vdash \Box_i p$. Hence if $x\sqsubseteq c_{n-1}$, we have $\mathcal{M},x\Vdash\Box_ip$, so we are done. If instead  $xR_ic_{n-1}$, then from $\mathcal{M},x\Vdash\Box(\Box_i p\to p)$ we have $\mathcal{M},c_{n-1}\Vdash\Box_ip\to p$, which with $\mathcal{M},c_{n-1}\Vdash \Box_i p$ from above yields $\mathcal{M},c_{n-1}\Vdash p$ and hence $\mathcal{M},c_{n-2}\Vdash \Box_ip$, so  $\mathcal{M},x\Vdash \Box_ip$ if $x\sqsubseteq c_{n-2}$.  Repeating this reasoning, if $x$ is under some child of $f(i)$, we obtain $\mathcal{M},x\Vdash \Box_ip$, and otherwise we obtain $\mathcal{M},c_k\Vdash p$ for every child $c_k$ of $f(i)$ and hence $\mathcal{M},x\Vdash \Box_ip$ for every $x$ in $\mathcal{M}$.

Next, we show that $p\to Sp$ is globally true. Suppose $\mathcal{M},x\Vdash p$ and $x'\sqsubseteq x$. Let $x''$ be the first child of $x'$, i.e., $c_1$, in our enumeration of the children of $x'$. Then observe that $f^{-1}(x')\in\mathrm{G}(x'')$ and $\mathcal{M},x''\Vdash \Box_{f^{-1}(x')}p$. Hence $\mathcal{M},x\Vdash Sp$.

Finally, it is immediate from our definition of $\mathrm{G}$ that if $\mathcal{M},x\Vdash \Box_i\bot$, then $i\not\in\mathrm{G}(x)$. Hence there is no possibility that makes $S\bot$ true.\end{proof}

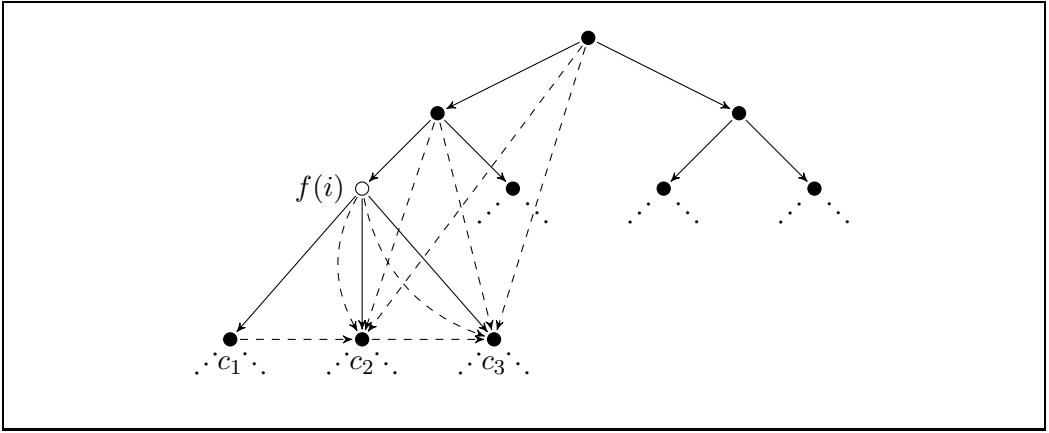
\begin{figure}
\begin{center}
\begin{tikzpicture}[yscale=1, ->,>=stealth',shorten >=1pt,shorten <=1pt, auto,node
distance=2cm,every loop/.style={<-,shorten <=1pt}]
\tikzstyle{every state}=[fill=gray!20,draw=none,text=black]
\node[circle,draw=black!100,fill=black!100, label=above:$$,inner sep=0pt,minimum size=.175cm] (0) at (0,0) {{}};
\node[circle,draw=black!100,fill=black!100, label=left:$$,inner sep=0pt,minimum size=.175cm] (00) at (-2,-1) {{}};
\node[circle,draw=black!100,fill=black!100, label=right:$$,inner sep=0pt,minimum size=.175cm] (01) at (2,-1) {{}};
\node[circle,draw=black!100, label=left:$f(i)$,inner sep=0pt,minimum size=.175cm] (000) at (-3,-2) {{}};

\node[circle,draw=black!100,fill=black!100, label=below:$c_1$,inner sep=0pt,minimum size=.175cm] (000a) at (-4.75,-4) {{}};
\node[circle,draw=black!100,fill=black!100, label=below:$c_2$,inner sep=0pt,minimum size=.175cm] (000b) at (-3,-4) {{}};
\node[circle,draw=black!100,fill=black!100, label=below:$c_3$,inner sep=0pt,minimum size=.175cm] (000c) at (-1.25,-4) {{}};

\node[circle,draw=black!100,fill=black!100, label=right:$$,inner sep=0pt,minimum size=.175cm] (001) at (-1,-2) {{}};
\node[circle,draw=black!100,fill=black!100, label=left:$$,inner sep=0pt,minimum size=.175cm] (010) at (1,-2) {{}};
\node[circle,draw=black!100,fill=black!100, label=right:$$,inner sep=0pt,minimum size=.175cm] (011) at (3,-2) {{}};
\path (0) edge[->] node {{}} (00);
\path (0) edge[->] node {{}} (01);
\path (00) edge[->] node {{}} (000);
\path (00) edge[->] node {{}} (001);
\path (01) edge[->] node {{}} (010);
\path (01) edge[->] node {{}} (011);

\path (000) edge[->] node {{}} (000a);
\path (000) edge[->] node {{}} (000b);
\path (000) edge[->] node {{}} (000c);

\path (0) edge[->,dashed] node {{}} (000b);
\path (0) edge[->,dashed] node {{}} (000c);

\path (00) edge[->,dashed] node {{}} (000b);
\path (00) edge[->,dashed] node {{}} (000c);

\path (000) edge[->,dashed,bend right] node {{}} (000b);
\path (000) edge[->,dashed,bend right] node {{}} (000c);
\path (000a) edge[->,dashed] node {{}} (000b);
\path (000b) edge[->,dashed] node {{}} (000c);

\node at (-5.05,-4.3) {{$ \rotatebox{45}{\dots}$}};
\node at (-4.41,-4.34) {{$ \rotatebox{-45}{\dots}$}};

\node at (-3.3,-4.3) {{$ \rotatebox{45}{\dots}$}};
\node at (-2.66,-4.34) {{$ \rotatebox{-45}{\dots}$}};

\node at (-1.55,-4.3) {{$ \rotatebox{45}{\dots}$}};
\node at (-0.91,-4.34) {{$ \rotatebox{-45}{\dots}$}};

\node at (-1.3,-2.3) {{$ \rotatebox{45}{\dots}$}};
\node at (-0.66,-2.34) {{$ \rotatebox{-45}{\dots}$}};

\node at (3.34,-2.34) {{$ \rotatebox{-45}{\dots}$}};
\node at (2.7,-2.3) {{$ \rotatebox{45}{\dots}$}};

\node at (1.34,-2.34) {{$ \rotatebox{-45}{\dots}$}};
\node at (0.7,-2.3) {{$ \rotatebox{45}{\dots}$}};

\end{tikzpicture}
\end{center}
\caption{illustration of the construction in the proof of Theorem \ref{ProvThm}.\ref{ProvThm2} for $n=k_i=3$. Solid lines are for $\sqsubseteq$ and dashed lines are for $R_i$. Every possibility with an accessibility arrow to $c_j$ also has accessibility arrows to all descendants of $c_j$, but these arrows are not shown in the diagram.}\label{GLtree}
\end{figure}

\begin{corollary}\label{SlogicWorldIncomplete} \textnormal{There is a varying-agent domain full relational possibility frame whose $\mathcal{L}_S(I)$-logic is not that of any varying agent-domain full monotonic neighborhood world frame}.\end{corollary}

For actual arithmetic interpretations of $\mathcal{L}_S(I)$, we may want $p\to Sp$ to be valid while ${\Box_i(p\to Sp)}$ is not, as in the following example.

\begin{example}\label{Arithmetic} Consider the language of arithmetic and the theories $\mathsf{Q}$ (Robinson's $\mathsf{Q}$), $\mathsf{EA}$ (Elementary Arithmetic), and $\mathsf{PA}$ (Peano Arithmetic). Let $\mathbb{T}$ be an arithmetically definable set of elementarily presented consistent extensions of $\mathsf{Q}$, containing at least $\mathsf{EA}$ and any consistent extension of $\mathsf{EA}$ with a single sentence. Let $T$ be an arithmetic predicate expressing that a number is the G\"{o}del number of a theory in $\mathbb{T}$ or the G\"{o}del number of a consistent extension of $\mathsf{EA}$ with a single sentence. The latter disjunct is redundant from the point of view of truth, but it ensures that $\mathsf{PA}$ proves ``if $x$ is the G\"{o}del number of a consistent extension of $\mathsf{EA}$ with a single sentence, then $T(x)$,'' which we will use below.

Say that a $\mathbb{T}$-\textit{realization} of a formula $\varphi$ of $\mathcal{L}_S(I)$ is a sentence of arithmetic obtained by: uniformly replacing the propositional variables in $\varphi$ by sentences of arithmetic; for each $i\in I$, uniforming replacing $\Box_i$ by the provability predicate of some theory in $\mathbb{T}$, given by its elementary presentation;\footnote{Note that different realizations may interpret $\Box_i$ as the provability predicate of different theories in $\mathbb{T}$, just as they may interpret $p$ as different arithmetic sentences.}  and replacing $S$ by the arithmetic predicate using $T$ that codes ``there is some theory in $\mathbb{T}$ that proves\dots''. Say that $\varphi$ is \textit{arithmetically $\mathbb{T}$-valid} if all of its $\mathbb{T}$-realizations are provable in $\mathsf{PA}$. (By analogy with modal validity, given a class $\mathfrak{T}$ of such sets $\mathbb{T}$, we could define $\varphi$ to be \textit{arithmetically $\mathfrak{T}$-valid} if $\varphi$ is arithmetically $\mathbb{T}$-valid for every $\mathbb{T}\in\mathfrak{T}$.) 

Then by L\"{o}b's theorem, $\Box_i (\Box_ip\to p)\to \Box_i p$
is arithmetically $\mathbb{T}$-valid for each $i\in I$, as is the Second Incompleteness Theorem formula, $ \Box_i\neg\Box_i\bot\to \Box_i \bot$. In addition, we claim that $p\to Sp$ 
is arithmetically $\mathbb{T}$-valid. First, $\mathsf{PA}$ proves all local reflection principles for $\mathsf{EA}$ \cite{Kreisel1968,Beklemishev2005}: 
\[\mathsf{PA}\vdash \mathsf{Prov}_\mathsf{EA}(\left\ulcorner\chi\right\urcorner)\to \chi \]
and hence
\[\mathsf{PA}\vdash \phi\to \neg \mathsf{Prov}_\mathsf{EA}(\left\ulcorner\phi \to 0=1 \right\urcorner).\]
In addition, we have the following formalized deduction theorem \cite{Feferman1960}:
\[\mathsf{PA}\vdash \mathsf{Prov}_{\mathsf{EA}+\phi}(\left\ulcorner\psi\right\urcorner)\leftrightarrow  \mathsf{Prov}_{\mathsf{EA}}(\left\ulcorner\phi\to\psi \right\urcorner).\]
From the previous two facts, we obtain:
\[\mathsf{PA}\vdash \phi \to \neg \mathsf{Prov}_{\mathsf{EA}+\phi}(\left\ulcorner 0=1\right\urcorner)\]
and hence 
\[\mathsf{PA}\vdash \phi \to \exists x(T(x)\wedge \mathsf{Prov}_x (\left\ulcorner\phi\right\urcorner)),\]
 as $\mathsf{EA}+\phi$ is a witness for the existential statement (recall our specification of the predicate $T$ above). Next, observe that $S\bot$ is not arithmetically $\mathbb{T}$-valid, because all theories in $\mathbb{T}$ are consistent and $\mathsf{PA}$ is sound. Finally, note that $\Box_i(p\to Sp)$ may fail to be arithmetically $\mathbb{T}$-valid, if we realize $i$ as $\mathsf{EA}$ and $\mathsf{EA}$ is unable to prove the consistency of other theories in $\mathbb{T}$ and hence their membership in $\mathbb{T}$. \end{example}

 To validate $p\to Sp$ but not $\Box_j(p\to Sp)$, we can use a possibility frame with a distinguished possibility---or equivalently, a distinguished principal filter---as in \S~\ref{QuasiSection}. Recall that polymodal $\mathsf{GL}$ (G\"{o}del-L\"{o}b logic) is the smallest normal polymodal logic containing the L\"{o}b axiom for each modality $\Box_i$.

\begin{theorem}\label{ProvThmC}$\,$\textnormal{
\begin{enumerate}
\item\label{ProvThm1C} Any varying agent-domain full monotonic neighborhood world frame with a distinguished principal filter (``with distinguished worlds'') that validates $p\to Sp$ and $ \Box_i\neg\Box_i\bot\to \Box_i \bot$ for each $i\in I$ also validates $S\bot$.
\item\label{ProvThm2C} For $j\in I$, there are varying agent-domain full relational possibility frames with a distinguished possibility validating $p\to Sp$ and all theorems of polymodal $\mathsf{GL}$ but not $S\bot$ or $\Box_j(p\to Sp)$.
\end{enumerate}}
\end{theorem}
\begin{proof} For part 1, pick any world $w$ that belongs to the proposition that generates the principal filter. Then re-run the proof of Theorem \ref{ProvThm}.\ref{ProvThm1}.

For part 2, partition $I$ into infinite sets $I_1$ and $I_2$ with $j\in I_2$, and take a possibility frame $\mathcal{F}$ for $\mathcal{L}(I_1)$ as in the proof of Theorem \ref{ProvThm}.\ref{ProvThm2}. We will build a frame based on the disjoint union of $\mathcal{F}$ and any world frame $\mathcal{G}$ for $\mathcal{L}(I_2)$ validating $\mathsf{GL}$ such that for some $w\in\mathcal{G}$, $R_i(w)\neq\varnothing$ for all $i\in I_2$. Keep all accessibility links between possibilities in $\mathcal{F}$ and between possibilities in $\mathcal{G}$. Then for every $i\in I_2$ and $x\in \mathcal{F}$, let $R_i(x)$ be the set of all worlds in $\mathcal{G}$. For every $i\in I_1$ and $x\in \mathcal{G}$, let $R_i(x)=\varnothing$. Then we have a possibility frame validating $\mathsf{GL}$. Moreover, at the root $r$ of $\mathcal{F}$, $p\to Sp$ is valid, but $\Box_j(p\to Sp)$ is not, due to the fact that $rR_jw$. Thus, we may select $r$ as the distinguished possibility to obtain the result.\end{proof}

\begin{remark} In response to Theorem \ref{ProvThmC}.\ref{ProvThm1C}, one may try to salvage the use of full world frame semantics for $\mathcal{L}_S(I)$ by employing distinguished \textit{non-principal} filters (e.g., the cofinite filter, in which case the semantic value of $p\to Sp$ will trivially belong to the filter whenever the semantic value of $p$ is a finite set). However, Theorem \ref{ProvThmC}.\ref{ProvThm2C} suggests that switching to possibility frames for $\mathcal{L}_S(I)$ may be a more natural approach.
\end{remark}

The use of varying agent-domains was essential for Theorem \ref{ProvThm}.\ref{ProvThm2}, as the following shows. 

\begin{proposition} \textnormal{Any \textit{universal agent-domain} quasi-normal relational possibility frame validating $p\to Sp$ and $ \Box_i\neg\Box_i\bot\to \Box_i \bot$ for each $i\in I$ also validates $S\bot$.}
\end{proposition}
\begin{proof} For a quasi-normal frame (Definition \ref{Quasi2}), say that $\varphi$ is \textit{super-valid} if $\varphi$ is valid on the underlying frame without the set $S_0$. Then note the following:
\begin{enumerate}
\item $\neg S\bot\to S \neg S\bot$ is valid, substituting $\neg S\bot$ for $p$ in  $p\to S p$;
\item $\Box_i\bot \to S\bot$ is super-valid in universal agent-domain frames
\item[$\Rightarrow$] $ \neg S\bot\to \neg \Box_i\bot $ is super-valid
\item[$\Rightarrow$] $\Box_i( \neg S\bot\to \neg \Box_i\bot)$ is super-valid\footnote{Although $\varphi$ being valid does not entail that $\Box_i\varphi$ is valid in a quasi-normal frame, $\varphi$ being \textit{super-valid} does.}
\item[$\Rightarrow$] $\Box_i \neg S\bot\to \Box_i \neg \Box_i \bot$ is super-valid
\item[$\Rightarrow$] $\Box_i \neg S\bot\to \Box_i\bot$ is valid since $ \Box_i\neg\Box_i\bot\to \Box_i \bot$ is valid
\item[$\Rightarrow$] $S\neg S\bot \to S\bot$ is valid
\item[$\Rightarrow$] $\neg S\bot \to S\bot$ is valid given 1
\item[$\Rightarrow$] $S\bot$ is valid,
\end{enumerate}
which completes the proof.\end{proof}

However, varying agent-domains are not required in general for semantically separating principles that cannot be separated by world frames. We give an example in $\mathcal{TML}$. Suppose that for every truth $p$ and theory $u$, there is some theory $v$ or other such that $u$ correctly believes that $v$ believes $p$:
\begin{equation}p\to \exists \mathrm{v} (\Box_\mathrm{v}p\wedge \Box_\mathrm{u} \Box_\mathrm{v} p).\label{UofV}\end{equation}
In a world frame, take $p$ to be a ``world proposition'' $\{w\}$. Then at $w$, the theory $u$ must believe that the witness to the existential quantifier is a negation-complete theory. But by the Second Incompleteness Theorem, a negation-complete theory must believe itself to be inconsistent: $\Box_\mathrm{v} \Box_\mathrm{v}\bot$. For by negation-completeness, we have $\Box_\mathrm{v} \Box_\mathrm{v}\bot\vee \Box_\mathrm{v}\neg\Box_\mathrm{v}\bot$, and the right disjunct entails $\Box_\mathrm{v}\bot$ and hence $\Box_\mathrm{v} \Box_\mathrm{v}\bot$. Since $u$ believes all of this, $u$ will believe that $v$ believes itself to be inconsistent. So as a world frame consequence of (\ref{UofV}), we will have:
\begin{equation}\exists \mathrm{v} \Box_\mathrm{u} \Box_\mathrm{v} \Box_\mathrm{v} \bot.\label{SingleVU}\end{equation}
Yet (\ref{SingleVU}) is not a possibility frame consequence of (\ref{UofV}).

\begin{theorem}\label{ProvThmB} $\,$\textnormal{
\begin{enumerate}
\item\label{ProvThm1B} Any  varying agent-domain full monotonic neighborhood world frame (with distinguished principal filter) that validates $p\to \exists \mathrm{v} (\Box_\mathrm{v}p\wedge \Box_\mathrm{u} \Box_\mathrm{v} p)$ and $\Box_\mathrm{v}\neg\Box_\mathrm{v}\bot\to \Box_\mathrm{v}\bot$ for each $v\in \mathsf{Var}$ also validates $\exists \mathrm{v} \Box_\mathrm{u} \Box_\mathrm{v} \Box_\mathrm{v}\bot$.
\item\label{ProvThm2B} There are universal agent-domain full relational possibility frames validating $p\to \exists \mathrm{v} (\Box_\mathrm{v}p\wedge \Box_\mathrm{u} \Box_\mathrm{v} p)$ and $\Box_\mathrm{v} (\Box_\mathrm{v} p\to p)\to \Box_\mathrm{v} p$ for each $\mathrm{v}\in \mathsf{Var}$ but not $\exists \mathrm{v} \Box_\mathrm{u} \Box_\mathrm{v} \Box_\mathrm{v}\bot$.
\end{enumerate}}
\end{theorem}

\begin{proof} For part \ref{ProvThm1}, let $\mathcal{F}$ be a  frame satisfying the hypothesis. To show that $\mathcal{F}$ validates $\exists \mathrm{v} \Box_\mathrm{u} \Box_\mathrm{v} \Box_\mathrm{v}\bot$, consider any world $w$ in $\mathcal{F}$ (that belongs to the proposition that generates the distinguished principal filter). Let $\mathcal{M}$ be a model based on $\mathcal{F}$ such that $\pi(p)=\{w\}$. Then since $\mathcal{F}$ validates $p\to \exists \mathrm{v} (\Box_\mathrm{v}p\wedge \Box_\mathrm{u} \Box_\mathrm{v} p)$, we have $\mathcal{M},w\Vdash_{g_0} \exists \mathrm{v} (\Box_\mathrm{v}p\wedge \Box_\mathrm{u} \Box_\mathrm{v} p)$, so there is an $i\in \mathrm{G}(w)$ such that for $g={g_0[\mathrm{v}:=i]}$, we have $\mathcal{M},w\Vdash_{g}  \Box_\mathrm{v}p\wedge \Box_\mathrm{u} \Box_\mathrm{v} p$, which implies (a) $\{w\}\in N_{g(\mathrm{v})}(w)$ and (b) $\llbracket \Box_\mathrm{v} p\rrbracket^\mathcal{M}_{g}\in N_{g(\mathrm{u})}(w)$. By (a) and monotonicity,  $N_{g(\mathrm{v})}(w)$ contains every $A\subseteq W$ with $w\in A$. Thus, if  $\mathcal{M},w\Vdash_g \neg\Box_\mathrm{v} \bot$, then $\mathcal{M},w\Vdash_g \Box_\mathrm{v}\neg\Box_\mathrm{v}\bot$. But $\mathcal{F}$ validates $\Box_\mathrm{v}\neg\Box_\mathrm{v}\bot\to \Box_\mathrm{v}\bot$, so we conclude that $\mathcal{M},w\Vdash_g \Box_\mathrm{v} \bot$. Hence $\llbracket p\rrbracket^\mathcal{M}_g\subseteq \llbracket \Box_\mathrm{v} \bot\rrbracket^\mathcal{M}_{g}$, so by monotonicity,  $\llbracket \Box_\mathrm{v} p\rrbracket^\mathcal{M}_{g}\subseteq \llbracket \Box_\mathrm{v}\Box_\mathrm{v} \bot\rrbracket^\mathcal{M}_{g}$. Then by (b) and monotonicity again, $\llbracket \Box_\mathrm{v} \Box_\mathrm{v}\bot\rrbracket^\mathcal{M}_{g}\in N_{g(\mathrm{u})}(w)$. Hence $\mathcal{M},w\Vdash_g \Box_\mathrm{u} \Box_\mathrm{v}\Box_\mathrm{v} \bot$, so $\mathcal{M},w\Vdash_{g_0}\exists_\mathrm{v} \Box_\mathrm{u} \Box_\mathrm{v}\Box_\mathrm{v} \bot$.

For part \ref{ProvThm2}, we use a possibility frame as in the proof of Theorem \ref{ProvThm}.\ref{ProvThm2} such that the root node $r$ has four children and every other node has at least three children (recall that the proof worked on the assumption that every node has a finite number of children greater than one) and (b) $k_i=1$ for all $i\in I$. But now we let $\mathrm{G}(x)=I$ for all $x\in T$, so we have a universal agent-domain frame. 

The argument that $\Box_\mathrm{v} (\Box_\mathrm{v} p\to p)\to \Box_\mathrm{v} p$ is valid on the frame is the same as in the proof of Theorem \ref{ProvThm}.\ref{ProvThm2}. The proof that $p\to \exists \mathrm{v} (\Box_\mathrm{v}p\wedge \Box_\mathrm{u} \Box_\mathrm{v} p)$ is valid is also similar to the argument that $p\to Sp$ is valid in the proof of Theorem \ref{ProvThm}.\ref{ProvThm2}. Simply add to that argument that not only do we have $\mathcal{M},x''\Vdash_g \Box_\mathrm{v}p$ for the assignment $g=g_0[\mathrm{v}:=f^{-1}(x')]$, but in fact  $\mathcal{M},y\Vdash_g \Box_\mathrm{v}p$ for any $y\in S$, which implies $\mathcal{M},x''\Vdash_g \Box_\mathrm{v}p\wedge \Box_\mathrm{u} \Box_\mathrm{v}p$. Hence $\mathcal{M},x\Vdash_{g_0} \exists \mathrm{v} (\Box_\mathrm{v}p\wedge \Box_\mathrm{u} \Box_\mathrm{v} p)$.

Finally, we claim that $ \exists \mathrm{v} \Box_\mathrm{u} \Box_\mathrm{v} \Box_\mathrm{v}\bot$ is not valid on $\mathcal{F}$. Suppose $f^{-1}(r)=i$. Let $g_0$ be an assignment such that $g_0(u)=i$. We claim that for any model $\mathcal{M}$ based on $\mathcal{F}$, $\mathcal{M},r\nVdash_{g_0} \exists \mathrm{v} \Box_\mathrm{u} \Box_\mathrm{v} \Box_\mathrm{v}\bot$. Where $c_1$ is the first child of $r$ in our enumeration of $r$'s children, it suffices to show $\mathcal{M},c_1\Vdash_{g_0}  \forall  \mathrm{v} \neg \Box_\mathrm{u} \Box_\mathrm{v} \Box_\mathrm{v}\bot$. For $j\in I$, let $g=g_0[\mathrm{v}:=j]$. We will show  $\mathcal{M},c_1\Vdash_g   \neg \Box_\mathrm{u} \Box_\mathrm{v} \Box_\mathrm{v}\bot$, i.e., for all $x\sqsubseteq c_1$, $\mathcal{M},x\nVdash_g \Box_\mathrm{u} \Box_\mathrm{v} \Box_\mathrm{v}\bot$. For all $x\sqsubseteq c_1$, we have $xR_ic_2$. Thus, it suffices to show that $\mathcal{M},c_2\nVdash_g \Box_\mathrm{v}\Box_\mathrm{v}\bot$. If $i=j$, then $\mathcal{M},c_2\nVdash_g \Box_\mathrm{v}\Box_\mathrm{v}\bot$ follows from the facts that $c_2R_ic_3$ and $c_3R_ic_4$. So suppose $i\neq j$. Then $f(j)\neq f(i)$, i.e., $f(j)\neq r$, so $c_3$ is not under any child of $f(j)$. It follows, given (b) above, that $c_3R_jc_2'$ where $c_2'$ is the second child of $f(j)$ in the enumeration of $f(j)$'s children. Then since $c_2'R_jc_3'$, where $c_3'$ is the third child in the enumeration of $f(j)$'s children,  we have $\mathcal{M},c_2\nVdash_g \Box_\mathrm{v}\Box_\mathrm{v}\bot$.\end{proof}

\begin{corollary}\label{UniversalCor} \textnormal{There is a universal agent-domain full relational possibility frame whose propositional term modal logic is not the logic of any class of  varying agent-domain full monotonic neighborhood world frames.}
\end{corollary}

A systematic study of possibility semantics for the five languages introduced above---and of the relation of possibility semantics to arithmetic interpretations of these languages for provability logic---remains to be carried out.

Note that in this chapter, the natural examples of Kripke incomplete but full-possibility-frame complete \textit{normal} modal logics all involve devices for quantifying over modalities or propositions. This leads to the following question.

\begin{question} Is there a ``natural'' example of a Kripke incomplete but full-possibility-frame complete normal propositional modal logic without devices for quantifying over modalities or propositions?
\end{question}

\subsubsection{Principal frames and $\mathcal{V}$-incompleteness in provability logic}\label{PrincRelFrame}

Since relational possibility frames are more general than Kripke frames, they can be used not only to overcome some Kripke incompleteness results but also to deepen other Kripke incompleteness results, i.e., to prove that some logics known to be Kripke incomplete are in fact incomplete with respect to even more general semantics. An open problem concerning modal incompleteness from \cite{Litak2005,Venema2007} was whether there is any modal logic that is not the logic of any class of $\mathcal{V}$-BAOs---whether there is a ``$\mathcal{V}$-incomplete'' logic. Ideas from possibility semantics led to a positive answer to this question in \cite{Litak2019}. Here we give a possibility semantic proof of two of the main theorems of \cite{Litak2019}, reproduced in Corollary \ref{Vincompleteness} below.

Let $\mathsf{vB}$ be the smallest normal modal logic containing the axiom 
\begin{itemize}
\item $\Box\neg\Box\bot\rightarrow \Box (\Box (\Box p\rightarrow p)\rightarrow p)$.
\end{itemize}
Van Benthem \cite{Benthem1979} introduced this logic, inspired by provability logic (the formula under the $\Box$ in the consequent, $\Box (\Box p\rightarrow p)\rightarrow p)$, is a theorem of the provability logic $\mathsf{GLS}$), and proved that it is Kripke frame incomplete: $\Box\neg\Box\bot\to \Box\bot$ is a Kripke frame consequence of $\mathsf{vB}$ but  is not a theorem of $\mathsf{vB}$.  

Another Kripke incomplete logic---this time not just inspired by provability logic but a key system in the field---is the bimodal provability logic $\mathsf{GLB}$, which is the smallest normal bimodal logic containing the following axioms:
\begin{itemize}
\item $\Box_i(\Box_i\to p)\to \Box_i p$ for $i\in \{0,1\}$;
\item $\Box_0 p\to\Box_1 p$;
\item $\Diamond_0p\to \Box_1\Diamond_0p$.
\end{itemize}
An arithmetic interpretation of $\mathsf{GLB}$ interprets $\Box_0$ as provability in $\mathsf{PA}$ and $\Box_1$ as provability in $\mathsf{PA}$ with one application of the $\omega$-rule  (or in $\mathsf{PA}$ together with all $\Pi^0_1$ arithmetic truths). Japaridze \cite{Japaridze1988} proved the arithmetic completeness of $\mathsf{GLB}$ under this interpretation, as well as the Kripke incompleteness of  $\mathsf{GLB}$: $\Box_1\bot$ is a Kripke frame consequence of $\mathsf{GLB}$ but is not a theorem of $\mathsf{GLB}$.

The logics $\mathsf{vB}$ and $\mathsf{GLB}$ are related by the fact that the following is a theorem of \textsf{GLB} \cite{Litak2019}:
\begin{itemize}
\item $\Box_1 (\Box_0 (\Box_0 p\rightarrow p)\rightarrow p)$.
\end{itemize}
Now we will show that $\Box\neg\Box\bot\to \Box\bot$ is a consequence of \textsf{vB} and $\Box_1\bot$ a consequence of $\mathsf{GLB}$ over a class of possibility frames that can represent all $\mathcal{V}$-BAOs.

\begin{lemma}\label{KeyVLem} \textnormal{Let $\mathcal{F}$ be a paradigm possibility frame in which every principal downset is an admissible set, and let $\varphi$ be a formula not containing $p$. Suppose that $\varphi\to \Box_1(\Box_0(\Box_0 \neg p\to \neg p)\to \neg p)$ (where $0$ and $1$ are two modal indices that may be equal) is globally true in every model based on $\mathcal{F}$ in which $\pi(p)$ is a principal downset. Then $\varphi \to\Box_1\bot$ is valid on $\mathcal{F}$.}
\end{lemma}
\begin{proof} Suppose $\mathcal{N}$ is a model based on $\mathcal{F}$ in which $\mathcal{N},x\Vdash \varphi$. We claim that $\mathcal{N},x\Vdash \Box_1\bot$. Suppose for contradiction that $\mathcal{N},x\nVdash\Box_1\bot$, so there is some $y$ such that $xR_1y$. Let $\mathcal{M}$ be the model on $\mathcal{F}$ that differs from $\mathcal{N}$ at most in that $\pi(p)=\mathord{\downarrow}y$. Then since $\varphi$ does not contain $p$, we have $\mathcal{M},x\Vdash\varphi$. We claim that 
\begin{equation}\mathcal{M},y\nVdash \Diamond_0 (p\wedge\Box_0\neg p).\label{DS}\end{equation}
For if 
\begin{equation}\mathcal{M},y\Vdash \Diamond_0 (p\wedge\Box_0\neg p),\label{FI1}\end{equation}
then $\mathord{\downarrow}y=\llbracket p\rrbracket^\mathcal{M}\subseteq \llbracket\Diamond_0 (p\wedge\Box_0\neg p)\rrbracket^\mathcal{M}$, in which case (\ref{FI1}) and the monotonicity of $\Diamond_0$ together imply 
\begin{equation}\mathcal{M},y\Vdash \Diamond_0 (\Diamond_0 (p\wedge\Box_0\neg p)\wedge\Box_0\neg p),\label{FI2}\end{equation}
which is impossible, since 
\begin{eqnarray*}
&&\Diamond_0 (\Diamond_0 (p\wedge\Box_0\neg p)\wedge\Box_0\neg p)\\
\Rightarrow && \Diamond_0(\Diamond_0 p\wedge\Box\neg p) \\
\Rightarrow && \Diamond_0\bot \\
\Rightarrow && \bot.
\end{eqnarray*}
 Given (\ref{DS}) and $\mathcal{M},y\Vdash p$, by \textit{refinability} and \textit{persistence} there is a $y'\sqsubseteq y$ such that $\mathcal{M},y'\Vdash \neg \Diamond_0 (p\wedge\Box_0\neg p)\wedge p$, which is equivalent to $\mathcal{M},y'\Vdash \Box_0 (\Box_0\neg p\to\neg p)\wedge p$. Now by \Rdown{} and \Rref{}, given $xR_1y$ and $y'\sqsubseteq y$, there is an  $x'\sqsubseteq x$ such that for all $x''\sqsubseteq x'$, there is a $y''\sqsubseteq y'$ such that $x''R_1y''$. Since $y''\sqsubseteq y'$ and $\mathcal{M},y'\Vdash \Box_0 (\Box_0\neg p\to\neg p)\wedge p$, we have $\mathcal{M},y''\Vdash \Box_0 (\Box_0\neg p\to\neg p)\wedge p$. Thus, for every $x''\sqsubseteq x'$, there is a $y''\sqsubseteq y'$ such that $x''R_1y''$ and $\mathcal{M},y''\Vdash \Box_0 (\Box_0\neg p\to\neg p)\wedge p$, which implies that $\mathcal{M},x'\Vdash \Diamond_1 (\Box_0 (\Box_0\neg p\to\neg p)\wedge p)$. Then since $x'\sqsubseteq x$ and $\mathcal{M},x\Vdash \varphi$, we have  $\mathcal{M},x'\Vdash \varphi\wedge \Diamond_1 (\Box_0 (\Box_0\neg p\to\neg p)\wedge p)$. But this contradicts the fact that $\mathcal{M},x'\Vdash \varphi\to \Box_1(\Box_0(\Box_0 \neg p\to \neg p)\to \neg p)$.\end{proof}
 
 \begin{theorem}\label{GLBinc}\textnormal{
 \begin{enumerate}
 \item \textsf{vB} is not the logic of any class of paradigm relational possibility frames in which every principal downset is an admissible set.
 \item $\mathsf{GLB}$ is not the logic of any class of paradigm relational possibility frames in which every principal downset is an admissible set.
 \end{enumerate}}
 \end{theorem}

\begin{proof} For part 1, by Lemma \ref{KeyVLem}, any such possibility frame validating \textsf{vB} also validates the non-theorem $\Box\neg\Box\bot\to \Box\bot$. For part, by Lemma \ref{KeyVLem}, any such possibility frame validating \textsf{GLB} also validates the non-theorem $\Box_1\bot$.\end{proof}

\begin{corollary}\label{Vincompleteness}$\,$\textnormal{
 \begin{enumerate}
 \item \textsf{vB} is not the logic of any class of $\mathcal{V}$-BAOs.
 \item $\mathsf{GLB}$ is not the logic of any class of $\mathcal{V}$-BAOs.
 \end{enumerate}}
\end{corollary}

\begin{proof} By Theorem \ref{VtoPoss}, every $\mathcal{V}$-BAO has the same logic as a paradigm relational possibility frame in which every principal downset is admissible. Thus, if $\mathsf{L}$ is the logic of a class of $\mathcal{V}$-BAOs, then it is the logic of a class of possibility frames in which every principal downset is admissible. Now apply Theorem \ref{GLBinc}.
\end{proof}

We have now seen several fruitful connections between possibility semantics and provability logic in  Remark \ref{QArith}, \S~\ref{FullRelFrameSection}, and the results above.

\subsubsection{General frames}\label{GenRelFrame}

As in \S~\ref{GenFrame1}, where we saw that general neighborhood possibility frames can represent arbitrary BAEs, let us briefly consider how general relational possibility frames can represent arbitrary BAOs---choice free.

\begin{proposition}[\cite{Holliday2018}]\label{BAOstoFrames} \textnormal{Given any BAO $\mathbb{B}=( B, \{\Box_i\}_{i\in I})$, define  
\[\mathbb{B}_\mathsf{f}=(B_\mathsf{f}, \{R_i\}_{i\in I})\mbox{ and }\mathbb{B}_\mathsf{g}=(B_\mathsf{g}, \{R_i\}_{i\in I}) \]
where $B_\mathsf{f}$ and $B_\mathsf{g}$ are as in Definition \ref{FiltFrameDef}, and for $F,F'\in \mathsf{PropFilt}(B)$, \[\mbox{$FR_iF'$ iff for all $a\in B$, $\Box_ia\in F$ implies $a\in F'$}.\]
Then:
\begin{enumerate}
\item $\mathbb{B}_\mathsf{f}$ and $\mathbb{B}_\mathsf{g}$ are strong relational possibility frames;
\item\label{BAEstoFrames2} $(\mathbb{B}_\mathsf{f})^\mathsf{b}$ is (up to isomorphism) the canonical extension of $\mathbb{B}$;
\item $(\mathbb{B}_\mathsf{g})^\mathsf{b}$ is isomorphic to $\mathbb{B}$.
\end{enumerate}}
\end{proposition}

\begin{proposition}[\cite{Holliday2018}]\textnormal{Let $\mathcal{F}= ( S,\sqsubseteq,P, \{R_i\}_{i\in I}) $ be a general relational possibility frame. Then $\mathcal{F}$ is isomorphic to $(\mathcal{F}^\mathsf{b})_\mathsf{g}$ if and only if $( S,\sqsubseteq,P )$ is a filter-descriptive possibility frame and $\mathcal{F}$ is $R$-tight (recall Definition \ref{RTightDef}). We also call such an $\mathcal{F}$ \textit{filter-descriptive}.}
\end{proposition}

Morphisms between filter-descriptive relational possibility frames are \\ p-morphisms as in Definition \ref{PossMorphs} such that the inverse image of any admissible proposition is also an admissible proposition.

\begin{definition} Given relational possibility frames $\mathcal{F}= ( S,\sqsubseteq, P, \{R_i\}_{i\in I}) $ and $\mathcal{F}'= ( S',\sqsubseteq',P' , \{R'_i\}_{i\in I}) $, a \textit{p-morphism} from $\mathcal{F}$ to $\mathcal{F}'$ is a map $h:S\to S'$ satisfying the back and forth conditions 1, $2'$, $3$, and $4'$ from Definition \ref{PossMorphs}, as well as the condition that for all $ Z'\in P'$, $h^{-1}[Z']\in P$.\end{definition}

We can now relate filter-descriptive relational possibility frames and BAOs categorically as follows.

\begin{proposition} $\,$\textnormal{
\begin{enumerate}
\item Filter-descriptive relational possibility frames together with p-morhisms \\ form a category, \textbf{FiltRelPoss} \cite{Holliday2018};
\item BAOs together with BAO-homomorphisms form a category, \textbf{BAO}.
\end{enumerate}}
\end{proposition}

\begin{theorem}[\cite{Holliday2018}]\label{FiltRelPossThm} \textnormal{\textbf{FiltRelPoss} is dually equivalent to \textbf{BAO}.}
\end{theorem}

\begin{remark} Theorem \ref{FiltRelPossThm} is a choice-free possibility semantic analogue of Goldblatt's \cite{Goldblatt1974} result that the category of descriptive world frames is dually equivalent to \textbf{BAO}. In \cite{Holliday2018} it is also shown that \textbf{FiltRelPoss} is a reflective subcategory of the category \textbf{Poss} of \textit{all} relational possibility frames together with maps called \textit{possibility morphisms}. This is an analogue of Goldblatt's \cite{Goldblatt2006b} result that the category of descriptive world frames is a reflective subcategory of the category of all general world frames with \textit{modal maps}.\end{remark}

Building on \S~\ref{TopFrameSection1},  general relational possibility frames can be related to topological relational possibility frames by using the admissible propositions to generate a topology. Filter-descriptive relational possibility frames then correspond to ``modal UV-spaces'' \cite{BH2018}.

\subsection{Functional frames}\label{FunctionalPoss}

The move from worlds to possibilities makes possible a move from accessibility relations to accessibility functions. Suppose $\mathcal{F}=(S,\sqsubseteq, P, \{R_i\}_{i\in I})$ is a possibility frame for $\mathcal{L}(I)$ such that for each  $i\in I$ and $x\in S$, there is an $f_i(x)\in S$ such that \[R_i(x)=\mathord{\downarrow}f_i(x),\]
recalling that $\mathord{\downarrow}y=\{y'\in S\mid y'\sqsubseteq y\}$. Then we may replace the relations $R_i$ by functions $f_i$, obtaining a functional frame $\mathcal{F}=(S,\sqsubseteq, P, \{f_i\}_{i\in I})$,  so that the truth clause for $\Box_i$ becomes:
\[\mathcal{M},x\Vdash \Box_i \varphi\mbox{ iff }\mathcal{M},f_i(x)\Vdash \varphi.\]
This idea was developed in \cite{Holliday2014} (cf.~\cite{HM2021}). For a reading of $\Box_i$ as a belief/knowledge modality, $f_i(x)$ represents ``the world as agent $i$ in $x$ believes/knows it to be,'' which may of course be partial. As observed in \cite{Holliday2018}, full functional possibility frames correspond to complete $\mathcal{T}$-BAOs as in \S~\ref{VtoFrame}, which we saw were equivalent to complete $\mathcal{V}$-BAOs. Hence we may always switch from a full relational possibility frames to a full functional possibility frame realizing the same~BAO.

For functional frames, we have the following elegant analogues of the conditions on the interaction of $R_i$ and $\sqsubseteq$ from \S~\ref{RelationalFrameSection}:
\begin{itemize}
\item \textbf{$f$-persistence}: if $x'\sqsubseteq x$, then $f_i(x')\sqsubseteq f_i(x)$;
\item \textbf{$f$-refinability}: if $y\sqsubseteq f_i(x)$, then $\exists x'\sqsubseteq x$ $\forall x'\sqsubseteq x$  $f_i(x'')\comp y$.
\end{itemize}

\begin{example} The full possibility frame constructed in the proof of Proposition \ref{NoOmniscience} may be viewed as a functional possibility frame.
\end{example}

\subsection{First-order modal logic}\label{QuantModalSection}

In this section, we combine possibility semantics for first-order logic (\S~\ref{FOSection}) and possibility semantics for modal logic (\S~\ref{ModalSection}) to obtain possibility semantics for first-order modal logic. This was done first by Harrison-Trainor \cite{HT2019}, though our setup below will be slightly different. The language is defined just as in first-order logic but with the extra inductive clause that if $\varphi$ is a formula, so is $\Box_i\varphi$ for any $i\in I$. 

In addition to adding accessibility relations (or neighborhood functions) to the first-order possibility models of Definition \ref{FOmodel} to interpret the $\Box_i$, we add a \textit{varying domain} function for the usual reasons that varying domains are allowed in first-order modal logic (see \cite[\S~4.7]{Fitting1998}). As a result, the base logic is a \textit{free logic} instead of classical first-order logic, as $\forall x\varphi \to \varphi^x_t$ is no longer valid and must be replaced by \[(\forall x\varphi\wedge \exists x \,t=x)\to \varphi^x_t.\]

 For simplicity, here we only consider full frames in which all regular open downsets are admissible propositions.

\begin{definition} A (\textit{full}) \textit{first-order relational possibility frame} is a tuple \\ $\mathbf{F}=(S,\sqsubseteq, D,\asymp, d,\{R_i\}_{i\in I})$  where:
\begin{enumerate}
\item $(S,\sqsubseteq)$ is a poset;
\item $D$ is a nonempty set;
\item $\asymp$ is a function assigning to each $s\in S$ an equivalence relation $\asymp_s$ on $D$ as in Definition \ref{FOmodel};
\item $d$ is a function assigning to each $s\in S$ a subset $d(s)\subseteq D$ satisfying:
\item[$\bullet$] \textit{persistence} for $d$: if $a\in d(s)$, $s'\sqsubseteq s$, and $a\asymp_{s'} b$, then $b\in d(s')$;
\item[$\bullet$] \textit{refinability} for $d$: if $a\in D\setminus d(s)$, then $\exists s'\sqsubseteq s$ $\forall s''\sqsubseteq s'$  $a\not\in d(s'')$.
\item $(S,\sqsubseteq, \mathcal{RO}(S,\sqsubseteq), \{R_i\}_{i\in I})$ is a full relational possibility frame as in Definition~\ref{RellPossFrame}.
\end{enumerate}
A \textit{first-order relational world frame} is a first-order relational possibility frame in which $\sqsubseteq$ is the identity relation.

A \textit{first-order relational possibility model} is a tuple \[\mathbf{M}=(S,\sqsubseteq, D,\asymp, d,\{R_i\}_{i\in I},V)\] where $(S,\sqsubseteq, D,\asymp, d,\{R_i\}_{i\in I})$ is a first-order relational possibility frame and \\ $(S,\sqsubseteq, D, \asymp, V)$ is a first-order possibility model as in Definition \ref{FOmodel}.
\end{definition}

\begin{definition}We define the satisfaction relation $\mathbf{M},s\Vdash_g\varphi$ as in Definition \ref{FOPossSat} but with a modified quantifier clause and the additional modal clause:
\begin{itemize}
\item $\mathbf{M},s\Vdash_g \forall x \varphi$ iff for all $a\in d(s)$, $\mathbf{M},s\Vdash_{g[x:=a]} \varphi$.
\item $\mathbf{M},s\Vdash_g \Box_i \varphi$ iff for all $s'\in R_i(s)$, $\mathbf{M},s'\Vdash_g\varphi$.
\end{itemize}
\end{definition}

In addition to satisfaction, we have the usual notion of frame validity from modal logic.

\begin{definition} Given a first-order relational possibility frame $\mathbf{F}$ and formula $\varphi$, we say that $\varphi$ is \textit{valid on $\mathbf{F}$} if for every model $\mathbf{M}$ based on $\mathbf{F}$, every possibility $s$ in $\mathbf{F}$, and every variable assignment $g$, we have  $\mathbf{M},s\Vdash_g\varphi$.
\end{definition}

In the first-order setting, we have some of the simplest world-incompleteness results. Consider a first-order modal language with a single modality $\Box$ and a single unary predicate $Q$. Let $\mathsf{E}(x)$ (``$x$ exists'') abbreviate the formula $\exists x' \, x'=x$, where $x'$ is a variable distinct from $x$. Now consider the following principle:
\begin{equation}Q(c)\to \exists x\Box (\mathsf{E}(x)\leftrightarrow Q(c)).\tag{\textsc{Fact}}\label{Fact}\end{equation}
Intuitively, (\ref{Fact}) says that if $c$ has property $Q$, then there exists an $x$---namely \textit{the fact that $c$ has property $Q$}---such that necessarily, $x$ exists iff $c$ has property $Q$. This principle was considered for first-order modal logic by Fine \cite[p.~88]{Fine1982}. It can be traced back to Bertrand Russell's logical atomism, as explained in \cite{Klement2020}:

\begin{quote}
The simplest propositions in the language of \textit{Principia Mathematica} are what Russell there called ``elementary propositions'', which take forms such as ``$a$ has quality $q$'', ``$a$ has relation [in intension] $R$ to $b$'', or ``$a$ and $b$ and $c$ stand in relation $S$'' (PM, 43-44). Such propositions consist of a simple predicate, representing either a quality or a relation, and a number of proper names. According to Russell, such a proposition is true when there is a corresponding fact or complex, composed of the entities named by the predicate and proper names related to each other in the appropriate way. E.g., the proposition ``$a$ has relation $R$ to $b$'' is true if there exists a corresponding complex in which the entity $a$ is related by the relation $R$ to the entity $b$. If there is no corresponding complex, then the proposition is false.
\end{quote}

Any world frame validating (\ref{Fact}) validates the additional principle, also considered by Fine \cite[p.~88]{Fine1982}, that there is a single fact---the ``world fact''---whose existence entails the existence of everything else:
\begin{equation}\exists x\forall y\Box(\mathsf{E}(x)\to \mathsf{E}(y)).\tag{\textsc{World}}\label{World}\end{equation}
Yet (\ref{World}) is not a theorem of the minimal free first-order modal logic containing as axioms (\ref{Fact}) and the axioms of \textsf{S5} for $\Box$. Both of these claims are shown by the following.

\begin{proposition}\label{WorldFact}$\,$\textnormal{
\begin{enumerate}
\item\label{WorldFact1} Any first-order relational world frame that validates (\ref{Fact}) also validates  (\ref{World}).
\item\label{WorldFact2} There are first-order relational possibility frames in which $R$ is the universal relation that validate (\ref{Fact}) but not (\ref{World}).
\end{enumerate}}
\end{proposition}

\begin{proof} For part \ref{WorldFact1}, suppose (\ref{Fact}) is valid on a world frame $\mathbf{F}$. Define a model $\mathbf{M}$ based on $\mathbf{F}$ where $Q(c)$ is true only at a single world $w$. Then since (\ref{Fact}) is valid, it follows that there is an $a\in d(w)$ such that for all $v\in R(w)$, if $v\neq w$, then $a\not\in d(v)$. Thus, $\mathsf{E}(x)$ is true only at $w$ relative to the relevant variable assignment mapping $x$ to $a$. It follows that $a$ is a witness to the truth of (\ref{World}) at $w$.

For part \ref{WorldFact2}, we define a first-order relational possibility frame based on the full infinite binary tree $2^{<\omega}$ regarded as a poset $(2^{<\omega},\sqsubseteq)$ as in Examples \ref{BinaryTree0} and \ref{BinaryTree0b}. Let $D$, the domain of objects, be $\mathcal{RO}(S,\sqsubseteq)$, the set of propositions. For each $s\in 2^{<\omega}$, let $\asymp_s$ be the identity relation, and let $d(s)$ be the set of all sets from $\mathcal{RO}(S,\sqsubseteq)$ to which $s$ belongs. Then note that \textit{persistence} and \textit{refinability}  for $d$ both hold. Finally, let $R$ be the universal relation on $2^{<\omega}$. Then we have defined a first-order relational possibility frame. Note that for any proposition $Z\in  \mathcal{RO}(S,\sqsubseteq)$, there is an object in $D$, namely $Z$ itself, such that for all $s\in 2^{<\omega}$, \[Z\in d(s)\mbox{ iff }s\in Z.\] It follows that (\ref{Fact}) is valid on the frame. But since $\mathcal{RO}(S,\sqsubseteq)$ is atomless, it is easy to see that (\ref{World}) is not valid on the frame.
\end{proof}

\begin{remark} In part \ref{WorldFact1}, we can more generally consider first-order \textit{monotonic neighborhood} world frames \cite{Arlo-Costa2006} and obtain the same result that any such frame that validates (\ref{Fact}) validates (\ref{World}), which yields the following.
\end{remark}
\begin{corollary}\label{FOincomp} \textnormal{There is a first-order modal logic that is the logic of a first-order relational possibility frame but not of any class of first-order monotonic neighborhood world frames.}
\end{corollary}

The theory of possibility semantics for first-order modal logic is still in its infancy. Any logic complete with respect to first-order world frames is immediately complete with respect to first-order possibility frames, but it remains to systematically study first-order modal logics that are possibility complete but world incomplete.

\section{Further connections and directions}

In this section, we briefly cover some connections between possibility semantics and other frameworks (for more connections, see \cite[\S~8.1]{Holliday2018}), namely interval semantics for temporal logic (\S~\ref{Intervals}) and Kripke semantics for bimodal logics (\S~\ref{Bimodal}).  We then turn to two non-classical generalizations of possibility semantics, namely possibility semantics for intuitionistic logic (\S~\ref{IntuitionisticSection}) and inquisitive logic (\S~\ref{InquisitiveSection}). As noted in \S~\ref{IntroSection}, another impossibility non-classical generalization is possibility semantics for orthologic, for which we refer the reader to \cite{HM2021}.

\subsection{Interval semantics}\label{Intervals}

Humberstone's \cite{Humberstone1981} original development of possibility semantics for modal logic was inspired by \textit{interval semantics} for temporal logic (see \cite{Humberstone1979}, \cite[Ex.~4.1.12]{Humberstone2015}). In such semantics, the truth of a  formula of temporal logic is defined not with respect to an \textit{instant} of time, as in the usual semantics (see, e.g.,  \cite{Goranko2020}), but with respect to an \textit{interval} of time. For influential early developments of interval semantics, see  \cite{Humberstone1979,Kamp1979,Roper1980,Benthem1981}. We will briefly compare one version of interval semantics, due to R\"{o}per  \cite{Roper1980}, to possibility semantics for temporal logic.

Given a relational possibility frame as in \S~\ref{RelationalFrameSection}, call each $x\in S$ a \textit{stretch} of time, and take $x\sqsubseteq y$ to mean that $x$ is a \textit{substretch} of $y$. R\"{o}per \cite{Roper1980} argued that propositions about ``states and processes'' should correspond to regular downsets in $(S,\sqsubseteq)$. The main difference between R\"{o}per's approach and one using relational possibility frames concerns the accessibility relation used to interpret temporal modalities. In possibility semantics for temporal logic, we adopt the following reading of $xRy$:
\[xRy\mbox{ iff  the stretch $x$ begins before the stretch $y$ does}.\]
We claim that under this interpretation of $R$, the conditions \upR{}, \Rdown{}, \Rref{}, and \Rdense{} from \S~\ref{RelationalFrameSection} all make good sense, not only for the specific temporal poset in Example \ref{TempEx} but more generally for any poset of temporal stretches.

We then define ``$\varphi$ is true throughout the stretch $x$'' with our usual clauses, including those for $\Box$ and $\Diamond$ (in frames satisfying \Rdown{}), now read as ``it will always be that\dots'' and ``it will be sometime be that\dots'':
\begin{itemize}
\item $\mathcal{M},x\Vdash \Box\varphi$ iff  $\forall y\in S$, if $xRy$ then $\mathcal{M},y\Vdash\varphi$.
\item $\mathcal{M},x\Vdash\Diamond\varphi$ iff $\forall x'\sqsubseteq x$ $\exists y\in S$: $x'Ry$ and $\mathcal{M},y\Vdash\varphi$.
\end{itemize}

Instead of working with the relation $xRy$ interpreted to mean ``$x$ begins before $y$,'' R\"{o}per, like many authors in the literature on interval semantics, works with a relation $>$ interpreted as follows:
\[x>y\mbox{ iff ``$x$ wholly precedes $y$.''}\]  
Now the following principle, which yields \textit{persistence} for modal formulas in possibility semantics, holds for $Q=R$ but not for $Q=\,>$:
\[\mbox{if $x'\sqsubseteq x$ and $x'Qy'$, then $xQy'$}.\]
(For $R$, this is just \upR{}.) Thus, R\"{o}per is forced to build downward persistence into his semantic clauses for $\Box$ and $\Diamond$ that use $>$:
\begin{itemize}
\item $\mathcal{M},x\Vdash \Box\varphi$ iff $\forall x'\sqsubseteq x$  $\forall y\in S$, if $x'>y$ then $\mathcal{M},y\Vdash\varphi$.
\item $\mathcal{M},x\Vdash \Diamond\varphi$ iff $\forall x'\sqsubseteq x$ $\exists x''\sqsubseteq x'$ $\exists y\in S$: $x''>y$ and $\mathcal{M},y\Vdash\varphi$.
\end{itemize}
Thus, by working with $>$ instead of $R$, R\"{o}per needs one more quantifier in each clause. However, the two approaches can be related by assuming the following:
\[x>y\mbox{ iff }\forall x'\sqsubseteq x\; x'Ry,\]
i.e., $x$ wholly precedes $y$ iff every substretch of $x$ begins before $y$ does. Indeed, in R\"{o}per's canonical model construction in \cite{Roper1980}, he defines $x>y$ in precisely this way from the canonical relation that we would take to be $R$.

There is much more than this brief sketch to say about the relation between possibility semantics and interval semantics. We refer the interested reader to van Benthem's textbook \cite{Benthem1983a} as a starting point for systematic comparison.

\subsection{A bimodal perspective}\label{Bimodal}

Given a full relational possibility frame $(S,\sqsubseteq, \mathcal{RO}(S,\sqsubseteq),R)$ for a unimodal language, we may regard the triple $(S,\sqsubseteq,R)$ as a \textit{bimodal Kripke frame}, or in possibilistic terms, as a bimodal full world frame  $(S,=,\wp(S),\{\sqsubseteq, R\})$ with $\sqsubseteq$ now thought of as an accessibility relation for a modality $\Box_\sqsubseteq$ alongside the modality $\Box_R$, and the poset $(S,=)$ now being discrete. Corresponding to this semantic transformation is the following syntactic translation:
\begin{itemize}
\item $\mathrm{p}(p)=\Box_\sqsubseteq\Diamond_\sqsubseteq p$;
\item $\mathrm{p}(\neg\varphi)=\Box_\sqsubseteq\neg \mathrm{p}(\varphi)$;
\item $\mathrm{p}(\varphi\wedge\psi)=\mathrm{p}(\varphi)\wedge \mathrm{p}(\psi)$;
\item $\mathrm{p}(\Box\varphi)=\Box_R\mathrm{p}(\varphi)$.
\end{itemize}
This \textit{bimodal perspective} on possibility semantics is the topic of \cite{Benthem2017}, where it is shown to lead to interesting translations between some well-known unimodal logics and bimodal logics that can be seen as dynamic topological logics.

\subsection{Intuitionistic case}\label{IntuitionisticSection}\label{IntuitionisticSection}

In this section, we briefly sketch the generalization of possibility semantics for intuitionistic logic. A much more thorough investigation of these ideas appears in \cite{B&H2019}, which studies a hierarchy of semantics for intuitionistic logic.

Recall that algebraic semantics for intuitionistic logic is given using Heyting algebras. A Heyting algebra $H$ may be defined as a bounded distributive lattice with a binary operation $\to$ that is a residual of $\wedge$, i.e., such that for all $a,b,x\in H$,
\[a\wedge x\leq b\mbox{ iff }x\leq a\to b.\]
A complete Heyting algebra is a Heyting algebra that is complete as a lattice. Complete Heyting algebras, also called \textit{locales} or \textit{frames} in pointfree topology,\footnote{The use of `frame' in both modal logic and pointfree topology for  different kinds of objects is unfortunate but now well established.} can be equivalently defined as complete lattices $L$ satisfying the join-infinite distributive law \cite[p.~128]{Birkhoff1967} stating that for all $a\in L$ and $B\subseteq L$,
\[a\wedge \bigvee B= \bigvee \{a\wedge b\mid b\in B\}.\label{JoinLaw}\]
Given any poset $(S,\sqsubseteq)$, the collection $\mathsf{Down}(S,\sqsubseteq)$ of all downward closed sets, ordered by inclusion, forms a locale. More generally, given any topological space $X$, the collection $\Omega(X)$ of all open sets of $X$, ordered by inclusion, forms a locale. These two facts are the basis for the Kripke frame semantics \cite{Dummett1959,Kripke1963b,Grzegorczyk1964,Kripke1965} and topological semantics \cite{Sto37,Tarski1938} for intuitionistic logic, respectively. However, only special locales arise in these ways, as shown by the following well-known results (see, e.g., \cite[Prop.~1.1]{Davey1979} for part 1 and \cite[Prop.~5.3]{Picado2012} for part 2).

\begin{proposition}\label{SpecialLocales}\textnormal{Let $L$ be a lattice.
\begin{enumerate}
\item $L$ is isomorphic to $\mathsf{Down}(S,\sqsubseteq)$ for a poset $(S,\sqsubseteq)$ iff $L$ is a locale in which every element is a join of completely join-prime elements (where an element $p\neq 0$ is completely join-prime if $p\leq \bigvee A$ implies $p\leq a$ for some $x\in A$).
\item $L$ is isomorphic to $\Omega(X)$ for some space $X$ iff $L$ is a locale in which every element is a meet of meet-prime elements (where an element $p\neq 1$ is meet-prime if $a\wedge b\leq p$ implies that $a\leq p$ or $b\leq p$).
\end{enumerate}}
\end{proposition}

Just as classical possibility semantics seeks to represent all complete Boolean algebras, not just atomic ones, intuitionistic possibility semantics seeks to represent all locales, not just the special ones in Proposition \ref{SpecialLocales}. This is achieved by equipping a poset with an additional datum. To motivate the key definition (Definition \ref{NucFrames}), we first recall relevant notions and results (cf.~\cite{Fourman1979}).

\begin{definition} A \textit{nucleus} on a Heyting algebra $H$ is a unary function \\ ${j:H\to H}$ such that for all $a,b\in H$:
\begin{enumerate}
\item $a\leq ja$ (increasing);
\item $jja\leq ja$ (idempotent);
\item  $j(a\wedge b)=ja\wedge jb$ (multiplicative).
\end{enumerate}
A nucleus $j$ is \textit{dense} if $j0=0$.

A \textit{nuclear algebra} is a pair $(H,j)$ of a Heyting algebra $H$ and a nucleus $j$ on~$H$.
\end{definition}

\begin{example}\label{NucEx} For any Heyting algebra $H$, the operation $j_{\neg\neg}$ defined by \\ $j_{\neg\neg}(a)=\neg\neg a$ is a nucleus, the \textit{double negation nucleus}. In the case of a poset $(S,\sqsubseteq)$, the nucleus $j_{\neg\neg}$ on $\mathsf{Down}(S,\sqsubseteq)$ is given by:
\[ j_{\neg\neg}Z=\{x\in S\mid \forall x'\sqsubseteq x\,\exists x''\sqsubseteq x':x''\in Z\}.\]
Another nucleus $j_B$ on $\mathsf{Down}(S,\sqsubseteq)$, the \textit{Beth nucleus}, is defined by:
\[j_B Z=\{x\in S\mid \mbox{every maximal chain in $(S,\sqsubseteq)$ containing $x$ intersects }Z\}.\]
We return to the connection between these nuclei and possibility semantics below.
\end{example}

For our purposes, the key point about nuclei is that the fixpoints of a nucleus on a Heyting algebra again form a Heyting algebra, as in the following well known-theorem (see, e.g., \cite{Macnab1981} or \cite[p.~71]{Dragalin1988}), where part 3 gives an algebraic version of Glivenko's Theorem (see, e.g., \cite[p.~134]{RS63}).

\begin{theorem}\label{FixpointThm} \textnormal{For any nuclear algebra $(H,j)$, consider the collection \[H_j=\{a\in H\mid ja=a\}\] of fixpoints of $j$.
\begin{enumerate}
\item $H_j$ is a Heyting algebra, called the \textit{algebra of fixpoints} in $(H,j)$, under the following operations for $a,b\in H_j$: 
\begin{eqnarray*}
0_j&=&j0;\\
a\wedge_j b&=& a\wedge b;\\
a\vee_j b &=& j(a\vee b);\\
a\to_j b &=& a\to b. 
\end{eqnarray*}
\item If $H$ is a locale, then so is $H_j$, where $\bigwedge_j A = \bigwedge A$ and $\bigvee_j A = j(\bigvee A)$ for $A\subseteq H_j$.
\item\label{FixpointThm3} If $j$ is the nucleus of double negation, then $H_j$ is a Boolean algebra.
\end{enumerate}}
 \end{theorem}

It is easy to see in light of part \ref{FixpointThm3} that classical possibility semantics is based on taking the double negation fixpoints in the Heyting algebra $\mathsf{Down}(S,\sqsubseteq)$ of downsets of a poset. By switching the nucleus from double negation to some other non-Boolean nucleus, we obtain an intuitionistic generalization.

\begin{definition}\label{NucFrames} A \textit{nuclear frame} is a triple $(S,\sqsubseteq,j)$ such that $(S,\sqsubseteq)$ is a poset and $j$ is a nucleus on $\mathsf{Down}(S,\sqsubseteq)$. The nuclear frame is \textit{dense} if $j$ is dense.
\end{definition}

Then any nuclear frame gives us a locale $\mathsf{Down}(S,\sqsubseteq)_j$. As explained in \cite{Dragalin1979,Dragalin1988,B&H2019}, Beth semantics for intuitionistic logic \cite{Beth1956}, which predated Kripke semantics, can be seen as a special case of this construction. Crucially, as shown by Dragalin \cite[p.~75]{Dragalin1979,Dragalin1988}, \textit{every} locale can be represented as $\mathsf{Down}(S,\sqsubseteq)_j$ for some nuclear frame, in contrast to Proposition \ref{SpecialLocales}. We include a proof, presented as in \cite{BH2016}, to show its connection with the analogous Theorem \ref{FirstThm}.\ref{FirstThm2} in the Boolean case.

\begin{theorem}[\cite{Dragalin1979}]\label{DragThm} \textnormal{A lattice $L$ is a locale iff $L$ is isomorphic to $\mathsf{Down}(S,\sqsubseteq)_j$ for some dense nuclear frame $(S,\sqsubseteq,j)$.}
\end{theorem}
\begin{proof} The right-to-left direction is given by Theorem \ref{FixpointThm}. For the left-to-right, as in Theorem \ref{FirstThm}.\ref{FirstThm2}, let $L_+$ be $L\setminus\{0\}$ and $\leq_+$ the restriction to $L_+$ of the lattice order $\leq$ of $L$. Define a unary function $j$ on $\mathsf{Down}(L_+,\leq_+)$ by 
\begin{equation} jX=\mathord{\downarrow}\bigvee X, \label{jeq}\end{equation}
where $\bigvee X$ is the join of $X$ in $L$, and $\mathord{\downarrow}$ indicates the downset in $(L_+,\leq_+)$. It is easy to see that $j$ is inflationary, idempotent, and that $\mathord{\downarrow} \bigvee (X\cap Y)\subseteq (\mathord{\downarrow}\bigvee X)\cap (\mathord{\downarrow}\bigvee Y)$ for $X,Y\in \mathsf{Down}(S,\sqsubseteq)$. To see that  $\mathord{\downarrow} \bigvee (X\cap Y)\supseteq (\mathord{\downarrow}\bigvee X)\cap (\mathord{\downarrow}\bigvee Y)$, suppose that $a\in L_+$ is in the right-hand side, so $a\leq \bigvee X$ and $a\leq \bigvee Y$, whence $a\leq (\bigvee X)\wedge (\bigvee Y)$. By the join-infinite distributive law for locales,
\[(\bigvee X)\wedge (\bigvee Y) = \bigvee \{x\wedge y\mid x\in X, y\in Y\},\]
so $a\leq \bigvee \{x\wedge y\mid x\in X, y\in Y\}$. Since $X$ and $Y$ are downsets, we have $\{x\wedge y\mid x\in X, y\in Y\}\subseteq (X\cap Y)\cup\{0\}$, so $\bigvee \{x\wedge y\mid x\in X, y\in Y\}\leq \bigvee ((X\cap Y) \cup\{0\}) = \bigvee (X\cap Y)$. Thus, $a\leq \bigvee (X\cap Y)$ and hence $a\in \mathord{\downarrow} \bigvee (X\cap Y)$. Therefore, $j$ is a nucleus. To see that $j$ is dense, observe that $j\varnothing = \mathord{\downarrow}\bigvee \varnothing=\mathord{\downarrow} 0=\varnothing$ since $0\not\in S$.

Finally, we must check that $L$ is isomorphic to $\mathsf{Down}(L_+,\leq_+)_j$. Observe that the elements of $\mathsf{Down}(L_+,\leq_+)_j$ are exactly the principal downsets in $(L_+,\leq_+)$ plus $\varnothing$. Thus, the map sending each $x$ to $\mathord{\downarrow} x$ is the desired isomorphism.\end{proof}

The next step is to realize $j$ in some more concrete way. For example, \cite{BH2016} considers \textit{Dragalin frames}, originating in \cite{Dragalin1979,Dragalin1988}, which are triples $(S,\sqsubseteq, D)$ where $D$ is a certain function from which a nucleus $j_D$ of $\mathsf{Down}(S,\sqsubseteq)$ can be defined. It is then shown in \cite{BH2016} that for every nuclear frame $(S,\sqsubseteq,j)$, there is a Dragalin frame $(S,\sqsubseteq,D)$ (based on the same poset) such that $j=j_D$. Thus, by Theorem \ref{DragThm}, Dragalin frames can be used to represent all locales.  

Another approach, originating in \cite{Fairtlough1997}, works with triples $(S,\sqsubseteq, \leqslant)$ where $\leqslant$ is another partial order (or preorder) and an operation $j_{\Box_\sqsubseteq\Diamond_\leqslant}$ is defined by
\[j_{\Box_\sqsubseteq\Diamond_\leqslant}Z=\{x\in S\mid \forall x'\sqsubseteq x\,\exists x''\leqslant x': x''\in Z\}.\]
A sufficient (but not necessary) condition for $j_{\Box_\sqsubseteq\Diamond_\leqslant}$ to be a nucleus on $\mathsf{Down}(S,\sqsubseteq)$ is that $\leqslant$ is a subrelation of $\sqsubseteq$. Such frames are called FM-frames after Fairtlough and Mendler \cite{Fairtlough1997}. For a short proof of the following, see \cite[Thm.~4.33]{B&H2019}.

\begin{theorem}[\cite{BH2016,Massas2016}] \textnormal{For every locale $L$, there is an FM-frame $(S,\sqsubseteq, \leqslant)$ such that $L$ is isomorphic to $\mathsf{Down}(S,\sqsubseteq)_{j_{\Box_\sqsubseteq\Diamond_\leqslant}}$.}
\end{theorem}
\noindent This representation of locales is a special case of a representation of arbitrary complete lattices using doubly ordered structures \cite{Allwein2001}, which is compared to other representations of complete lattices in \cite{Holliday2021}.

Let us now connect these algebraic issues back to logic. A \textit{superintuitionistic logic} (si-logic) is a set of formulas in the language of propositional logic containing all theorems of the intuitionistic propositional calculus and closed under modus ponens and uniform substitution (see \cite{Chagrov1997}). Every class of Heyting algebras determines an si-logic in the obvious way, and every si-logic is the logic of a Heyting algebra, namely the Lindenbaum-Tarski algebra of the logic. Well-known si-logics are determined by more concrete Heyting algebras, such as those arising from posets (Kripke frames) or topological spaces as in Proposition \ref{SpecialLocales}. However, Shehtman \cite{Shehtman1977,Shehtman1980,Shehtman2005} has shown the following.

\begin{theorem}[\cite{Shehtman1977,Shehtman1980,Shehtman2005}] \textnormal{There is an si-logic that is not the logic of any class of Kripke frames.}
\end{theorem}
\noindent In fact, there are continuum-many such si-logics \cite{Litak2018}.

A famous open problem, posed by Kuznetsov \cite{Kuznetsov1975} in 1975, is whether there are not only Kripke-incomplete but even topologically-incomplete si-logics.

\begin{question}[Kuznetsov's Problem for spaces]\label{K1} Is every si-logic the logic of a class of topological spaces?
\end{question}

We can also consider the analogous question for locales.

\begin{question}[Kuznetsov's Problem for locales]\label{K2} Is every si-logic the logic of a class of  locales?
\end{question}

Given that FM-frames can represent arbitrary locales, one may hope that they or related structures could be used as tools to make progress on Questions \ref{K1}-\ref{K2}. Indeed, Massas \cite{Massas2020} has developed a theory of intuitionistic possibility frames $(S,\sqsubseteq, \leqslant)$ leading to a new kind of si-incompleteness result. We recall that a \textit{complete bi-Heyting algebra} is a locale $L$ whose order dual $L^\partial$ is also a locale (for recent applications of these algebras in topos quantum theory, see \cite{Doering2012,Flori2018}). The following result was also obtained independently in \cite{Bezhanishvili2021} using Esakia duality.

\begin{theorem}[\cite{Bezhanishvili2021,Massas2020}] \textnormal{There is an si-logic that is not the logic of any class of complete bi-Heyting algebras.}\end{theorem}
\noindent Indeed, there are continuum-many such si-logics \cite{Bezhanishvili2021}. 

To develop possibility semantics not only for complete Heyting algebras but for arbitrary Heyting algebras, we need an appropriate analogue of ``choice-free Stone duality'' for Heyting algebras. We leave this as an open problem.

\begin{question} What is the most natural analogue for Heyting algebras of UV-spaces for Boolean algebras?
\end{question}

\subsection{Inquisitive case}\label{InquisitiveSection}

So far we have used possibility frames to interpret standard languages, e.g., the language of propositional logic, the language of  modal logic, the language of first-order logic, etc. However, as suggested in \cite{Benthem2016b}, once we make the switch from structures based on a set $W$ to structures based on a poset $(S,\sqsubseteq)$, we may naturally introduce new logical operators, beyond those in the traditional repertoire, that exploit the partial order structure. An example of exactly this move comes from the program of \textit{inquisitive logic}  \cite{Ciardelli2009,Ciardelli2011,Ciardelli2013,Ciardelli2016,Ciardelli2018}, which adds a new question-forming disjunction $\Inqvee$ to the language of propositional (or modal or first-order) logic. Intuitively, the formula $p \inqvee q$ represents the question of \textit{whether $p$ or $q$}, in contrast to the formula $p\vee q$, which represents the declarative sentence \textit{$p$ or $q$}. 

Classical inquisitive logic is based on the language $\mathcal{L}_{\Inqvee}$ defined by 
\[\varphi::= p\mid \neg\varphi\mid (\varphi\wedge\varphi)\mid (\varphi\inqvee\varphi).\]
Let $\mathcal{L}$ be the fragment without $\Inqvee$. Then classical inquisitive semantics uses posets $(S,\sqsubseteq)$ (in fact, special posets, as in Remark \ref{SpecialPosets} below), takes the semantic values of propositional variables to be elements of $\mathcal{RO}(S,\sqsubseteq)$, and interprets $\mathcal{L}$ exactly as in possibility semantics. But $\inqvee$ is interpreted as follows:
\begin{itemize}
\item $\mathcal{M},x\Vdash \varphi\inqvee\psi$ iff $\mathcal{M},x\Vdash\varphi$ or $\mathcal{M},x\Vdash\psi$.
\end{itemize}
Compare this with the clause for $\vee$, which is defined in terms of $\neg$ and $\wedge$ by the classical definition $(\varphi\vee\psi):=\neg(\neg\varphi\vee\neg\psi)$:
\begin{itemize}
\item $\mathcal{M},x\Vdash\varphi\vee\psi$ iff $\forall x'\sqsubseteq x'$ $\exists x''\sqsubseteq x'$: $\mathcal{M},x''\Vdash\varphi$ or $\mathcal{M},x''\Vdash\psi$.
\end{itemize}
While $\mathcal{M},x\Vdash\varphi\vee\psi$ represents that $x$, now conceived as an \textit{information state}, settles that the disjunctive statement is true, $\mathcal{M},x\Vdash \varphi\inqvee\psi$ represents that the information state $x$ \textit{answers the question of whether $\varphi$ or $\psi$}. This idea has numerous applications to natural language semantics (see the references above). For further discussion of inquisitive disjunction and related issues in connection with possibility semantics, see \cite{Humberstone2019}.

\begin{example}\label{InqEx} Consider a possibility model with the following poset and valuation (so $p$ is forced only at the bottom left node, etc.):
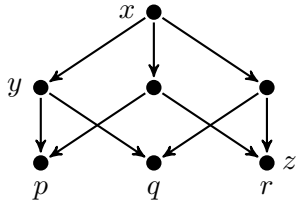
\begin{center}
\begin{tikzpicture}[->,>=stealth',shorten >=1pt,shorten <=1pt, auto,node
distance=2cm,thick,every loop/.style={<-,shorten <=1pt}] \tikzstyle{every state}=[fill=gray!20,draw=none,text=black] 

\node[circle,draw=black!100,fill=black!100, label=left:$x$,inner sep=0pt,minimum size=.175cm] (bot) at (1.5,0) {{}}; 
\node[circle,draw=black!100,fill=black!100, label=left:$y$,inner sep=0pt,minimum size=.175cm] (A) at (0,-1) {{}}; 
\node[circle,draw=black!100,fill=black!100, label=right:,inner sep=0pt,minimum size=.175cm] (B) at (1.5,-1) {{}}; 
\node[circle,draw=black!100,fill=black!100, label=right:,inner sep=0pt,minimum size=.175cm] (C) at (3,-1) {{}}; 
\node[circle,draw=black!100,fill=black!100, label=below:$p$,inner sep=0pt,minimum size=.175cm] (A') at (0,-2) {{}}; 
\node[circle,draw=black!100,fill=black!100, label=below:$q$,inner sep=0pt,minimum size=.175cm] (B') at (1.5,-2) {{}}; 
\node[circle,draw=black!100,fill=black!100, label=below:$r$,inner sep=0pt,minimum size=.175cm] (C') at (3,-2) {{}}; 
\node at (3.3,-2) {{$z$}};

\path (A) edge[->] node {{}} (A'); 
\path (A) edge[->] node {{}} (B'); 
\path (B) edge[->] node {{}} (A'); 
\path (B) edge[->] node {{}} (C'); 
\path (C) edge[->] node {{}} (C'); 
\path (C) edge[->] node {{}} (B'); 
\path (A) edge[<-] node {{}} (bot); 
\path (B) edge[<-] node {{}} (bot); 
\path (C) edge[<-] node {{}} (bot);

\end{tikzpicture}
\end{center}
At the root node $x$, we have $\mathcal{M},x\Vdash (p\vee q)\vee r$ but $\mathcal{M},x\nVdash (p\vee q)\inqvee r$. For the latter claim, observe that $\mathcal{M},x\nVdash r$ by the valuation and $\mathcal{M},x\nVdash (p\vee q)$ because  $z\sqsubseteq x$ and $\mathcal{M},z\nVdash p\vee q$. By contrast, $\mathcal{M},y\Vdash  (p\vee q)\inqvee r$ because $\mathcal{M},y\Vdash p\vee q$.\end{example}

\begin{remark}\label{SpecialPosets} Inquisitive logicians typically assume that their poset $(S,\sqsubseteq)$ is of a very special form: it is the poset of all nonempty subset of a set $W$, ordered by inclusion (in the finite case these posets---or their duals---appear as \textit{Medvedev frames} in semantics for the si-logic known as Medvedev's logic \cite{Medvedev1966}, as in \cite{Holliday2017}). The poset used in Example \ref{InqEx} is isomorphic to such a poset. One can check that the description of classical inquisitive semantics above is equivalent to the usual one in the literature, assuming $(S,\sqsubseteq)$ is of the special form.\end{remark}

For the special posets in Remark \ref{SpecialPosets}, an axiomatization of the validities of $\mathcal{L}_{\Inqvee}$ is available, yielding classical inquisitive propositional logic (see, e.g., \cite{Ciardelli2009}). 

Classical inquisitive logic may be understood algebraically in the terms of \S~\ref{IntuitionisticSection} as follows: working with the Heyting algebra $\mathsf{Down}(S,\sqsubseteq)$ of downsets of a poset, inquisitive logicians (i) interpret propositional variables as elements of the fixpoint algebra $\mathsf{Down}(S,\sqsubseteq)_{j_{\neg\neg}}$ of the double negation nucleus $j_{\neg\neg}$, (ii) interpret $\vee$ (when defined in terms of $\neg$ and $\wedge)$ as $j_{\neg\neg}$ applied to the join in $\mathsf{Down}(S,\sqsubseteq)$, as in possibility semantics, and (iii) interpret $\Inqvee$ as the ordinary join in $\mathsf{Down}(S,\sqsubseteq)$. The same perspective is adopted in \cite{Bezhanishvili2019}, where a new topological semantics for classical inquisitive logic is given using the UV-spaces discussed in \S~\ref{TopFrameSection2}. 

Once we have this perspective, there is a natural way to study inquisitive logic on an \textit{intuitionistic} base \cite{Holliday2020}. That is, we now fix the language $\mathcal{L}_{\mbox{$\vee$},\Inqvee}$ as
\[\varphi::= \bot \mid p\mid (\varphi\wedge\varphi)\mid (\varphi\vee\varphi)\mid (\varphi\to\varphi)\mid (\varphi\inqvee\varphi),\]
and we interpret $\mathcal{L}_{\mbox{$\vee$},\Inqvee}$ in nuclear frames $(S,\sqsubseteq, j)$ where $j$ need not be the nucleus of double negation. As before, we (i) interpret propositional variables as elements of $\mathsf{Down}(S,\sqsubseteq)_j$, (ii) interpret $\vee$ as $j$ applied to the join in $\mathsf{Down}(S,\sqsubseteq)$, and (iii) interpret $\Inqvee$ as the ordinary join in $\mathsf{Down}(S,\sqsubseteq)$. In \cite{Holliday2020} this approach is instantiated using the Beth nucleus defined in Example \ref{NucEx}, yielding a Beth semantics (cf.~\cite{Beth1956,B&H2019}) for inquisitive intuitionistic logic, which is completely axiomatized. 

\subsection{Further language extensions}

The distinction between the declarative disjunction $\vee$ and the inquisitive disjunction $\Inqvee$ in \S~\ref{InquisitiveSection} is an example of the kind of distinction that can be made in a semantics based on posets $(S,\sqsubseteq)$ instead of sets $W$.\footnote{While the poset typically used in inquisitive semantics is the poset $(\wp(W)\setminus\{\varnothing\},\subseteq)$ for some set $W$ (see Remark \ref{SpecialPosets}), we still consider inquisitive semantics a semantics based on posets, since it uses the structure of the poset $(\wp(W)\setminus\{\varnothing\},\subseteq)$, not just the structure of the set $W$.} There is a large literature on using posets and related structures to interpret operators beyond the standard connectives, modalities, and quantifiers (see, e.g., \cite{Benthem1983a} and \cite{Halpern1991} in a temporal context and \cite{Benthem2016} and \cite[\S~12]{Benthem2019} in an informational context). Less has been written about adding semantic clauses for new operators alongside the specifically possibility-semantic clauses for the standard connectives, modalities, and quantifiers. Thus, we finish this survey with a broad open problem for the future development of possibility semantics.

\begin{question} What other operators can we add to the languages already interpreted by possibility semantics, exploiting the extra structure in a poset of possibilities beyond that of a flat set of worlds? For example: Epistemic modals \cite{HM2021}? Indicative conditionals \cite{HM2021}? Counterfactuals \cite{Santorio2021}? Determinacy operators~\cite{Cariani2021}? Awareness operators (cf.~\cite{Halpern2009})?\end{question}

\section{Conclusion}

Classical possibility semantics offers, in the case of discrete duality, freedom from the restriction to atomic algebras, and in the case of topological duality, freedom from the reliance on nonconstructive choice principles. In both cases, possibility semantics shares some of the spirit of other research programs. 

In the first case, just as forcing and Boolean-valued semantics for set theory made possible proofs of independence results in set theory that could not be achieved by previous methods, possibility semantics for modal logic makes possible proofs of independence results that cannot be proved using atomic algebras (Propositions \ref{AtomNotThm}, \ref{Snotvalid}-\ref{ESTvalid}, and \ref{OmProp}-\ref{NoOmniscience}, Theorems \ref{ProvThm}, \ref{ProvThmC}, and \ref{ProvThmB}, and Proposition \ref{WorldFact}).  Indeed, we have not only independence results but completeness results for logics that are incomplete with respect to semantics based on atomic algebras (for axiomatized logics, recall Theorems \ref{S5PiComp}, \ref{S5PiExtComp}, \ref{DingThm}, \ref{DHthm}, \ref{DHthm2}, and for semantically defined logics, recall Theorems \ref{CVnotCAV} and Corollaries \ref{SlogicWorldIncomplete}, \ref{UniversalCor}, and \ref{FOincomp}). And as we have seen, possibility semantics for first-order logic essentially \textit{is} forcing and Boolean-valued semantics writ large---applied not only to the language of set theory but to arbitrary first-order languages. Possibility semantics for first-order modal logic remains to be systematically explored, as do further applications of possibility semantics in provability logic (recall Remark \ref{QArith} and \S\S~\ref{FullRelFrameSection}-\ref{PrincRelFrame}).

 In the second case mentioned above, just as pointfree topology seeks to separate classical results on spaces into a constructive localic argument plus the application of a nonconstructive choice principle, ``choice-free Stone duality'' and ``neighborhood possibility systems'' may effect a similar separation but now using a spatial or poset-based---but still constructive---argument. 
 
 Finally, there lies the frontier beyond classical logics---intuitionistic logics and logics for extended languages that exploit the extra structure of partial states, not provided by worlds. It remains to be seen how in these areas possibility semantics may open up yet more semantic possibilities.

\section*{Acknowledgment}

For helpful feedback on this chapter, I thank Johan van Benthem, Guram Bezhanishvili, Yifeng Ding, Mikayla Kelley, Guillaume Massas, James Walsh, and the anonymous referee for Landscapes in Logic. For other helpful discussions of possibility semantics, I thank Nick Bezhanishvili, Matthew Harrison-Trainor, Lloyd Humberstone, Tadeusz Litak, and Kentar\^{o} Yamamoto.

\appendix
\section{Appendix}\label{NecSucc}

\counterwithin{theorem}{section}

In this appendix, extending \S~\ref{AccessSection}, we identify the interaction conditions on $R_i$ and  $\sqsubseteq$ that are necessary and sufficient for $\mathcal{RO}(S,\sqsubseteq)$ to be closed under $\Box_i$, as required for a \textit{full} relational possibility frame as in Definition \ref{RellPossFrame}. For $x,y\in S$, let
\begin{eqnarray*}
&&x \incomp y\mbox{ iff }\mathord{\downarrow}x\cap\mathord{\downarrow}y=\varnothing; \\
&&x\comp y\mbox{ iff }\mathord{\downarrow}x\cap\mathord{\downarrow}y\neq\varnothing.
\end{eqnarray*}

\begin{proposition}[\cite{Holliday2018}] \textnormal{For any poset $( S,\sqsubseteq )$ and binary relation $R$ on $S$, the following are equivalent: 
\begin{enumerate}
\item\label{ROtoRO1} $\mathcal{RO}(S,\sqsubseteq)$ is closed under the operation $\Box_i$ defined for $Z\in \mathcal{RO}(S,\sqsubseteq)$ by $\Box_i Z=\{x\in S\mid R_i(x)\subseteq Z\}$;
\item\label{ROtoRO2}  $R$ and $\sqsubseteq$ satisfy the following conditions:
\begin{enumerate}
\item \Rrule{}: if for all $y\in R(x)$, $y\incomp z$, and ${x}'\sqsubseteq{x}$, then for all $y'\in R(x')$, $y'\incomp z$;
\item \Rwinweak{}: if ${x}R{y}$, then $\forall {y'}\sqsubseteq {y}$ $\exists{x'}\sqsubseteq {x}$ $\forall {x''}\sqsubseteq {x'}$ $\exists {y''}\in R(x'')$: $y''\comp y'$.
\end{enumerate}
\end{enumerate}}
\end{proposition}

The \Rrule{} condition says if $x$ has ``ruled out'' the possibility $z$, then $z$ remains ruled out by every refinement $x'$ of $x$. Contrapositively, if $z$ is compatible with some possibility accessible from $x'$, then $z$ must be compatible with some possibility accessible from $x$, as shown in Figure \ref{RruleFig}:
\[\mbox{if ${x}'\sqsubseteq{x}$ and ${x}' R_i{y}'\comp{z}$, then $\exists {y}$: ${x} R_i{y}\comp{z}$}.\]

 \begin{figure}[h]
\begin{center}
\begin{tikzpicture}[->,>=stealth',shorten >=1pt,shorten <=1pt, auto,node
distance=2cm,thick,every loop/.style={<-,shorten <=1pt}]
\tikzstyle{every state}=[fill=gray!20,draw=none,text=black]
\node (x-up) at (0,0) {{$x$}};
\node (x) at (0,-1.5) {{$x'$}};
\node (y) at (2,-1.5) {{$y'$}};
\node (z) at (4,-1.5) {{$z$}};
\node (comp1a) at (3.15, -2.25) {{}};
\node (comp1b) at (2.85, -2.25) {{}};

\node at (5,-1.5) {{\textit{$\Rightarrow$}}};

\path (x) edge[dashed,->] node {{}} (y);
\path (x-up) edge[->] node {{}} (x);

\path (y) edge[->] node {{}} (comp1a);
\path (z) edge[->] node {{}} (comp1b);

\node at (7,.3) {{$\exists$}};
\node (x-up') at (6,0) {{$x$}};
\node (x') at (6,-1.5) {{$x'$}};
\node (y') at (8,-1.5) {{$y'$}};
\node (y-up') at (8,0) {{$y$}};

\node (z') at (10,0) {{$z$}};
\node (comp1a') at (9.15, -.75) {{}};
\node (comp1b') at (8.85, -.75) {{}};

\path (x') edge[dashed,->] node {{}} (y');
\path (x-up') edge[->] node {{}} (x');

\path (x-up') edge[dashed,->] node {{}} (y-up');

\path (y-up') edge[->] node {{}} (comp1a');
\path (z') edge[->] node {{}} (comp1b');

\end{tikzpicture}
\end{center}
\caption{the \Rrule{} condition.}\label{RruleFig}
\end{figure}

The \Rwinweak{} condition has a natural game-theoretic interpretation. 

\begin{definition} Given a poset $(S,\sqsubseteq)$, $x,y\in S$, and binary relation $R$ on $S$,  the \textit{accessibility game} $\mathrm{G}(S,\sqsubseteq, R,x,y)$ for players \textbf{A} and \textbf{E} has the following rounds, depicted in Figure \ref{AccGame}:
\begin{enumerate} 
\item \textbf{A} chooses a $y'\sqsubseteq y$; 
\item \textbf{E} chooses an $x'\sqsubseteq x$;
\item \textbf{A} chooses an $x''\sqsubseteq x'$;
\item[] if $R(x'')=\emptyset$, then \textbf{A} wins, otherwise play continues;
\item \textbf{E} chooses a $y''\in R(x'')$;
\item[] if $y''\comp y'$, then \textbf{E} wins, otherwise \textbf{A} wins.
\end{enumerate}
\end{definition}
One can think of \textbf{A} and \textbf{E} as arguing about whether $y$ is  \textit{accessible} to $x$: if it is, then for any way $y'$ of further specifying $y$, there should be some way $x'$ of further specifying $x$ that ``locks in'' access to possibilities compatible with $y'$, i.e., such that all refinements $x''$ of $x'$ have access to some possibility $y''$ compatible with $y'$. If refinements of $x$ cannot keep up with refinements of $y$ in this way, then $y$ is not accessible to $x$. Thus, player \textbf{A} is trying to show that $y$ is not accessible to $x$, while player \textbf{E} is trying to block \textbf{A}'s argument.
\begin{figure}[h]
 \begin{center}
\begin{tikzpicture}[->,>=stealth',shorten >=1pt,shorten <=1pt, auto,node
distance=2cm,thick,every loop/.style={<-,shorten <=1pt}]
\tikzstyle{every state}=[fill=gray!20,draw=none,text=black]

\node (x) at (0,0) {{$x$}};
\node (y) at (3,0) {{$y$}};

\node (x') at (0,-1.5) {{$x'$}};
\node (x'') at (0,-3) {{$x''$}};
\node (y') at (3,-1.5) {{$y'$}};
\node (y'') at (3,-3) {{$y''$}};

\node (5) at (3.1,-2.3) {{\scalebox{1.4}{$\comp$} ?}};

\node (4) at (1.4,-3.4) {{4.\,\textbf{E} chooses}};

\path (x') edge[<-] node {{ 2.\,\textbf{E} chooses\;}} (x);
\path (y) edge[->] node {{ 1.\,\textbf{A} chooses}} (y');
\path (y'') edge[dashed,<-] node {{}} (x'');
\path (x'') edge[<-] node {{ 3.\,\textbf{A} chooses\;}} (x');

\end{tikzpicture}
\end{center}
\caption{the accessibility game $\mathrm{G}$ -- if $y''\comp y'$, \textbf{E} wins, otherwise \textbf{A} wins.}\label{AccGame}
\end{figure}

\begin{lemma}[\cite{Holliday2018}]\textnormal{For any poset $( S,\sqsubseteq )$ and binary relation $R$ on $S$, the following are equivalent: 
\begin{enumerate}
\item $R$ and $\sqsubseteq$ satisfy \Rwinweak{};
\item for any $x,y\in S$, if $xRy$, then \textbf{E} has a winning strategy in $\mathrm{G}(S,\sqsubseteq, R,x,y)$.
\end{enumerate}}
\end{lemma}

A natural strengthening of  \Rwinweak{} entails both  \Rwinweak{} and \Rrule{} and can be assumed without loss of generality. There are two obvious ways to consider strengthening \Rwinweak{}: change the final condition from $y''\comp y'$ to $y''\sqsubseteq y'$, and change the ``if\dots then'' to and ``if and only if''. Making both of these modifications, we obtain:
\begin{itemize}
\item \RWin{} -- $xR_iy$ iff $\forall y'\sqsubseteq y$ $\exists x'\sqsubseteq x$ $\forall x''\sqsubseteq x'$ $\exists y''\in R(x'')$: $y''\sqsubseteq y'$.
\end{itemize}
Changing the winning condition of the game $\mathrm{G}(S,\sqsubseteq, R,x,y)$ to $y''\sqsubseteq y''$ gives us the game $\underline{\mathrm{G}}(S,\sqsubseteq, R,x,y)$, as in Figure \ref{AccGame2}, for which the following holds.

\begin{lemma}[\cite{Holliday2018}]\textnormal{For any poset $( S,\sqsubseteq )$ and binary relation $R$ on $S$, the following are equivalent: 
\begin{enumerate}
\item $R$ and $\sqsubseteq$ satisfy \RWin{};
\item for any $x,y\in S$, $xRy$ iff \textbf{E} has a winning strategy in $\underline{\mathrm{G}}(S,\sqsubseteq, R,x,y)$.
\end{enumerate}}
\end{lemma}

\begin{figure}[h]
 \begin{center}
\begin{tikzpicture}[->,>=stealth',shorten >=1pt,shorten <=1pt, auto,node
distance=2cm,thick,every loop/.style={<-,shorten <=1pt}]
\tikzstyle{every state}=[fill=gray!20,draw=none,text=black]

\node (x) at (0,0) {{$x$}};
\node (y) at (3,0) {{$y$}};

\node (x') at (0,-1.5) {{$x'$}};
\node (x'') at (0,-3) {{$x''$}};
\node (y') at (3,-1.5) {{$y'$}};
\node (y'') at (3,-3) {{$y''$}};

\node (4) at (1.4,-3.4) {{4.\,\textbf{E} chooses}};

\path (x') edge[<-] node {{ 2.\,\textbf{E} chooses\;}} (x);

\path (y) edge[->] node {{ 1.\,\textbf{A} chooses}} (y');
\path (y') edge[->] node {{?}} (y'');
\path (y'') edge[dashed,<-] node {{}} (x'');
\path (x'') edge[<-] node {{ 3.\,\textbf{A} chooses\;}} (x');

\end{tikzpicture}
\end{center}
\caption{the accessibility game $\underline{\mathrm{G}}$ -- if $y''\sqsubseteq y'$, \textbf{E} wins, otherwise \textbf{A} wins.}\label{AccGame2}
\end{figure}
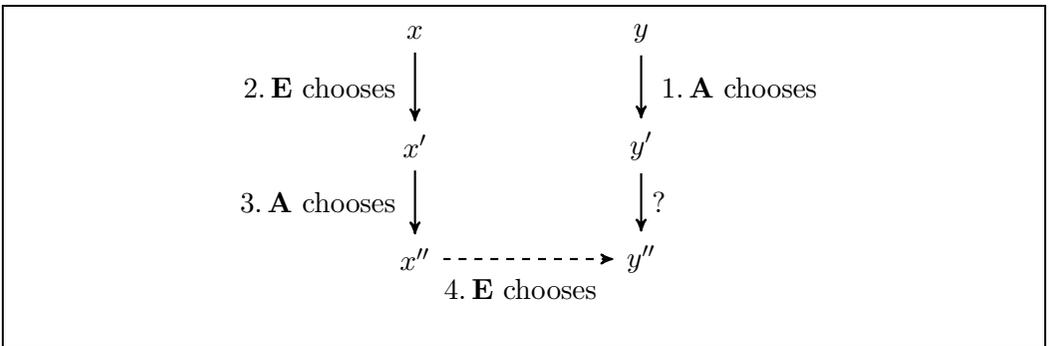

Finally, we can relate \RWin{} to the notion of a \textit{strong} possibility frame from Definition \ref{StrongDef}.

\begin{lemma}[\cite{Holliday2018}] \textnormal{\RWin{} is equivalent to the conjunction of \upR{}, \Rdown{}, \Rref{}, and \Rdense{}. Hence a possibility frame is strong iff it satisfies \RWin{}.}
\end{lemma}




{\footnotesize

\bibliographystyle{acm}
\bibliography{possibilities}
}

\def\DITTO{---}

\end{document}